\pdfoutput=1
\RequirePackage[l2tabu, orthodox]{nag}

\documentclass[reqno]{amsart}
\usepackage[letterpaper, portrait, margin=0.9in, headheight=1in]{geometry}
\usepackage{lmodern}
\usepackage[T1]{fontenc}
\usepackage[utf8]{inputenc}
\usepackage[english]{babel}
\usepackage{microtype} 
\usepackage{comment} 
\usepackage{float} 
\usepackage{multicol} 
\usepackage{animate}

\usepackage{amsmath,amssymb,amsthm,mathrsfs, latexsym,mathtools,mathdots,booktabs,tabularx,enumerate,bm,url,nicefrac,leftidx}
\usepackage[centertableaux]{ytableau}
\usepackage{cases} 
\usepackage{caption,subcaption} 
\usepackage{xparse,tikz,xcolor}
\usepackage{listofitems} 

\usetikzlibrary{arrows,calc,positioning,shapes.geometric,patterns}
\usepackage[pdftex]{hyperref}
\usepackage[capitalize,noabbrev]{cleveref}

\numberwithin{equation}{section}

\hypersetup{
    unicode=false,          
    pdftoolbar=true,        
   pdfmenubar=true,        
   pdffitwindow=false,     
   pdfstartview={FitH},    
   pdftitle={},    
   pdfauthor={Author},     
    pdfcreator={Creator},   
    pdfproducer={Producer}, 
   pdfkeywords={keyword1} {key2} {key3}, 
   pdfnewwindow=true,      
    colorlinks=true,       
    linkcolor=blue,          
    citecolor=black,        
    filecolor=magenta,      
    urlcolor=black    
    }

\usepackage{todonotes}
\presetkeys%
    {todonotes}%
    {inline,backgroundcolor=yellow}{}

\newtheorem{theorem}{Theorem}[section]
\newtheorem{proposition}[theorem]{Proposition}
\newtheorem{lemma}[theorem]{Lemma}
\newtheorem{corollary}[theorem]{Corollary}
\newtheorem{conjecture}[theorem]{Conjecture}

\theoremstyle{definition}
\newtheorem{definition}[theorem]{Definition}
\newtheorem{example}[theorem]{Example}
\newtheorem{remark}[theorem]{Remark}
\newtheorem{question}[theorem]{Question}

\definecolor{lightblue}{rgb}{0.8,0.8,1.0}
\definecolor{lightgreen}{rgb}{0.8,1.0,0.8}
\definecolor{pBlue}{RGB}{86,139,190}
\definecolor{pCyan}{RGB}{149,186,201}
\definecolor{pSand}{RGB}{184,166,121}
\definecolor{pAlgae}{RGB}{87,115,135}
\definecolor{pSkin}{RGB}{236,216,167}
\definecolor{pGray}{RGB}{156,175,156}
\definecolor{pPink}{RGB}{215,114,127}
\definecolor{pOrange}{RGB}{211,153,80}

\newcommand{\defin}[1]{%
\relax\ifmmode%
\textcolor{blue}{#1}%
\else \textcolor{blue}{\emph{#1}}%
\fi%
}

\newcommand{\oeis}[1]{\href{http://oeis.org/#1}{\textcolor{blue}{\texttt{#1}}}}

\newcommand{\twin}{\textnormal{twin}}

\newcommand{\setN}{\mathbb{N}}

\newcommand{\avec}{\mathbf{a}}
\newcommand{\bvec}{\mathbf{b}}

\newcommand{\wvec}{\mathbf{w}}

\newcommand{\abelian}[1]{\mathcal{A}(#1)}

\newcommand{\DP}{\mathcal{DP}}
\newcommand{\motzkin}{\mathcal{MP}} 

\newcommand{\interl}{\ll}

\DeclareMathOperator{\asc}{asc}
\DeclareMathOperator{\des}{des}

\DeclareMathOperator{\peak}{peak}
\DeclareMathOperator{\Peak}{Peak}

\DeclareMathOperator{\plateau}{plateau}
\DeclareMathOperator{\decor}{decor}
\DeclareMathOperator{\height}{height}

\DeclareMathOperator{\pnv}{\mathbf{pnv}} 
\DeclareMathOperator{\pnest}{pnest} 
\DeclareMathOperator{\Pnest}{Pnest}
\DeclareMathOperator{\rank}{rank}

\tikzset{every picture/.append
  style={scale=1,
	baseline=(current bounding box.center),
	x=1em,
	y=1em,
	thinLine/.style={line width=0.7pt},
	thickLine/.style={line width=1.4pt,line join=round},
	snaking/.style={decorate,decoration={zigzag,segment length=1.0mm,amplitude=0.15mm}},
	entries/.style={xshift=-0.5em,yshift=-0.5em,font=\small},
	bgEntry/.style={xshift=-0.5em,yshift=-0.5em,font=\small,
		regular polygon,regular polygon sides=4,fill,inner sep=0pt,minimum size=1em
	}
	}
}

\ExplSyntaxOn
\clist_new:N \l__dyck_seq_clist
\tl_new:N    \l__dyck_coords_tl
\int_new:N   \l__dyck_n_int
\int_new:N   \l__dyck_x_int
\int_new:N   \l__dyck_y_int
\int_new:N   \l__dyck_ai_int
\int_new:N   \l__dyck_cap_int
\int_new:N   \l__dyck_start_int
\int_new:N   \l__dyck_col_int
\int_new:N   \l__dyck_target_int
\NewDocumentCommand{\dyckDiagram}{ O{0.35cm} m }{%
  \group_begin:
  \clist_set:Nn \l__dyck_seq_clist { #2 }
  \int_set:Nn  \l__dyck_n_int { \clist_count:N \l__dyck_seq_clist }
  \tl_set:Nn   \l__dyck_coords_tl { (0,0) } 

  \begin{tikzpicture}[x=#1,y=#1,baseline=(current  \c_space_tl bounding  \c_space_tl box.south)]
    \int_step_inline:nn { \l__dyck_n_int }{
      \int_set:Nn \l__dyck_ai_int  { \clist_item:Nn \l__dyck_seq_clist {##1} }
      \int_compare:nNnTF { \l__dyck_ai_int } > { ##1 - 1 }
        { \int_set:Nn \l__dyck_cap_int { ##1 - 1 } }
        { \int_set_eq:NN \l__dyck_cap_int \l__dyck_ai_int }
      \int_set:Nn \l__dyck_start_int { ##1 - \l__dyck_cap_int }
      \int_step_inline:nn { \l__dyck_cap_int }{
        \int_set:Nn \l__dyck_col_int { \l__dyck_start_int + ####1 - 1 }
        \int_set:Nn \l__dyck_x_int { \l__dyck_col_int - 1 } 
        \int_set:Nn \l__dyck_y_int { ##1 - 1 }             
        \path[fill=blue!20,draw=none]
          ( \int_use:N \l__dyck_x_int , \int_use:N \l__dyck_y_int )
          rectangle
          ( \int_use:N \l__dyck_col_int , ##1 );
      }
    }

    \draw[step=1,thin,gray!70] (0,0) grid (\int_use:N\l__dyck_n_int,\int_use:N\l__dyck_n_int);
    \int_step_inline:nn { \l__dyck_n_int }{
      \node[font=\footnotesize] at (##1-0.5,##1-0.5) {##1};
    }

    \int_zero:N \l__dyck_x_int
    \int_zero:N \l__dyck_y_int
    \int_step_inline:nn { \l__dyck_n_int }{
      \int_set:Nn \l__dyck_ai_int  { \clist_item:Nn \l__dyck_seq_clist {##1} }
      \int_compare:nNnTF { \l__dyck_ai_int } > { ##1 - 1 }
        { \int_set:Nn \l__dyck_cap_int { ##1 - 1 } }
        { \int_set_eq:NN \l__dyck_cap_int \l__dyck_ai_int }
      \int_set:Nn \l__dyck_target_int { ##1 - 1 - \l__dyck_cap_int }
      \int_while_do:nn { \l__dyck_x_int < \l__dyck_target_int }{
        \int_incr:N \l__dyck_x_int
        \tl_put_right:Nx \l__dyck_coords_tl
          { ~ ( \int_use:N \l__dyck_x_int , \int_use:N \l__dyck_y_int ) }
      }
      \int_incr:N \l__dyck_y_int
      \tl_put_right:Nx \l__dyck_coords_tl
        { ~ ( \int_use:N \l__dyck_x_int , \int_use:N \l__dyck_y_int ) }
    }
    \int_while_do:nn { \l__dyck_x_int < \l__dyck_n_int }{
      \int_incr:N \l__dyck_x_int
      \tl_put_right:Nx \l__dyck_coords_tl
        { ~ ( \int_use:N \l__dyck_x_int , \int_use:N \l__dyck_y_int ) }
    }

    \draw[line  \c_space_tl width=0.6mm,black] plot coordinates { \l__dyck_coords_tl };
  \end{tikzpicture}
  \group_end:
}
\ExplSyntaxOff

\definecolor{areablue}{RGB}{205,205,245}

\definecolor{myblue}{RGB}{205,205,245}
\definecolor{myred}{RGB}{245,170,170}

\newcommand{\AbelianDyckDiagram}[4][0.9cm]{%
    \pgfmathtruncatemacro{\n}{#2}
    \pgfmathtruncatemacro{\k}{#3}
    \readlist\B{#4}

    \begin{tikzpicture}[x=#1,y=#1,line cap=round,line join=round]

        \foreach \c in {0,...,\numexpr\n-1\relax} {

            \ifnum\c<\k
                \pgfmathtruncatemacro{\ii}{\c+1}
                \pgfmathtruncatemacro{\H}{\k+\B[\ii]}
            \else
                \pgfmathtruncatemacro{\H}{\n}
            \fi

            \pgfmathtruncatemacro{\firstrow}{\c+2}
            \pgfmathtruncatemacro{\lastrow}{\H}

            \ifnum\firstrow>\lastrow
            \else
                \foreach \r in {\firstrow,...,\lastrow} {
                    \fill[myblue] (\c,\r-1) rectangle ++(1,1);
                }
            \fi
        }

        \foreach \i in {1,...,\k} {
            \pgfmathtruncatemacro{\bi}{\B[\i]}
            \pgfmathtruncatemacro{\col}{\i-1}
            \ifnum\bi>0
                \foreach \r in {\numexpr\k+1\relax,...,\numexpr\k+\bi\relax} {
                   \fill[white] (\col,\r-1) rectangle ++(1,1);
                    \fill[pattern=north west lines, pattern color=myred] (\col,\r-1) rectangle ++(1,1);
                }
            \fi
        }

        \draw[step=1,gray!60,thin] (0,0) grid (\n,\n);

        \foreach \i in {1,...,\n} {
            \node[font=\large] at (\i-0.5,\i-0.5) {\i};
        }

        \draw[black,line width=1.2pt] (0,0) -- (0,\k);

        \pgfmathtruncatemacro{\prevh}{\k}
        \foreach \i in {1,...,\k} {
            \pgfmathtruncatemacro{\bi}{\B[\i]}
            \pgfmathtruncatemacro{\xleft}{\i-1}
            \pgfmathtruncatemacro{\h}{\k+\bi}

            \draw[red,line width=1.5pt] (\xleft,\prevh) -- (\xleft,\h);

            \draw[red,line width=1.5pt] (\xleft,\h) -- (\i,\h);

            \xdef\prevh{\h}
        }

        \draw[red,line width=1.5pt] (\k,\prevh) -- (\k,\n);

        \draw[black,line width=1.2pt] (\k,\n) -- (\n,\n);

    \end{tikzpicture}%
}

\makeatletter
\renewcommand*\env@matrix[1][\arraystretch]{%
  \edef\arraystretch{#1}%
  \hskip -\arraycolsep
  \let\@ifnextchar\new@ifnextchar
  \array{*\c@MaxMatrixCols c}}
\makeatother

\title{Peaks and peak-nestings on unit interval graphs}

\author[P.~Alexandersson]{Per Alexandersson}
\address{Department of Mathematics, Stockholm University, SE-106 91 Stockholm, Sweden}
\email{per.w.alexandersson@gmail.com}

\author[L.~Saud]{Leonardo Saud Maia Leite}
\address{Department of Mathematics, KTH Royal Institute of Technology, SE-100 44 Stockholm,
Sweden}
\email{lsml@kth.se}


\begin{document}

\begin{abstract}
    Unit interval graphs admit a classical Catalan encoding by area sequences, or equivalently by Dyck paths. We introduce a new statistic on these graphs, called the \emph{peak-nesting}, defined as the largest number of peak-cliques containing a common vertex. Through the correspondence with area sequences, peak-nesting also defines a new statistic on Dyck paths. We construct a bijection between Dyck paths and bicolored Motzkin paths which simultaneously records peak-nesting and the number of peaks. This yields a refinement of Touchard's identity, coefficient formulas involving Dyck paths of given height, and new combinatorial interpretations for several sequences recorded in the OEIS. We study the peak and peak-nesting polynomials over all unit interval graphs and over the subclasses of connected, Abelian, $2$-nested, and reduced unit interval graphs. For these families we obtain closed formulas, recurrences, and, for the symmetric cases, nonnegative expansions in the gamma-basis with explicit combinatorial interpretations. We also investigate questions regarding the location of zeros and the log-concavity of the corresponding coefficient sequences: some of the polynomial families form generalized Sturm sequences, whereas others fail to be real-rooted but appear nevertheless to have log-concave coefficients. Finally, we relate the peak-clique presentation of a unit interval graph to lattice path matroids.
\end{abstract}

\maketitle
\thispagestyle{empty}
\setcounter{tocdepth}{2}
\tableofcontents

\newpage


\section{Introduction}

Unit interval graphs were introduced by Roberts under the name \emph{indifference graphs} \cite{Rob69}. They are the intersection graphs of families of intervals of equal length, and Roberts proved that they coincide with proper interval graphs, namely, interval graphs admitting a representation in which no interval properly contains another. Besides their role in structural and algorithmic graph theory, unit interval graphs arise naturally from semiorders and models of indifference. Their enumeration and the removal of redundant, or twin, vertices were studied by Hanlon \cite{HanlonCounting1982}; the latter point of view leads to the reduced unit interval graphs considered in Section~\ref{sec:reduced}.

Unit interval graphs have also become important objects in algebraic combinatorics. Stanley and Stembridge formulated their celebrated positivity conjecture for incomparability graphs of $(3+1)$-free posets \cite{SS93}, and Stanley subsequently recast it in terms of his chromatic symmetric function \cite{Sta95}. Guay-Paquet showed that this conjecture reduces to the case of posets that are simultaneously $(3+1)$-free and $(2+2)$-free, equivalently unit interval orders \cite{GP13}. Shareshian and Wachs later introduced a quasisymmetric refinement of Stanley's invariant and highlighted the distinguished role played by natural unit interval orders \cite{SW16}. In a closely related direction, the Dyck path model for unit interval graphs has been used in the study of chromatic quasisymmetric functions and unicellular LLT polynomials; see, for example, \cite{AlexanderssonPanova2018}. Thus, area sequences, Dyck paths, natural unit interval orders, and naturally labeled unit interval graphs provide interchangeable models for a common Catalan family. Throughout this paper, we use the canonical labeling supplied by area sequences and refer to the resulting objects simply as unit interval graphs on $[n]$.

Dyck paths carry many classical statistics. The distribution of their number of peaks is given by the Narayana numbers, while Touchard's identity expresses the Catalan numbers as a weighted sum of Motzkin numbers \cite{Touchard1928}; see also \cite{motzkin-aigner} for background on Motzkin numbers. Under the Dyck path--unit interval graph correspondence, peaks become certain maximal cliques, which we call \emph{peak-cliques}. This suggests studying not only how many such cliques occur, but also how they overlap. The main new statistic of this paper measures precisely this phenomenon. For a vertex $i$, its \emph{peak-nesting} is the number of peak-cliques containing $i$, and the peak-nesting of the graph is the maximum of these numbers over all vertices. Equivalently, this defines a new statistic on Dyck paths through their associated unit interval graphs. We study its distribution together with the classical number of peaks.

Questions about the zeros and coefficients of combinatorial polynomials form a second theme of the paper. A polynomial with nonnegative coefficients and only real zeros has an ultra log-concave coefficient sequence, and therefore a log-concave and unimodal one. Symmetric real-rooted polynomials also admit nonnegative expansions in the gamma-basis. These implications have motivated a large body of work on real-rootedness, interlacing, log-concavity, and gamma-positivity; see \cite{Branden2015,Athanasiadis2018} and the references therein. Our peak and peak-nesting polynomials provide a varied collection of examples: some are real-rooted and even form generalized Sturm sequences, while others have nonreal zeros but retain strong coefficient inequalities. This makes it natural to ask which positivity property is the strongest one valid for each graph class.

Our main tool is a bijection
\[
   \varphi_n \colon \DP(n) \to \motzkin^{*}(n)
\]
between Dyck paths of semilength $n$ and bicolored Motzkin paths of length $n-1$; see Theorem \ref{th:bij-dyck-decorated-motzkin}. The uncolored height sequence of $\varphi_n(P)$ is the peak-nesting vector of the unit interval graph associated with $P$. Consequently,
\[
   \pnest(P) = 1 + \height( \varphi_n(P) ).
\]
At the same time, the number of peaks of $P$ is one plus the number of up-steps and marked horizontal steps of $\varphi_n(P)$. The bijection therefore records the two statistics simultaneously. It is the organizing principle behind most of our enumerative results and explains the repeated appearance of Motzkin and Riordan paths throughout the paper.

We next describe the main results in more detail. In Section \ref{sec:preliminaries} we recall the necessary background on log-concave sequences, interlacing polynomials, multiplier sequences, and gamma-expansions. In Section \ref{sec:unit} we introduce peak-cliques, the peak number, and peak-nesting. The bijection of Theorem \ref{th:bij-dyck-decorated-motzkin} implies that the vectors occurring as peak-nesting vectors are exactly the Motzkin sequences. It also gives a refinement of Touchard's identity in which Catalan objects are enumerated by peak-nesting; see Corollary \ref{cor:refinement}. In Corollary \ref{ref:pnest-formulas} we express the coefficients of the peak-nesting polynomials, both in the unrestricted and connected cases, in terms of the number of Dyck paths of given height. These formulas identify the first coefficients with powers of two, Stirling numbers of the second kind, and Fibonacci numbers, and they allow us to prove log-concavity in the upper range of the coefficient sequences. The peak-nesting polynomials are not real-rooted in general, but computations lead us to conjecture ultra log-concavity in Conjecture \ref{conj:pnest-ulc}. We also obtain recursive generating functions for unit interval graphs of bounded peak-nesting; in the limit they recover the classical continued fraction for the Catalan generating function. The peak polynomial over all unit interval graphs is the classical Narayana polynomial, and the connected version is obtained by lowering the semilength by one.

Section \ref{sec:abelian} concerns Abelian unit interval graphs. We first characterize them as the unit interval graphs whose vertex set can be divided into two consecutive cliques, and deduce that there are $2^{n-1}$ such graphs on $[n]$. The peak-nesting polynomial has coefficients $\displaystyle \binom{n}{2i-1}$, giving a new interpretation of \cite[\oeis{A034867}]{OEIS}, while the peak polynomial has coefficients $\displaystyle \binom{n}{2i}$, and yields a new interpretation of \cite[\oeis{A034839}]{OEIS}. We obtain radical expressions and recurrences for both families and prove that they form generalized Sturm sequences; see Theorem \ref{thm:pnest-pol-abelian} and Corollaries \ref{pnest-pol-abelian-rr} and \ref{cor:peak-pol-albelian-rr}. The corresponding connected peak polynomial is also real-rooted. For the symmetric even-indexed members, Propositions \ref{th:gamma-Abelian} and \ref{prop:gamma-peak-abelian} give explicit gamma-polynomials and interpret their coefficients by pattern-avoiding Abelian Dyck paths.

In Section \ref{sec:line-graphs} we prove a structural characterization that illustrates the usefulness of the new statistic: a unit interval graph is a line graph of a triangle-free multigraph if and only if its peak-nesting is at most two; see Proposition \ref{charac-unit-interval-line}. The bicolored Motzkin model then reduces to paths of height at most one. This gives formulas for the numbers of $2$-nested unit interval graphs and their connected counterparts, as well as closed expressions and recurrences for their peak and peak-nesting polynomials. The unrestricted and connected peak-nesting and peak polynomials are real-rooted and form generalized Sturm sequences; see Corollaries \ref{cor:pnest-line-rr}, \ref{cor:peak-line-rr} and \ref{cor:cpeak-line-rr}. Their coefficients give graph-theoretic interpretations of the arrays \cite[\oeis{A056241}]{OEIS} and \cite[\oeis{A085478}]{OEIS}. Moreover, the unrestricted peak polynomial is symmetric after division by $t$, and Proposition \ref{prop:gamma-line} determines its gamma-vector and interprets it using Motzkin paths of height at most one.

Section \ref{sec:reduced} is devoted to reduced unit interval graphs, that is, unit interval graphs without adjacent twin vertices. Under our bijection, reduced graphs correspond to Motzkin paths in which every horizontal step is marked. As a consequence, reduced and connected reduced unit interval graphs are counted by Motzkin and Riordan numbers, respectively. Their intersections with the class of $2$-nested graphs are counted by powers of two and Fibonacci numbers. More generally, Proposition \ref{prop:twin-formula} shows that the distribution of the twin number is given by binomial transforms of Motzkin numbers, providing a new interpretation of \cite[\oeis{A091869}]{OEIS}; its coefficient sequences are simultaneously log-concave and ultra log-convex. The peak-nesting polynomials of reduced graphs satisfy natural recurrences involving the maximum product of polynomials. They are neither real-rooted nor ultra log-concave in general, but our computations suggest that their coefficients are always log-concave; see Conjecture \ref{conj:pnest-reduced}. In contrast, the peak polynomials of reduced and connected reduced unit interval graphs are real-rooted. This follows from explicit formulas for the ascent polynomials of Motzkin and Riordan paths, together with interlacing and multiplier-sequence arguments; see Theorem \ref{th:peak-reduced}, Corollary \ref{cor:peak-reduced-rr}, Theorem \ref{th:peak-connected-reduced}, and Corollary \ref{cor:peak-connected-reduced-rr}. These formulas also give new interpretations of the relevant OEIS arrays.

Table \ref{tab:real-rootedness-summary} summarizes the results obtained for the polynomials studied in this project. Observe that the real-rootedness of some of the polynomials presented in Table \ref{tab:real-rootedness-summary}, such as the peak-nesting polynomial of abelian, $2$-nested and connected $2$-nested unit interval graphs, and the peak polynomial of abelian, connected $2$-nested and reduced unit interval graphs, follows from either the theory of Pólya frequency sequences \cite{Brenti1989}, or orthogonal polynomials. We also show that they form generalized Sturm sequences.

We conclude in Section \ref{sec:open} with generalizations and further directions. We introduce $m$-Abelian and bounded peak-nesting unit interval graphs and ask for the corresponding enumerative and real-rootedness properties. Finally, the ordered family of peak-cliques of a unit interval graph is an interval presentation of a lattice path matroid. Using this correspondence and the structural theory of lattice path matroids \cite{BoninMierNoy2003, BoninMier2006, Bonin2010, multi-path}, we translate the Abelian, $m$-nested, and reduced conditions into restrictions on lattice path matroid presentations and diagrams. This connection suggests studying bases, circuits, flats, minors, duality, Tutte polynomials, and Ehrhart polynomials of lattice path matroids through their associated unit interval graphs.

\begin{table}[H]
    \centering
    \begin{tabularx}{\textwidth}{@{}X
        >{\centering\arraybackslash}p{0.25\textwidth}
        >{\centering\arraybackslash}p{0.25\textwidth}@{}}
        \toprule
        \textbf{Class of unit interval graphs}
        & \textbf{Peak-nesting}
        & \textbf{Peaks} \\
        \midrule
        All
        & ULC (Conjecture \ref{conj:pnest-ulc})
        & SS \cite{Branden2015}\\

        Connected
        & ULC (Conjecture \ref{conj:pnest-ulc})
        & GSS \cite{Branden2015}\\

        \addlinespace
        Abelian
        & GSS (\cref{pnest-pol-abelian-rr})
        & GSS (\cref{cor:peak-pol-albelian-rr}) \\

        Connected Abelian
        & ULC (\cref{cor:ulc-pnest})
        & RR (\cref{cor:peak-connected-abelian-rr}) \\

        \addlinespace
        $2$-nested
        & GSS (\cref{cor:pnest-line-rr})
        & SS (\cref{cor:peak-line-rr}) \\

        Connected $2$-nested
        & GSS (\cref{cor:pnest-line-rr})
        & GSS (\cref{cor:cpeak-line-rr}) \\

        \addlinespace
        Reduced
        & LC (\cref{conj:pnest-reduced})
        & GSS (\cref{cor:peak-reduced-rr}) \\

        Connected reduced
        & LC (\cref{conj:pnest-reduced})
        & RR (\cref{cor:peak-connected-reduced-rr}) \\
        \bottomrule
    \end{tabularx}
    \caption{Positivity results of the peak-nesting and peak polynomials for the classes considered in this paper. The acronyms GSS, SS, RR, ULC and LC refer to generalized Sturm sequence, Sturm sequence, real-rooted, ultra log-concavity and log-concavity, respectively.}
    \label{tab:real-rootedness-summary}
\end{table} 

\section{Preliminaries}
\label{sec:preliminaries}

Let $n \in \setN = \{1,2,3,\ldots\}$ and set $[n] \coloneqq \{1, 2, \ldots, n\}$.

\subsection{Sequences of real numbers}

Let $\mathcal{S} = \{b_i\}_{i=0}^d$ be a sequence of real numbers.
We say that 
\begin{itemize}
    \item $\mathcal{S}$ is \defin{unimodal} if there exists $0 \leq m \leq d$ such that $b_0 \leq \cdots \leq b_{m-1} \leq b_m \geq b_{m+1} \geq \cdots \geq b_d$;
    \item $\mathcal{S}$ is \defin{log-concave} (resp. \defin{log-convex}) if $b_i^2 \geq b_{i+1} b_{i-1}$ (resp. $b_i^2 \leq b_{i+1} b_{i-1}$) for all $i=1, \ldots, d-1$;
    \item $\mathcal{S}$ is \defin{ultra log-concave} (resp. \defin{ultra log-convex}) if $\left\{ \nicefrac{b_i}{\binom{d}{i}} \right\}_{i=0}^d$ is log-concave (resp. log-convex);
    \item the generating polynomial of $\mathcal{S}$, $p_\mathcal{S}(t) = b_0 + b_1 t + \ldots b_d t^d$, is \defin{real-rooted} if $p_\mathcal{S}(t)$ has only real zeros.
\end{itemize}

The following classical lemma connects the four properties above:

\begin{lemma}
\label{lem:sequences}
    Let $\mathcal{S} = \{b_i\}_{i=0}^d$ be a sequence of nonnegative real numbers. If
    \begin{enumerate}[(i)]
        \item $p_\mathcal{S}(t)$ is real-rooted, then $\mathcal{S}$ is ultra log-concave;
        \item $\mathcal{S}$ is ultra log-concave, then $\mathcal{S}$ is log-concave;
        \item $\mathcal{S}$ is log-concave and $b_i > 0$ for all $i = 0, 1, \ldots, d$, then $\mathcal{S}$ is unimodal.
    \end{enumerate}
\end{lemma}

\begin{remark}
    Let $\mathcal{S}$ be a sequence as in the lemma above. Then the log-convexity of $\mathcal{S}$ implies the ultra log-convexity of $\mathcal{S}$, which is the converse implication of what happens between log-concave and ultra log-concave sequences.
\end{remark}

\subsection{Interlacing sequences of polynomials}

The following definitions related to real-rootedness and interlacing are relevant in some proofs further down. Good sources of such topics are \cite{Wagner1992}, \cite{Branden2015} and \cite{Alexandersson2020}.

Let $\mathbb{R}[t]$ be the space of all polynomials with coefficients in $\mathbb{R}$. A sequence $\sigma = \{\sigma_i\}_{i \geq 0}$ of real numbers is called a \defin{multiplier sequence} if the linear operator
\[
    \begin{matrix}
            T_\sigma \colon & \mathbb{R}[t]          & \to     & \mathbb{R}[t] \\
                            & \sum_{i=0}^d b_i t^i & \mapsto & \sum_{i=0}^d \sigma_i b_i t^i
    \end{matrix}
\]
maps real-rooted polynomials to real-rooted polynomials. The following result characterizes multiplier sequences:

\begin{lemma}\cite{SchurPolya1914,Karlin1968}
\label{lem:schur-polya}
    A sequence $\sigma = \{\sigma_i\}_{i \geq 0}$ of real numbers is a multiplier sequence if and only if the exponential generating function
    \[
        \Phi_\sigma (z) = \sum_{i \geq 0} \sigma_i \frac{z^i}{i!}
    \]
    or $\Phi_\sigma (-z)$ has the form
    \[
        C z^m e^{az} \prod_{j \geq 1} (1 + \lambda_j z),
    \]
    where $C \in \mathbb{R}$, $m \in \mathbb{Z}_{\geq 0}$, $a \geq 0$, $\lambda_j \geq 0$ and $\sum_{j \geq 1} \lambda_j < \infty$.
\end{lemma}

We now turn our attention to interlacing polynomials, a powerful theory used to prove the real-rootedness of many sequences of polynomials. Suppose $f(t)$ and $g(t)$ are real-rooted polynomials with positive leading coefficients, and that the zeros of $f(t)$ and $g(t)$ are
\[
    \cdots \leq \lambda_3 \leq \lambda_2 \leq \lambda_1  \ \  \mbox{ and } \ \ \cdots \leq \delta_3 \leq \delta_2 \leq \delta_1,
\]
respectively. We say that the zeros of $g(t)$ \defin{interlace} those of $f(t)$ if 
\[
 \cdots \leq \delta_3 \leq \lambda_3 \leq \delta_2 \leq \lambda_2 \leq \delta_1 \leq \lambda_1, 
\]
and write $\defin{g(t) \interl f(t)}$. In particular the degrees of $f(t)$ and $g(t)$ differ by at most one. By convention we also write, $0 \interl 0$, $0 \interl f(t)$ and $f(t) \interl 0$ for any real-rooted polynomial $f(t)$ with positive leading coefficient. Let $\mathcal{F} = \{f_i(t)\}_{i=0}^d$ be a sequence of real-rooted polynomials with positive leading coefficients. If $f_{i-1}(t) \interl f_i(t)$ for all $i \in [n]$ and $\deg(f_i(t)) = i$ for all $i=0, 1, \ldots, d$, then $\mathcal{F}$ is called a \defin{Sturm sequence}, see \cite{Liu2007}. If we drop the degree condition, then $\mathcal{F}$ is called a \defin{generalized Sturm sequence}. Moreover, if $f_i(t) \interl f_j(t)$ for all $0 \leq i \leq j \leq d$, then $\mathcal{F}$ is called an \defin{interlacing sequence}.

\begin{lemma}
\label{lem:interlacing-properties}
    \hfill
    \begin{enumerate}[(i)]
        \item If $f(t), g(t), h(t) \in \mathbb{R}[t]$ are such that $h(t) \interl f(t)$ and $h(t) \interl g(t)$, then $h(t) \interl f(t) + g(t)$;
        \item If $f(t), g(t), h(t) \in \mathbb{R}[t]$ are such that $f(t) \interl h(t)$ and $g(t) \interl h(t)$, then $f(t) + g(t) \interl h(t)$;
        \item If $f(t), g(t) \in \mathbb{R}[t]$ have only nonnegative coefficients, then $g(t) \interl f(t)$ if and only if $f(t) \interl t \cdot g(t)$;        
     \end{enumerate}
\end{lemma}

\begin{proof}
    The proofs of (i) and (ii) can be found in \cite[Proposition 3.5]{Wagner1992}, while the proof of (iii) is trivial.
\end{proof}

\begin{lemma}\cite[Theorem 2.1]{Liu2007}
\label{recursion-rr}
    Let $F(t), f(t), g(t), a(t), b(t) \in \mathbb{R}[t]$ be such that
    \[
        F(t) = a(t) f(t) + b(t) g(t)
    \]
    and $\deg(F(t)) \in \{ \deg(f(t)), \deg(f(t)) + 1 \}$. Assume $g(t) \interl f(t)$, and $F(t)$ and $g(t)$ have leading coefficients of the same sign. If $b(r) \leq 0$ whenever $f(r) = 0$, then $F(t)$ is real-rooted and $f(t) \interl F(t)$.
\end{lemma}

\subsection{Symmetric polynomials and the gamma-polynomial}

Let $p(t) = \sum_{i \geq 0}^d b_i t^i$ be polynomial with $b_d \neq 0$. Then $p(t)$ is called \defin{symmetric} (or \defin{palindromic}) if $p(t) = t^d \cdot p(1/t)$, and $d$ is the \defin{symmetric degree} of $p(t)$.

Symmetric polynomials with symmetric degree $d$ span a vector space of dimension $\lfloor d / 2 \rfloor + 1$. Hence, they can be written in the basis
\[
    \{ t^i (1+t)^{d-2i} \}_{i=0}^{\lfloor d / 2 \rfloor},
\]
that is,
\[
    p(t) = \sum_{0 \leq 2i \leq d} \gamma_i t^i(1+t)^{d-2i}.
\]
The polynomial $\gamma_p (t) = \sum_{i=0}^{\lfloor d / 2 \rfloor} \gamma_i t^i$ is called the \defin{gamma-polynomial} of $p(t)$, and $p(t)$ is called \defin{gamma-nonnegative} if all the coefficients of $\gamma_p (t)$ are nonnegative.

Gamma-nonnegativity is a topic of interest, since this property is related to questions regarding real-rootedness, log-concavity and unimodality, see \cite{Athanasiadis2018, foata1970, Branden2004, Gal2005} for further reading. In fact, an immediate consequence of the nonnegativity of the coefficients of $\gamma_p (t)$ is the unimodality of the coefficients of $p(t)$, which arises by expanding the coefficients of $(1+t)^{d-2i}$ when $p(t)$ is written in the basis $\{ t^i (1+t)^{d-2i} \}_{i=0}^{\lfloor d / 2 \rfloor}$ and observing that each $\gamma_i$ contributes a symmetric unimodal summand.

The following two lemmas are classical results and will be used throughout this paper:

\begin{lemma}\cite[Lemma 4.1]{Branden2004}
\label{lem:gamma1}
    If $p(t)$ and $q(t)$ are symmetric polynomials of symmetric degrees $c$ and $d$, respectively, then $(pq)(t)$ is a polynomial of symmetric degree $c+d$. Moreover, if $p(t)$ is a symmetric polynomial with positive leading coefficient where all its zeros are nonpositive real numbers, then $p(t)$ is gamma-nonnegative. 
\end{lemma}

\begin{lemma}\cite[Section 4.6]{eulerian-kyle}
\label{lem:gamma2}
    A symmetric polynomial $p(t)$ is real-rooted if and only if all the roots of $\gamma_p(t)$ belong to $(-\infty, 1/4]$.
\end{lemma}

\section{Unit interval graphs}
\label{sec:unit}

In this section, we define the peak and the peak-nesting statistics on unit interval graphs. These statistics will be studied in detail in Sections \ref{sec:pnest-all} and \ref{sec:peak-all}, where we define the peak and the peak-nesting polynomials, and find combinatorial interpretations for their coefficients. We also study properties such as real-rootedness and log-concavity.

For undefined graph terminology, we refer to \cite{douglas-graph-theory}. A \defin{graph} is a pair $G = (V, E)$ where $V$, the \defin{vertex set} of $G$, is a set whose elements are called \defin{vertices} and $E \subseteq \binom{V}{2}$, the \defin{edge set} of $G$, is a set whose elements are called \defin{edges}. We will denote an edge $\{u,v\}$ of $G$ by $uv$ (or, equivalently, $vu$). The vertices $u,v \in V$ are called \defin{adjacent} (or \defin{neighbors}) if $uv \in E$, that is, $uv$ is an edge of $G$.

A \defin{Dyck path} of semilength $n$ is a North-East ($NE$) lattice path starting at $(0,0)$ and ending at $(n,n)$ where, for all $i = 1, \ldots, 2n$, there are at least as many North steps (represented by the vector $(0,1)$ or by the letter $N$) as East steps (represented by the vector $(1,0)$ or by the letter $E$) among its first $i$ steps. Equivalently, it can be defined as a lattice path starting at $(0,0)$ and ending at $(2n,0)$ with steps in $\{U = (1,1), D = (1,-1)\}$ where, for all $i = 1, \ldots, 2n$, there are at least as many $U$ steps as $D$ steps among its first $i$ steps. A \defin{Dyck diagram} of size $n$ is a diagram of boxes in $\{(a,b) \in \mathbb{Z}^2 \colon 0 \leq a \leq b \leq n \}$ associated with a Dyck path where the boxes under the Dyck path are colored, see Figure \ref{fig:dyck-path-diagram}. 

An \defin{area sequence} is a finite list of nonnegative integers $\avec = (a_1,a_2,\dotsc,a_n)$ such that
\[
    a_i < i \text{ and } a_{i} \leq a_{i-1}+1 \text{ for all } i\geq 2.
\]
Area sequences of size $n$ are in bijection with Dyck diagrams of size $n$ (and, consequently, in bijection with Dyck paths of semilength $n$), where the number of blue boxes in the $i$th row of the Dyck diagram is $a_i$, see Example \ref{ex:uig}. Hence, we denote the set of area sequences of size $n$ by $\DP(n)$ as these essentially describe Dyck paths.

An area sequence $\avec = (a_1, \ldots, a_n)$ defines a graph $\Gamma_\avec$ with vertex set $[n]$ and edge set given by
\[
    \left\{ \{j, i\} \colon 1 \leq j < i \leq n, i - a_i \leq j \right\}.
\]
In particular, $\Gamma_\avec$ is connected if and only if $a_2, \ldots, a_n \geq 1$.

\begin{example}
\label{ex:uig}
    The diagram in Figure \ref{fig:dyck-path-diagram} illustrates the Dyck path associated with the area sequence $\avec = (0,1,1,2,3,3,2,0)$. Each labeled cell represents a vertex of $\Gamma_\avec$, and the cells between the Dyck path (outlined) and the labeled cells correspond to the edges of $\Gamma_\avec$.
    \begin{figure}[H]
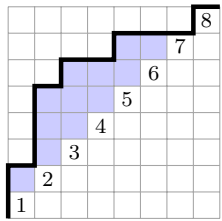

        \centering
            \dyckDiagram[0.35cm]{0,1,1,2,3,3,2,0}
        \caption{The Dyck diagram corresponding to the area sequence $(0,1,1,2,3,3,2,0)$.}
        \label{fig:dyck-path-diagram}
    \end{figure}
\end{example}

Let $\{S_i\}_i$ be a family of real intervals. An \defin{interval graph} is a graph whose vertex set is $\{S_i\}_i$ and whose edge set is
\[
    \{ S_i S_j \colon S_i \cap S_j \neq \emptyset \}.
\]
A \defin{unit interval graph} is an interval graph whose family of real intervals is a family of unit intervals, that is, intervals of length $1$. Since we work only with finite graphs, we may refer to finite unit interval graphs simply as unit interval graphs. The family of unit interval graphs is in bijection with the family of area sequences \cite{AlexanderssonPanova2018}, that is, for each area sequence $\avec \in \DP(n)$, $\Gamma_\avec$ is a unit interval graph and each unit interval graph is isomorphic to $\Gamma_\avec$ for some $\avec \in \DP(n)$. Hence, we may use properties of area sequences to study unit interval graphs.

Recall that a \defin{clique} in a graph $G$ is an isolated vertex or a subset $S \subseteq V$ of its set of vertices such that $uv$ is an edge of $G$ for every pair of vertices $u, v \in S$. The following lemma follows directly from the definition of area sequences:

\begin{lemma}
\label{lem:uig-clique}
    Let $1 \leq i < j \leq n$. If $ij$ is an edge in $\Gamma_\avec$, then $\{i,i+1,\dotsc, j\}$ is a clique in $\Gamma_\avec$.
\end{lemma}


\begin{definition}
    Let $1 \leq i \leq j \leq n$. A \defin{peak-edge} in a unit interval graph $\Gamma_\avec \in \mathcal{DP}(n)$ is an isolated vertex $i$ or an edge $ij$ in a where neither $\{i, j+1\}$ nor $\{i-1, j\}$ is an edge. The set of \defin{peak-cliques} of $\Gamma_\avec$ is the set of cliques determined by the peak-edges, as in Lemma~\ref{lem:uig-clique}. The \defin{peak-nesting of a vertex} $i$ in $\Gamma_\avec$ is the number of peak-cliques that contain $i$ and it is denoted by $\pnest(i)$. The \defin{peak-nesting vector} of $\Gamma_\avec$ is the vector $(\pnest(1), \ldots, \pnest(n))$. The \defin{peak-nesting} of $\Gamma_\avec$ is the maximal peak-nesting over all vertices, and is denoted by $\pnest(\avec)$. The \defin{peak number} of $\Gamma_\avec$ is the number of peak-edges of $\Gamma_\avec$ and it is denoted by $\peak(\avec)$.
\end{definition}

\begin{remark}
\label{rmk:bij-peaks-peak-edges}
    Equivalently, a \defin{peak} in a Dyck path is a North step followed by an East step, and there is a clear bijection between peaks in a Dyck path and peak-edges in its correspondent unit interval graph, see Figure \ref{fig:dyck-path-diagram} and Example \ref{ex:def-peaks}. Hence, we may refer to a peak-edge in a unit interval graph simply by a peak.
\end{remark}

\begin{example}
\label{ex:def-peaks}
    In Example~\ref{ex:uig}, the edges $1 2$, $2 5$, $3 6$, $5 7$ and the vertex $8$ are the peaks of $\Gamma_\avec$. Its set of peak-cliques is
    \[
        \mathcal{P}_\avec = \{ \{1,2\}, \quad \{2,3,4,5\}, \quad \{3,4,5,6\}, \quad \{5,6,7\}, \quad \{8\} \},
    \]
    from which follows that its peak number is $5$ and its peak-nesting vector is
    \[
        (1,2,2,2,3,2,1,1).
    \]
    Hence, $\pnest(\avec) = 3$.
\end{example}

\begin{example}
    The following picture shows the 14 area sequences of length 4, with the peak number, the peak-nesting and the peak-nesting vector written under each diagram.
    \begin{center}
        \setlength{\tabcolsep}{1pt}
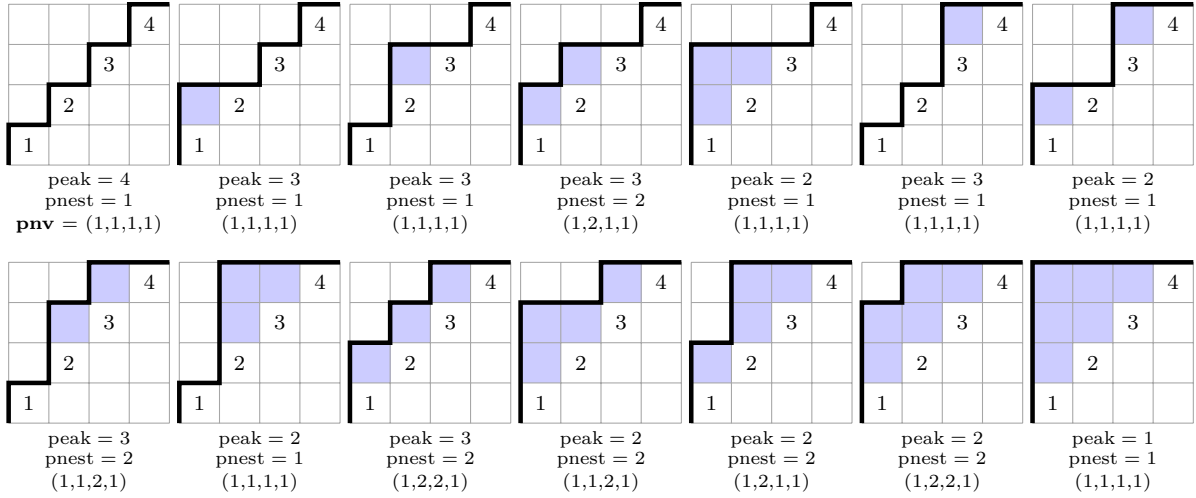
        \begin{tabular}{*{7}{c}}
            \shortstack{\dyckDiagram[0.03125\textwidth]{0,0,0,0}\\[-1pt]{\scriptsize $\peak = 4$}\\[-1pt]{\scriptsize $\pnest = 1$}\\[-1pt]{\scriptsize $\pnv$ = (1,1,1,1)}} &
            \shortstack{\dyckDiagram[0.03125\textwidth]{0,1,0,0}\\[-1pt]{\scriptsize $\peak = 3$}\\[-1pt]{\scriptsize $\pnest = 1$}\\[-1pt]{\scriptsize (1,1,1,1)}} &
            \shortstack{\dyckDiagram[0.03125\textwidth]{0,0,1,0}\\[-1pt]{\scriptsize $\peak = 3$}\\[-1pt]{\scriptsize $\pnest = 1$}\\[-1pt]{\scriptsize (1,1,1,1)}} &
            \shortstack{\dyckDiagram[0.03125\textwidth]{0,1,1,0}\\[-1pt]{\scriptsize $\peak = 3$}\\[-1pt]{\scriptsize $\pnest = 2$}\\[-1pt]{\scriptsize (1,2,1,1)}} &
            \shortstack{\dyckDiagram[0.03125\textwidth]{0,1,2,0}\\[-1pt]{\scriptsize $\peak = 2$}\\[-1pt]{\scriptsize $\pnest = 1$}\\[-1pt]{\scriptsize (1,1,1,1)}} &
            \shortstack{\dyckDiagram[0.03125\textwidth]{0,0,0,1}\\[-1pt]{\scriptsize $\peak = 3$}\\[-1pt]{\scriptsize $\pnest = 1$}\\[-1pt]{\scriptsize (1,1,1,1)}} &
            \shortstack{\dyckDiagram[0.03125\textwidth]{0,1,0,1}\\[-1pt]{\scriptsize $\peak = 2$}\\[-1pt]{\scriptsize $\pnest = 1$}\\[-1pt]{\scriptsize (1,1,1,1)}}
            \\[8pt]
            \shortstack{\dyckDiagram[0.03125\textwidth]{0,0,1,1}\\[-1pt]{\scriptsize $\peak = 3$}\\[-1pt]{\scriptsize $\pnest = 2$}\\[-1pt]{\scriptsize (1,1,2,1)}} &
            \shortstack{\dyckDiagram[0.03125\textwidth]{0,0,1,2}\\[-1pt]{\scriptsize $\peak = 2$}\\[-1pt]{\scriptsize $\pnest = 1$}\\[-1pt]{\scriptsize (1,1,1,1)}} &
            \shortstack{\dyckDiagram[0.03125\textwidth]{0,1,1,1}\\[-1pt]{\scriptsize $\peak = 3$}\\[-1pt]{\scriptsize $\pnest = 2$}\\[-1pt]{\scriptsize (1,2,2,1)}} &
            \shortstack{\dyckDiagram[0.03125\textwidth]{0,1,2,1}\\[-1pt]{\scriptsize $\peak = 2$}\\[-1pt]{\scriptsize $\pnest = 2$}\\[-1pt]{\scriptsize (1,1,2,1)}} &
            \shortstack{\dyckDiagram[0.03125\textwidth]{0,1,1,2}\\[-1pt]{\scriptsize $\peak = 2$}\\[-1pt]{\scriptsize $\pnest = 2$}\\[-1pt]{\scriptsize (1,2,1,1)}} &
            \shortstack{\dyckDiagram[0.03125\textwidth]{0,1,2,2}\\[-1pt]{\scriptsize $\peak = 2$}\\[-1pt]{\scriptsize $\pnest = 2$}\\[-1pt]{\scriptsize (1,2,2,1)}} &
            \shortstack{\dyckDiagram[0.03125\textwidth]{0,1,2,3}\\[-1pt]{\scriptsize $\peak = 1$}\\[-1pt]{\scriptsize $\pnest = 1$}\\[-1pt]{\scriptsize (1,1,1,1)}}
        \end{tabular}
    \end{center}
\end{example}

We are interested in studying the distribution of the peak number and of the peak-nesting statistic over unit interval graphs on $[n]$. In the following two subsections, we discuss both topics.

\subsection{The peak-nesting polynomial of unit interval graphs}
\label{sec:pnest-all}
In this section, we study the peak-nesting polynomial over all unit interval graphs and over connected unit interval graphs. We express the peak-nesting polynomial over all graphs as a formula of heights of Motzkin and Dyck paths, which in particular is a refinement of the Touchard's identity \cite{Touchard1928}, see Corollary \ref{cor:refinement}. We also find coefficientwise formulas for the peak-nesting polynomials, and prove the log-concavity for the upper half coefficients of both of them.

We start with the following simple result, which follows directly from the definition of peak-cliques:

\begin{lemma}
\label{peak-cliques-are-maximal}
    Let $P$ be a peak-clique of a unit interval graph $\Gamma_\avec$. Then there is no clique $C$ of $\Gamma_\avec$ such that $P \subsetneq C$, that is, each peak-clique is a maximal clique of a unit interval graph.
\end{lemma}

\begin{proof}
    Let $P$ be a peak-clique of $\Gamma_\avec$. By definition of area sequences, there exist $i,j \in [n]$ such that $i \leq j$ and $P = \{i, i+1, \ldots, j\}$. Moreover, by definition of peak-cliques, neither $\{i-1, j\}$ nor $\{i, j+1\}$ are edges of $\Gamma_\avec$. Hence, by definition of area sequences, there is neither $k < i$ such that $\{k, j\}$ is an edge of $\Gamma_\avec$, nor $k > j$ such that $\{i, k\}$ is an edge of $\Gamma_\avec$. Therefore there is no clique $C$ of $\Gamma_\avec$ such that $P \subsetneq C$.
\end{proof}

We are interested in studying the distribution of the peak-nesting statistic over unit interval graphs on $[n]$. The \defin{peak-nesting polynomial on $[n]$} is defined as $\Pnest_0(t) \coloneqq 1$, and
\[
    \Pnest_n(t) \coloneqq \sum_{\avec \in \mathcal{DP}(n)} t^{\pnest(\avec)}
\]
for $n \in \mathbb{N}$, see Table \ref{tab:peak-nesting-polynomial}. We also investigate the following variation of the peak-nesting polynomial, where the sum is over connected unit interval graphs: we set $\Pnest_0^c(t) \coloneqq 0$, and
\[
    \Pnest_n^c(t) \coloneqq \sum_{\substack{\avec \in \DP(n) \\ \Gamma_\avec \mbox{ \scriptsize is connected}}} t^{\pnest(\avec)}.
\]
for $n \in \mathbb{N}$.

We now describe the coefficients of $\Pnest_n(t)$ and $\Pnest_n^c(t)$ combinatorially using Motzkin numbers. A \defin{(uncolored) Motzkin path} of length $n \geq 0$ is a lattice path starting at $(0,0)$ and ending at $(n,0)$ with steps in $\{U = (1,1), D = (1,-1), H = (1,0)\}$ such that, for all $i = 1, 2, \ldots, n$, there are at least as many $U$ steps as $D$ steps among its first $i$ steps. In particular, every Dyck path is a Motzkin path. A \defin{(uncolored) Motzkin sequence} of size $n$ is a list of positive integers $\wvec = (w_1,w_2,\dotsc,w_n)$ where $w_1=w_n=1$ and $|w_i-w_{i+1}|\leq 1$ whenever $1 \leq i < n$. For each $n \in \mathbb{N}$, we construct a bijection between the set of all Motzkin paths of length $n-1$ and all Motzkin sequences of size $n$ as follows: given a Motzkin sequence $\wvec = (w_1,w_2,\dotsc,w_n)$, we define a Motzkin path $M = S_1 S_2 \cdots S_{n-1}$ whose $i$th step $S_i$ is given by
\begin{equation}
\label{eq:bij-motzkin-path-seq}
    S_i =
    \begin{cases}
        U &\mbox{if } w_{i+1} - w_i = 1, \\
        D &\mbox{if } w_{i+1} - w_i = -1, \\
        H &\mbox{if } w_{i+1} - w_i = 0
    \end{cases}
\end{equation}
for each $i \in [n-1]$. Given $n \geq 0$, the set of all Motzkin sequences of size $n$ (or the set of all Motzkin paths of length $n-1$ when $n \geq 1$), will be denoted by \defin{$\motzkin(n)$}. Observe that $\DP(n) \subseteq \motzkin(2n+1)$ for all $n \geq 0$.

An \defin{ascent}, a \defin{descent} and a \defin{plateau} of a Motzkin sequence $\wvec = (w_1,w_2,\dotsc,w_n)$ is an index $i \in [n-1]$ such that 
\[
    w_i < w_{i+1},\quad  w_i > w_{i+1},\quad  \text{ and } \quad  w_i = w_{i+1},
\]
respectively.  The number of ascents, descents and plateaus are denoted, respectively, by \defin{$\asc(\wvec)$}, \defin{$\des(\wvec)$} and \defin{$\plateau(\wvec)$}, and it is clear that $\asc(\wvec) = \des(\wvec)$ for every Motzkin sequence $\wvec$.

A \defin{bicolored Motzkin sequence} is a Motzkin sequence where each plateau is either marked with $*$ or not. If $\wvec^*$ is a bicolored Motzkin sequence, we let $\wvec$ denote the uncolored version. A \defin{bicolored Motzkin path} is a Motzkin path where each $H$ step is either marked with $*$ or not. The number of $U$ steps in a bicolored Motzkin path $M^*$ is denoted by $\defin{\mathrm{u}(M^*)}$, while its number of $H^*$ steps is denoted by $\defin{\mathrm{h}^*(M^*)}$. Observe that $\mathrm{u}(M^*)$ also denotes the number of $D$ steps of $M^*$.

As in the case of Motzkin sequences and Motzkin paths, there is also a bijection between bicolored Motzkin sequences of size $n$ and bicolored Motzkin paths of length $n-1$: given a bicolored Motzkin sequence $\wvec^* = (v_1, \ldots, v_n)$, let $\wvec = (w_1, \ldots, w_n)$ be its uncolored version. Define a bicolored Motzkin path $M^* = S_1 S_2 \cdots S_{n-1}$ whose $i$th step $S_i$ is given by
\begin{equation}
\label{eq:bij-dmotzkin}
    S_i =
    \begin{cases}
        U &\mbox{if } w_{i+1} - w_i = 1, \\
        D &\mbox{if } w_{i+1} - w_i = -1, \\
        H &\mbox{if } w_{i+1} - w_i = 0 \mbox{ and } v_i \neq w_{i+1}^*,\\
        H^* &\mbox{if } v_i = w_{i+1}^*.
    \end{cases}
\end{equation}
It is straightforward to verify that the map $\wvec^* \mapsto M^*$ is a bijection between the set of all bicolored Motzkin sequences of size $n$ and the set of all bicolored Motzkin paths of size $n-1$.

Given $n \geq 0$, we denote the set of all bicolored Motzkin sequences of size $n$ (or the set of all bicolored Motzkin paths of length $n-1$ if $n \geq 1$) by \defin{$\motzkin^*(n)$}. It is clear that $\motzkin(n) \subseteq \motzkin^*(n)$ for all $n \geq 0$.

The following theorem is crucial for the proof of Corollary \ref{cor:refinement}:

\begin{theorem}
\label{th:bij-dyck-decorated-motzkin}
    For each $n \geq 0$ there is a bijection $\varphi_n: \DP(n) \rightarrow \motzkin^*(n)$, with the following properties: if $\varphi_n(\avec)=\wvec^*_\avec$, then
    \begin{enumerate}[(i)]
        \item $\pnv(\avec) = \wvec_\avec$;
        \item $\peak(\avec) = 1 + \asc(\wvec^*_\avec)+\decor(\wvec^*_\avec)$.
    \end{enumerate}
\end{theorem}

\begin{proof}
    For $n=0$, the claim is trivial, so let $n \in \mathbb{N}$. We start by proving the existence of a bijection $\varphi_n$ as follows: given $\avec \in \mathcal{DP}(n)$, let $\wvec_\avec = (w_1, \ldots, w_n)$ be the peak-nesting vector of $\Gamma_\avec$. Let $\wvec^*_\avec = (v_1, \ldots, v_n)$ where, for each, $1 \leq i \leq n$,
    \[
        v_i =
        \begin{cases}
            w_i^* &\mbox{if $w_i = w_{i+1}$ and the collections of peak-cliques containing $i$ and $i+1$ are different}, \\
            w_i &\mbox{otherwise}.
        \end{cases}
    \]
    We define $\varphi_n(\avec) \coloneqq \wvec^*_\avec$ and claim that $\varphi_n$ is a bijection. We first prove that $\varphi_n$ is surjective. In fact, given $\wvec^* \in \motzkin^*(n)$, let $M^* = S_1 \cdots S_{n-1}$ be the bicolored Motzkin path given by \eqref{eq:bij-dmotzkin}. We define a finite sequence of integer intervals $[a_k, b_k] = \{i \in \mathbb{Z} \colon a_k \leq i \leq b_k \}$ as follows: start an interval at vertex $1$. For each $i \in [n-1]$, if $S_i$ equals
    \begin{itemize}
        \item $U$, start a new interval at the vertex $i+1$;
        \item $D$, end the oldest active interval at $i$;
        \item $H$, keep the intervals unchanged;
        \item $H^*$, end the oldest active interval at $i$, and start a new one at $i+1$.
    \end{itemize}
    End the remaining interval at $n$. The sequence of finite integer intervals obtained is the set of peak-cliques $\mathcal{P}_\avec$ of a unit interval graph $\avec$ for which $\varphi_n(\avec) \coloneqq \wvec^*$, by construction. Hence, $\varphi_n$ is surjective.
    
    For example, let $\wvec^*_\avec = (1,2,2^*,2,3^*,3,2,1)$ and let $\wvec_\avec$ be its uncolored version. We construct the peak-cliques of $\Gamma_\avec$ step-by-step as follows, see Figure \ref{fig:w-surjection}:
    \begin{enumerate}[a)]
        \item first, we let $1 \in P_1$;
        \item since $w_1^* < w_2^*$, there is another peak-clique $P_2$ containing $2$ and, since $w_2 = 2$, it follows that $2 \in P_1$;
        \item since $w_3^* = 2^*$, it follows that $3 \in P_2, P_1$. Moreover, since $w_3^*$ is a marked entry, it follows that there exists a peak-clique $P_3$ such that its smallest element is $4$, which implies that $P_1 = \{1,2,3\}$;
        \item since $w_4^* = 2$, it follows that $4 \in P_3, P_2$;
        \item since $w_4 < w_5 = 3$, there is another peak-clique $P_4$ containing $5$, and $6 \in P_4, P_3, P_2$. Moreover, since $w_5^*$ is a marked entry, there exists a peak-clique $P_5$ such that its smallest element is $6$, which implies that $P_2 = \{2,3,4,5\}$;
        \item since $w_6^* = 3$, it follows that $6 \in P_5, P_4, P_3$;
        \item since $w_7^* = 2$, it follows that $7 \in P_5, P_4$ and $P_3 = \{4,5,6,7\}$;
        \item since $w_8^* = 1$, it follows that $8 \in P_5$ and $P_4 = \{6,7,8\}$.
    \end{enumerate}
    Hence, $\Gamma_\avec$ is the unit interval graph whose set of peak-cliques is
    \[
        \mathcal{P}_\avec = \{ \{1,2,3\}, \{2,3,4,5\}, \{4,5,6\}, \{5,6,7\}, \{6,7,8\} \}.
    \]
    
    \begin{figure}[H]
        \centering
        \begin{tabular}{cccc}
            \subfloat[]{\begin{tikzpicture}[scale=1.0,x=1em,y=1em,baseline=(current bounding box.south)]
            \draw[step=1,very thin,gray!50] (0,0) grid (8,6);
            
            \foreach \x in {1,...,8} \node at (\x-0.5,5.5) {\footnotesize \x};
              
            \draw[line width=3pt, line cap=round] (0,0.5) -- (1,0.5);   
            
            \node at (0.5,-0.5) {\footnotesize 1};
        \end{tikzpicture}} &
        
        \subfloat[]{\begin{tikzpicture}[scale=1.0,x=1em,y=1em,baseline=(current bounding box.south)]
            \draw[step=1,very thin,gray!50] (0,0) grid (8,6);
            
            \foreach \x in {1,...,8} \node at (\x-0.5,5.5) {\footnotesize \x};
              
            \draw[line width=3pt, line cap=round] (0,0.5) -- (2,0.5);   
            \draw[line width=3pt, line cap=round] (1,1.5) -- (2,1.5);   
            
            \node at (0.5,-0.5) {\footnotesize 1};
            \node at (1.5,-0.5) {\footnotesize 2};
        \end{tikzpicture}} &
        
        \subfloat[]{\begin{tikzpicture}[scale=1.0,x=1em,y=1em,baseline=(current bounding box.south)]
            \draw[step=1,very thin,gray!50] (0,0) grid (8,6);
            
            \foreach \x in {1,...,8} \node at (\x-0.5,5.5) {\footnotesize \x};
              
            \draw[line width=3pt, line cap=round] (0,0.5) -- (3,0.5);   
            \draw[line width=3pt, line cap=round] (1,1.5) -- (3,1.5);   
            
            \node at (0.5,-0.5) {\footnotesize 1};
            \node at (1.5,-0.5) {\footnotesize 2};
            \node at (2.5,-0.5) {\footnotesize $2^*$};
        \end{tikzpicture}} &

        \subfloat[]{\begin{tikzpicture}[scale=1.0,x=1em,y=1em,baseline=(current bounding box.south)]
            \draw[step=1,very thin,gray!50] (0,0) grid (8,6);
            
            \foreach \x in {1,...,8} \node at (\x-0.5,5.5) {\footnotesize \x};
              
            \draw[line width=3pt, line cap=round] (0,0.5) -- (3,0.5);   
            \draw[line width=3pt, line cap=round] (1,1.5) -- (4,1.5);   
            \draw[line width=3pt, line cap=round] (3,2.5) -- (4,2.5);   
            
            \node at (0.5,-0.5) {\footnotesize 1};
            \node at (1.5,-0.5) {\footnotesize 2};
            \node at (2.5,-0.5) {\footnotesize $2^*$};
            \node at (3.5,-0.5) {\footnotesize 2};
            
            \draw[line width=0.8pt, blue, dashed] (3,2.5) -- (3,0.5);
        \end{tikzpicture}} \\
        
        & & \\
        
        \subfloat[]{\begin{tikzpicture}[scale=1.0,x=1em,y=1em,baseline=(current bounding box.south)]
            \draw[step=1,very thin,gray!50] (0,0) grid (8,6);
            
            \foreach \x in {1,...,8} \node at (\x-0.5,5.5) {\footnotesize \x};
              
            \draw[line width=3pt, line cap=round] (0,0.5) -- (3,0.5);   
            \draw[line width=3pt, line cap=round] (1,1.5) -- (5,1.5);   
            \draw[line width=3pt, line cap=round] (3,2.5) -- (5,2.5);   
            \draw[line width=3pt, line cap=round] (4,3.5) -- (5,3.5);   
            
            \node at (0.5,-0.5) {\footnotesize 1};
            \node at (1.5,-0.5) {\footnotesize 2};
            \node at (2.5,-0.5) {\footnotesize $2^*$};
            \node at (3.5,-0.5) {\footnotesize 2};
            \node at (4.5,-0.5) {\footnotesize $3^*$};
            
            \draw[line width=0.8pt, blue, dashed] (3,2.5) -- (3,0.5);
        \end{tikzpicture}} &

        \subfloat[]{\begin{tikzpicture}[scale=1.0,x=1em,y=1em,baseline=(current bounding box.south)]
            \draw[step=1,very thin,gray!50] (0,0) grid (8,6);
            
            \foreach \x in {1,...,8} \node at (\x-0.5,5.5) {\footnotesize \x};
              
            \draw[line width=3pt, line cap=round] (0,0.5) -- (3,0.5);   
            \draw[line width=3pt, line cap=round] (1,1.5) -- (5,1.5);   
            \draw[line width=3pt, line cap=round] (3,2.5) -- (6,2.5);   
            \draw[line width=3pt, line cap=round] (4,3.5) -- (6,3.5);   
            \draw[line width=3pt, line cap=round] (5,4.5) -- (6,4.5);   
            
            \node at (0.5,-0.5) {\footnotesize 1};
            \node at (1.5,-0.5) {\footnotesize 2};
            \node at (2.5,-0.5) {\footnotesize $2^*$};
            \node at (3.5,-0.5) {\footnotesize 2};
            \node at (4.5,-0.5) {\footnotesize $3^*$};
            \node at (5.5,-0.5) {\footnotesize 3};
            
            \draw[line width=0.8pt, blue, dashed] (3,2.5) -- (3,0.5);
            \draw[line width=0.8pt, blue, dashed] (5,4.5) -- (5,1.5);
        \end{tikzpicture}} &
        
        \subfloat[]{\begin{tikzpicture}[scale=1.0,x=1em,y=1em,baseline=(current bounding box.south)]
            \draw[step=1,very thin,gray!50] (0,0) grid (8,6);
            
            \foreach \x in {1,...,8} \node at (\x-0.5,5.5) {\footnotesize \x};
              
            \draw[line width=3pt, line cap=round] (0,0.5) -- (3,0.5);   
            \draw[line width=3pt, line cap=round] (1,1.5) -- (5,1.5);   
            \draw[line width=3pt, line cap=round] (3,2.5) -- (6,2.5);   
            \draw[line width=3pt, line cap=round] (4,3.5) -- (7,3.5);   
            \draw[line width=3pt, line cap=round] (5,4.5) -- (7,4.5);   
            
            \node at (0.5,-0.5) {\footnotesize 1};
            \node at (1.5,-0.5) {\footnotesize 2};
            \node at (2.5,-0.5) {\footnotesize $2^*$};
            \node at (3.5,-0.5) {\footnotesize 2};
            \node at (4.5,-0.5) {\footnotesize $3^*$};
            \node at (5.5,-0.5) {\footnotesize 3};
            \node at (6.5,-0.5) {\footnotesize 2};
            
            \draw[line width=0.8pt, blue, dashed] (3,2.5) -- (3,0.5);
            \draw[line width=0.8pt, blue, dashed] (5,4.5) -- (5,1.5);
        \end{tikzpicture}} &
        
        \subfloat[]{\begin{tikzpicture}[scale=1.0,x=1em,y=1em,baseline=(current bounding box.south)]
            \draw[step=1,very thin,gray!50] (0,0) grid (8,6);
            
            \foreach \x in {1,...,8} \node at (\x-0.5,5.5) {\footnotesize \x};
              
            \draw[line width=3pt, line cap=round] (0,0.5) -- (3,0.5);   
            \draw[line width=3pt, line cap=round] (1,1.5) -- (5,1.5);   
            \draw[line width=3pt, line cap=round] (3,2.5) -- (6,2.5);   
            \draw[line width=3pt, line cap=round] (4,3.5) -- (7,3.5);   
            \draw[line width=3pt, line cap=round] (5,4.5) -- (8,4.5);   
            
            \node at (0.5,-0.5) {\footnotesize 1};
            \node at (1.5,-0.5) {\footnotesize 2};
            \node at (2.5,-0.5) {\footnotesize $2^*$};
            \node at (3.5,-0.5) {\footnotesize 2};
            \node at (4.5,-0.5) {\footnotesize $3^*$};
            \node at (5.5,-0.5) {\footnotesize 3};
            \node at (6.5,-0.5) {\footnotesize 2};
            \node at (7.5,-0.5) {\footnotesize 1};
            
            \draw[line width=0.8pt, blue, dashed] (3,2.5) -- (3,0.5);
            \draw[line width=0.8pt, blue, dashed] (5,4.5) -- (5,1.5);
        \end{tikzpicture}}\\
        \end{tabular}
        \caption{A construction step-by-step of a unit interval graph $\Gamma_\avec \in \DP(8)$ for which $\varphi_8 (\avec) = (1,2,2^*,2,3^*,3,2,1)$.}
        \label{fig:w-surjection}
    \end{figure}
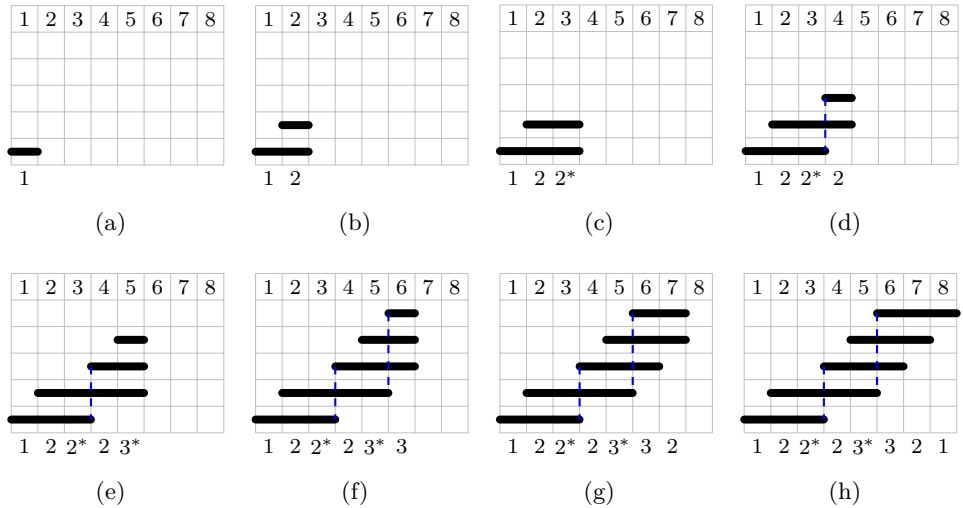

    Now we show that $\varphi_n$ is also an injection. We will show that $|\DP(n)| = |\motzkin^*(n)|$, from which the injectivity of $\varphi_n$ will follow. In fact, recall that $\displaystyle |\DP(n)| = C_n = \frac{1}{n+1} \binom{2n}{n}$, the $n$th Catalan number \cite[\oeis{A000108}]{OEIS}. Touchard's identity \cite{Touchard1928} states that
    \[
        C_{n} = \sum_{k = 0}^{\lfloor (n-1) / 2 \rfloor} \binom{n-1}{2k} 2^{n-1-2k} C_k.
    \]
    Now, observe that any bicolored Motzkin path $M^*$ of length $n-1$ is obtained as follows: among the $n-1$ steps of $M^*$, choose $k$ of them to be $U$ steps and $k$ of them to be $D$ steps, and distribute them in such a way that the path obtained by deleting the $H$ and the $H^*$ steps of $M^*$ is a Dyck path, which can be done in $\binom{n-1}{2k} C_k$ different ways. Now, each one of the $H$ steps can be marked or not, which is done in $2^{n-1-2k}$ different ways. Hence,
    \begin{equation}
    \label{eq:touchard-motzkin}
        |\motzkin^*(n)| = \sum_{k = 0}^{\lfloor (n-1) / 2 \rfloor} \binom{n-1}{2k} 2^{n-1-2k} C_k,
    \end{equation}
    from which $|\DP(n)| = |\motzkin^*(n)|$ follows.

    Now, we prove that $\varphi_n$ satisfies (i) and (ii). In fact, (i) is automatically satisfied by definition of $\varphi_n$. To prove (ii), observe that if we traverse the entries of $\wvec^*_\avec$ from the start, we see that every ascent of $\wvec_\avec$ corresponds to a new peak-clique. Moreover, every marked entry gives rise to the end of a peak-clique and the start of a new peak-clique, so the (ii) holds.
\end{proof}

\begin{definition}
    Given $\avec \in \DP(n)$, we call the unique $\wvec_\avec^* \in \mathcal{MP}^*(n)$ the \defin{bicolored Motzkin sequence associated to $\avec$}.
\end{definition}

From Theorem \ref{th:bij-dyck-decorated-motzkin}, it follows that, for each $\wvec \in \motzkin(n)$, $n \geq 0$, there might exist many $\avec \in \mathcal{DP}(n)$ for which $\pnv(\avec) = \wvec$. For example, for $\wvec = (1,1,1)$, the sequences $\avec_1 = (0,0,0)$ and $\avec_2 = (1,2,3)$ are such that $\pnv(\avec_1) = \pnv(\avec_2) = \wvec$. In fact, the set of peak-cliques of $\avec_1$ is $\mathcal{P}_{\avec_1} = \{ \{1\}, \{2\}, \{3\} \}$, and the set of peak-cliques of $\avec_2$ is $\mathcal{P}_{\avec_2} = \{ \{1,2,3\} \}$, see Figure \ref{fig:example-motivation}.
\begin{figure}[H]
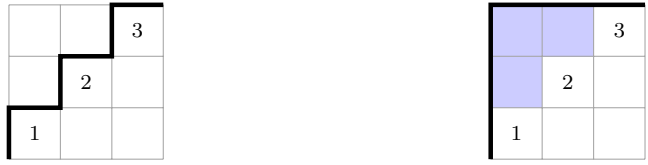

    \begin{align*}
        \dyckDiagram[0.04\textwidth]{0,0,0}
        &&&
        \dyckDiagram[0.04\textwidth]{0,1,2}
    \end{align*}
    \caption{The area sequences of $\avec_1 = (0,0,0)$, to the left, and $\avec_2 = (0,1,2)$, to the right.}
    \label{fig:example-motivation}
\end{figure}
This fact motivates the following corollary:

\begin{corollary}
\label{cor:peak-and-pnest}
    A vector $\wvec=(w_1,w_2,\dotsc,w_n)$ is a peak-nesting vector if and only if $\wvec \in \motzkin(n)$. Moreover,
    \[
        \sum_{\substack{\avec \in \DP(n) \\ \pnv(\avec)=\wvec}} t^{\peak(\avec)} = t^{1+\asc(\wvec)}(1+t)^{\plateau(\wvec)} = t^{1+\asc(\wvec)}(1+t)^{n-1-2\asc(\wvec)}.
    \]
\end{corollary}
\begin{proof}
    The first part follows from Theorem \ref{th:bij-dyck-decorated-motzkin}(i). The second part follows from Theorem \ref{th:bij-dyck-decorated-motzkin}(ii).
\end{proof}

The \defin{height of a bicolored Motzkin path $M^*$ at position $k$} is the number $h_k(M^*)$ given by the difference between the number of $U$ steps and $D$ steps in its first $k$ steps. The \defin{height vector} of $M^* \in \motzkin(n)^*$ is the vector $(h_0 (M^*), h_1(M^*), \ldots, h_{n-1}(M^*))$, where $h_0 (M^*) \coloneqq 0$. The \defin{height} of $M^*$ is the number given by $\height(M^*) \coloneqq \max \{h_0 (M^*), h_1(M^*), \ldots, h_{n-1}(M^*)\}$.

\begin{lemma}
\label{lem:pnest=height}
    Let $\avec \in \mathcal{DP}(n)$ and let $\wvec = \pnv(\avec)$. Then $\pnest(\avec) = \height(M_\wvec) + 1$, where $M_\wvec$ is obtained from $\wvec$ via the bijection \eqref{eq:bij-motzkin-path-seq}.
\end{lemma}
\begin{proof}
    First, observe that, if $\wvec = (w_1, \ldots, w_n)$ is a Motzkin sequence, then it defines a Motzkin path $M_\wvec$ with height vector $(w_1 - 1, w_2 - 1, w_3 - 1, \ldots, w_n -1)$. Hence, the height of $M_\wvec$ is given by
    \begin{equation}
    \label{eq:height-motzkin-seq}
        \height(M_\wvec) \coloneqq \max_{i \in [n]} \{w_i - 1\}.
    \end{equation}
    By Theorem \ref{th:bij-dyck-decorated-motzkin}(i), $\pnv(\avec) = \wvec$, from which follows that
    \begin{equation}
    \label{eq:pnv-motzkin-seq}
        \pnest(\avec) = \max_{i \in [n]} \{w_i\}.
    \end{equation}
    The claim now follows from a combination of \eqref{eq:height-motzkin-seq} and \eqref{eq:pnv-motzkin-seq}.
\end{proof}

We prove now the main result of this section, where we express the peak-nesting polynomial in terms of heights of uncolored Motzkin paths:

\begin{corollary}
\label{cor:refinement}
    The peak-nesting generating polynomial over all Dyck paths is given by the following refinement of Touchard's identity:
    \[
        \Pnest_{n}(t) = \sum_{i=0}^{\lfloor \frac{n-1}{2} \rfloor} 2^{n-1-2i} \sum_{\substack{M \in \motzkin (n) \\ \mathrm{u}(M) = i}}  t^{\height(M)+1} = \sum_{i=0}^{\lfloor \frac{n-1}{2} \rfloor} 2^{n-1-2i} \binom{n-1}{2i} \sum_{D \in \DP(i)} t^{\height(D) + 1}.
    \]
\end{corollary}

\begin{proof}
    The first equality follows from Theorem \ref{th:bij-dyck-decorated-motzkin}(i) and Lemma \ref{lem:pnest=height}. The second equality follows from the first equality and \eqref{eq:touchard-motzkin}.
\end{proof}

For each $1 \leq i \leq n$, let $H(n,i)$ be the number of Dyck paths of semilength $n$ with height $i$, see \cite[\oeis{A080936}]{OEIS}. Set $H(0,0) = 1$ and $H(n,i) = 0$ if $i > n$.

\begin{corollary}
\label{ref:pnest-formulas}
    For each $n \in \mathbb{N}$,
    \[
        \Pnest_n(t) = \sum_{i = 1}^{ \lceil n/2 \rceil} \left[ H(n,2i - 1) + H(n, 2i) \right] t^i \quad \mbox{and} \quad \Pnest_n^c(t) = \sum_{i = 1}^{ \lceil n/2 \rceil} \left[ H(n-1,2i - 2) + H(n-1, 2i - 1) \right] t^i .
    \]
\end{corollary}

\begin{proof}
    Fix $n \in \mathbb{N}$. We first prove the formula for $\Pnest_n(t)$. By Theorem \ref{th:bij-dyck-decorated-motzkin} and Lemma \ref{lem:pnest=height},
    \[
        \Pnest_n(t) = \sum_{M^* \in \mathcal{MP}^*(n)} t^{\height(M^*) + 1},
    \]
    so we may write
    \[
        \Pnest_n(t) = \sum_{i=1}^{ \lceil n/2 \rceil} h_{n,i-1} t^i,
    \]
    where $h_{n,i-1}$ is the number of bicolored Motzkin paths of semilength $n-1$ with height $i-1$. We will prove that
    \begin{equation}
    \label{eq:height-Motzkin-to-Dyck}
        h_{n,i-1} = H(n,2i - 1) + H(n, 2i)
    \end{equation}
    for each $i = 1, \ldots, \lceil \mbox{\Large $\nicefrac{n}{2}$} \rceil$.
    
    We recall the classical Delest–Viennot bijection between bicolored Motzkin paths of length $n-1$ and Dyck paths of semilength $n$ \cite[Equation 15]{DelestViennot1984}: let $\eta \colon \{U, D, H, H^*\} \to \{UU, DD, UD, DU\}$ be such that $\eta(U) = UU$, $\eta(D) = DD$, $\eta(H) = UD$ and $\eta(H^*) = DU$, and let $M^* = S_1 \cdots S_{n-1}$ be a bicolored Motzkin path of length $n-1$. Then $\mapsto D_{M^*} = U \eta(S_1) \cdots \eta(S_{n-1}) D \in \DP(n)$, and the map
    \[
        M^* \mapsto D_{M^*}
    \]
    is a bijection between $\mathcal{MP}^*(n)$ and $\DP(n)$.

    Let $\height(M^*) = i-1$. By the bijection constructed above, $\height(D_{M^*}) \in \{2i-1, 2i\}$, from which \eqref{eq:height-Motzkin-to-Dyck} follows.

    To prove the formula for $\Pnest_n^c(t)$, observe that a unit interval graph is connected if and only if its associated bicolored Motzkin sequence has no entry equal to $1^*$ (or, equivalently, if and only if its associated bicolored Motzkin path has no $H^*$ step at height $0$). Let $\motzkin^*_\circ (n)$ be the set of such bicolored Motzkin paths. Then
    \[
        \Pnest_n^c(t) = \sum_{M^* \in \motzkin^*_\circ(n)} t^{\height(M^*) + 1}.
    \]
    Let $D_{M^*} = \eta(S_1) \cdots \eta(S_{n-1})$. Since the map
    \[
        M^* \mapsto D_{M^*}
    \]
    is a bijection between $\motzkin^*_\circ(n)$ and $\DP(n-1)$, by repeating the argument used for $\Pnest_n(t)$, we verify that if $\height(M^*) = i-1$, $M^* \in \mathcal{MP}_\circ^*(n)$, then $\height(D_{M^*}) \in \{2i-2, 2i-1\}$, which finishes the proof.
\end{proof}

\begin{table}[H]
\centering
\begin{tabular}{@{}r r r@{}}
\toprule
$n$ & $\Pnest_n(t)$ & $\Pnest_n^c(t)$ \\
\midrule
0  & $1$ & $0$ \\
1  & $t$ & $t$ \\
2  & $2t$ & $t$ \\
3  & $4t+t^{2}$ & $t+t^{2}$ \\
4  & $8t+6t^{2}$ & $t+4t^{2}$ \\
5  & $16t+25t^{2}+t^{3}$ & $t+12t^{2}+t^{3}$ \\
6  & $32t+90t^{2}+10t^{3}$ & $t+33t^{2}+8t^{3}$ \\
7  & $64t+301t^{2}+63t^{3}+t^{4}$ & $t+88t^{2}+42t^{3}+t^{4}$ \\
8  & $128t+966t^{2}+322t^{3}+14t^{4}$ & $t+232t^{2}+184t^{3}+12t^{4}$ \\
9  & $256t+3025t^{2}+1463t^{3}+117t^{4}+t^{5}$ & $t+609t^{2}+731t^{3}+88t^{4}+t^{5}$ \\
10 & $512t+9330t^{2}+6174t^{3}+762t^{4}+18t^{5}$ & $t+1596t^{2}+2737t^{3}+512t^{4}+16t^{5}$ \\
\bottomrule
\end{tabular}
\caption{The distribution of peak-nestings over unit interval graphs and connected unit interval graphs with $n$ vertices, $0 \leq n \leq 10$.}
\label{tab:peak-nesting-polynomial}
\end{table}

Observe that the polynomials $\Pnest_n(t)$ and $\Pnest_n^c(t)$ are not always real-rooted. For example,
\[
    \Pnest_{17}(t) = 65536t + 21457825t^2 + 63739103t^3 + 35640450t^4 + 7907108t^5 + 799799t^6 + 34475t^7 + 493t^8 + t^9
\]
and
\[
    \Pnest^c_{16}(t) = t + 514228t^2 + 4543140t^3 + 3646720t^4 + 899000t^5 + 88536t^6 + 3192t^7 + 28t^8
\]
are not real-rooted polynomials. However, using Mathematica \cite{Mathematica}, we verified that the coefficients of $\Pnest_n(t)$ and $\Pnest_n^c(t)$ form ultra log-concave sequences for $n \leq 200$, from which we conjecture the following:

\begin{conjecture}
\label{conj:pnest-ulc}
    The coefficients of $\Pnest_n(t)$ and $\Pnest_n^c(t)$ form ultra log-concave sequences for all $n \in \mathbb{N}$.
\end{conjecture}

However, the coefficientwise formula obtained for $\Pnest_n(t)$ and $\Pnest_n^c(t)$ leads to simple formulas for some of its coefficients, from which we deduce the log-concavity of the upper-half coefficients of $\Pnest_n(t)$ and $\Pnest_n^c(t)$:

\begin{corollary}
\label{prop:pnest-1-and-2}
    For all $n \geq 1$, we have
    \begin{enumerate}[(i)]
        \item $[t] \Pnest_n(t) = 2^{n-1}$ and $[t] \Pnest_n^c(t) = 1$;
        \item $[t^2] \Pnest_n(t) = S(n,3)$, the Stirling numbers of the second kind \cite[\oeis{A000392}]{OEIS}, and $[t^2] \Pnest_n^c(t) = F_{2n-3} - 1$, where $F_k$ is the $k$th Fibonacci number \cite[\oeis{A000045}]{OEIS};
        \item $[t^i] \Pnest_n(t) = \dfrac{4(2i+1)[8i(i+1)-3(n+1)](2n+1)!}{(n-2i+1)! (n+2i+3)!}$  if $i \geq \lceil (n+2)/4 \rceil$. In particular, for $n \geq 1$,
        \[
            \left[ t^{\lceil \frac{n}{2} \rceil} \right] \Pnest_n(t) = 
            \begin{cases}
                1 &\mbox{ if } n \mbox{ is odd}, \\
                2(n-1) &\mbox{ if } n \mbox{ is even};
            \end{cases}
        \]
        \item $[t^i] \Pnest_n^c(t) = \dfrac{8i [8i^2 (2n-1)- (6n^2 + n -2)](2n-2)!}{(n-2i+1)! (n+2i+1)!}$ if $i > \lfloor (n+2)/4 \rfloor$. In particular, for $n \geq 1$,
        \[
            \left[ t^{\lceil \frac{n}{2} \rceil} \right] \Pnest_n^c(t) = 
            \begin{cases}
                1 &\mbox{ if } n \mbox{ is odd}, \\
                2(n-2) &\mbox{ if } n \mbox{ is even}.
            \end{cases}
        \]        
    \end{enumerate}
    Hence,
    \[
        \left( [t^i] \Pnest_n(t) \right)^2 \geq [t^{i-1}] \Pnest_n(t) \cdot [t^{i+1}] \Pnest_n(t) \quad \mbox{if } \lceil (n+2)/4 \rceil < i \leq \lceil n/2 \rceil
    \]
    and
    \[
        \left( [t^i] \Pnest_n^c(t) \right)^2 \geq [t^{i-1}] \Pnest_n^c(t) \cdot [t^{i+1}] \Pnest_n^c(t) \quad \mbox{if } \lfloor (n+2)/4 \rfloor + 1 < i \leq \lceil n/2 \rceil.
    \]
\end{corollary}

\begin{proof}
    \leavevmode
    \begin{enumerate}[(i)]
        \item It is clear that there exists only one Dyck path of semilength $n-1$ with height $1$, which is $(UD) \cdots (UD)$, from which follows that $[t] \Pnest_n^c(t) = 1$. To prove that $[t] \Pnest_n(t) = 2^{n-1}$ for all $n \geq 1$, we use induction on $n$: for $n = 1$, there exists only one Dyck path in $\DP(1)$, $UD$, and its height is $1$, so $[t] \Pnest_1(t) = 1$. Assume that there exist $2^{n-2}$ Dyck paths on $\DP(n-1)$ with height at most $2$. We construct two different Dyck paths on $\DP(n)$ with height at most $2$ for a Dyck path $P \in \DP(n-1)$ as follows: either we add $UD$ after $P$, or we add $UD$ immediately to the left of the last step of $P$, which is a $D$ step. This operation clearly gives all the Dyck paths of semilength $n$ with hight at most $2$, so there are $2 \cdot 2^{n-2}$ many of them;
        \item By \cite[\oeis{A124302}]{OEIS},
        \[
            H(n,1) + H(n,2) + H(n,3) + H(n,4) = \frac{1+3^{n-1}}{2}.
        \]
        Hence,
        \[
            [t^2] \Pnest_n(t) = \frac{1+3^{n-1}}{2} - [t] \Pnest_n(t) = \frac{1+3^{n-1}}{2} - 2^{n-1} = S(n,3),
        \]
        where the last equality follows from \cite[\oeis{A000392}]{OEIS}. On the other hand, $H(n-1, 1) + H(n-1, 2) + H(n-1, 3) = F_{2n-3}$ by \cite[\oeis{A001519}]{OEIS}, so $[t^2] \Pnest_n^c(t) = F_{2n-3} - 1$;
        \item By \cite[\oeis{A080936}]{OEIS}, if $k \geq \lfloor (n+1)/2 \rfloor$, then
        \[
            H(n,k) = \frac{2(2k+3)(2k^2+6k+1-3n)(2n)!}{(n-k)! (n+k+3)!},
        \]
        from which the formulas for $[t^i] \Pnest_n(t)$ and $[t^i] \Pnest_n^c(t)$ follow.
    \end{enumerate}
    To prove the log-concavity of the upper half coefficients, we observe that the sequences
    \[
        \{ 2i+1 \}_i, \quad \{8i(i+1)-3(n+1)\}_i \quad \mbox{and} \quad \left\{ \frac{1}{(n-2i+1)! (n+2i+3)!} \right\}_i
    \]
    are log-concave if $\lceil (n+2)/4 \rceil < i \leq \lceil n/2 \rceil$, just like the sequences
    \[
        \{i\}_i, \quad \{ 8i^2 (2n-1)- (6n^2 + n -2) \}_i \quad \mbox{and} \quad \left\{ \frac{1}{(n-2i+1)! (n+2i+1)!} \right\}_i .
    \]
    $\lfloor (n+2)/4 \rfloor +1 < i \leq \lceil n/2 \rceil$. Since the Hadamard product of log-concave sequences is log-concave (see, e.g., \cite{Branden2015}), the claim follows.
\end{proof}

We finish this section with the following proposition, where we find interesting recursions for the generating functions of unit interval graphs with bounded peak-nesting numbers:

\begin{proposition}
    Let $C_k(z)$ and $D_k(z)$ be the generating functions for the number of connected unit interval graphs and the number of unit interval graphs on $[n]$ with peak-nesting at most $k$, respectively. By definition, we let
    \[
        C_1(z) = \sum_{n \geq 0} z^{n+1} = \frac{z}{1-z} \quad \mbox{and} \quad D_1(z) = 1 + \sum_{n \geq 1} 2^{n-1} z^n = \frac{1}{1-C_1(z)}.
    \]
    In general, we have the recursions
    \begin{numcases}{}
        C_k(z) = \frac{z}{1-z D_{k-1}(z)} &\mbox{if } $k \geq 2$, \label{eq:Ck-recursion}\\
        D_k(z) = \frac{1}{1-C_{k}(z)} &\mbox{if } $k \geq 1$.  \label{eq:Dk-recursion}
    \end{numcases}
    Taking the limit $k \to \infty$, we recover the continued fraction identity \cite[Prop. 5]{Flajolet2006} for Catalan numbers:
    \begin{equation}
    \label{eq:continued-fraction}
        \sum_{n \geq 0} \frac{1}{n+1}\binom{2n}{n} z^n  = \dfrac{1}{ 1 - \dfrac{z}{1 - \dfrac{z}{1 -  \dfrac{z}{\ddots}}}}. 
    \end{equation}
\end{proposition}

\begin{proof}
    First, we prove \eqref{eq:Dk-recursion}. Since any non-empty unit interval graph with peak-nesting at most $k$ is an ordered disjoint union of connected unit interval graphs with peak-nesting number at most $k$, $[C_k(z)]^n$ is the generating function for the set of unit interval graphs with peak-nesting at most $k$ and $n$ connected components. Hence,
    \[
        D_k(z) = 1 + C_k(z) + [C_k(z)]^2 + [C_k(z)]^3 + \dotsb = \frac{1}{1 - C_k(z)},
    \]
    where the constant term corresponds to the empty graph, so \eqref{eq:Dk-recursion} holds.

    Now, we prove \eqref{eq:Ck-recursion}. First, observe that a unit interval graph $\Gamma_\avec$ is connected if and only if no entry of its associated bicolored Motzkin sequence is $1^*$. We use unit interval graphs with peak-nesting at most $k-1$ to construct connected unit interval graphs on $[n]$ with peak-nesting at most $k$ as follows: let $\Gamma_{\avec_1}, \ldots, \Gamma_{\avec_m}$ be a collection of unit interval graphs whose associated bicolored Motzkin sequences $\wvec_{\avec_1}^*, \ldots, \wvec_{\avec_m}^*$ are such that $\pnest(\avec_i) \leq k-1$ for all $i \in [m]$.

    For each $i \in [m]$, let $\tilde{\wvec}_{\avec_i}$ be the vector obtained by appending $1$ on $\wvec_{\avec_i}$, and by adding $1$ to each entry of $\wvec_{\avec_i}$. Moreover, let $\tilde{\wvec}_{\avec_i}^*$ be the bicolored vector obtained from $\tilde{\wvec}_{\avec_i}$ by adding $*$'s at the correspondent entries of $\wvec_{\avec_i}^*$ that have  $*$'s. For example, if $\wvec_{\avec_i}^* = (1,2^*,2,3^*,3,2,1^*,1)$, then $\wvec_{\avec_i} = (1,2,2,3,3,2,1,1)$, $\tilde{\wvec}_{\avec_i} = (2,3,3,4,4,3,2,2,1)$ and $\tilde{\wvec}_{\avec_i}^* = (2,3^*,3,4^*,4,3,2^*,2,1)$. Now, concatenate $\tilde{\wvec}_{\avec_1}^*, \ldots, \tilde{\wvec}_{\avec_m}^*$ and prepend a $1$ in the sequence obtained. By definition, this is a bicolored Motzkin sequence with peak-nesting at most $k$ where no entry equals to $1^*$, hence it is the bicolored Motzkin sequence associated to a connected unit interval graph with peak-nesting at most $k$. Clearly, any such graph can be constructed in this way, from which follows that
    \[
        C_k(z) = z \cdot \left( 1 + z D_{k-1}(z) + [z D_{k-1}(z)]^2 + [z D_{k-1}(z)]^3 + \dotsb  \right) = \frac{z}{1 - z D_{k-1}(z)}
    \]
    and verifies \eqref{eq:Ck-recursion}.

    Finally, we prove \eqref{eq:continued-fraction}. Let
    \[
        \mathcal{C}(z) = \frac{1 - \sqrt{1 - 4z}}{2z} = \sum_{n \geq 0} \frac{1}{n+1}\binom{2n}{n} z^n = \dfrac{1}{ 1 - \dfrac{z}{1 - \dfrac{z}{1 -  \dfrac{z}{\ddots}}}} .
    \]
    We will prove that
    \[
        \lim_{k \to \infty} [ D_k(z) ] = \mathcal{C}(z).
    \]
    First, set
    \[
        C(z) \coloneqq \lim_{k \to \infty} [ C_k(z) ] \quad \mbox{and} \quad D(z) \coloneqq \lim_{k \to \infty} [ D_k(z) ] .
    \]
    By \eqref{eq:Ck-recursion} and \eqref{eq:Dk-recursion},
    \begin{equation}
    \label{eq:lim-k}
        C(z) = \frac{z}{1 - z D(z)} \quad \mbox{and} \quad D(z) = \frac{1}{1 - C(z)} ,
    \end{equation}
    which implies
    \[
        C(z) = \frac{z}{1 - \frac{z}{1 - C(z)}} .
    \]
    Hence,
    \begin{equation}
    \label{eq:pol-C}
        C(z) = z + [C(z)]^2 .
    \end{equation}
    Observe that, by \eqref{eq:lim-k}, $C(0) = 0$. Thus, \eqref{eq:pol-C} implies
    \[
        C(z) = \frac{1 - \sqrt{1 - 4z}}{2} = z \cdot \mathcal{C}(z),
    \]
    from which follows that
    \[
        D(z) = \frac{1}{1 - z \cdot \mathcal{C}(z)} = \dfrac{1}{ 1 - \dfrac{z}{1 - \dfrac{z}{1 -  \dfrac{z}{\ddots}}}} .
    \]
\end{proof}

\subsection{The peak polynomial of unit interval graphs}
\label{sec:peak-all}
Recall that, given a unit interval graph $\Gamma_\avec$, its peak-number $\peak(\avec)$ is defined as the number of peak-cliques of $\Gamma_\avec$. The \defin{peak polynomial on $[n]$} is defined as
\[
    \Peak_n(t) \coloneqq \sum_{\avec \in \mathcal{DP}(n)} t^{\peak(\avec)}.
\]
It is known that $\Peak_n(t) = N_n(t)$, the $n$th Narayana polynomial, whose closed formula is
\[
    N_n(t) = \sum_{i=1}^n \frac{1}{n} \binom{n}{i} \binom{n}{i-1} t^i = t \cdot \leftidx{_2}F_1 ( 1-n, -n; 2; t )
\]
see \cite[\oeis{A001263}]{OEIS}, where $\leftidx{_2}F_1 (a, b, c; t)$ is the ordinary hypergeometric function.

The Narayana polynomials are known to be real-rooted, and $N_n(t)/t$ are gamma-nonnegative, with decomposition
\[
    \frac{N_n(t)}{t} = \sum_{0 \leq 2i \leq n} M(n,i) \cdot t^i (1+t)^{n-1-2i},
\]
where $M(n,i)$ is the number of Motzkin paths of length $n-1$ with $i$ $U$ steps by Corollary \ref{cor:peak-and-pnest}, see \cite[\oeis{A055151}]{OEIS}.

We also define
\[
    \Peak_n^c(t) \coloneqq \sum_{\substack{\avec \in \DP(n) \\ \Gamma_\avec \mbox{ \scriptsize is connected}}} t^{\peak(\avec)},
\]
and since connected unit interval graphs on $[n]$ are in bijection with Dyck paths of the form $UPD$, where $P \in \DP(n-1)$, it follows that $\Peak_n^c(t) = N_{n-1}(t)$ for $n \geq 2$, from which their real-rootedness and gamma-nonnegativity follow.

\section{Abelian unit interval graphs}
\label{sec:abelian}

In this section, we study Abelian area sequences, and the peak-nesting and peak polynomials associated to it. An area sequence $\avec$ of size $n$ (and its associated unit interval graph $\Gamma_\avec$) is called \defin{Abelian} if every vertex $i$ of $\Gamma_\avec$, $i \neq 1, n$, is adjacent to at least one of the vertices $1$ or $n$. Its associated Dyck path is called an \defin{Abelian Dyck path}. For example, Figure \ref{fig:all-abelians4} shows all the Abelian area sequences of size $4$. We denote the set of Abelian area sequences of size $n$ (or Abelian Dyck paths of semilength $n$) by $\abelian{n}$.

\begin{figure}[H]
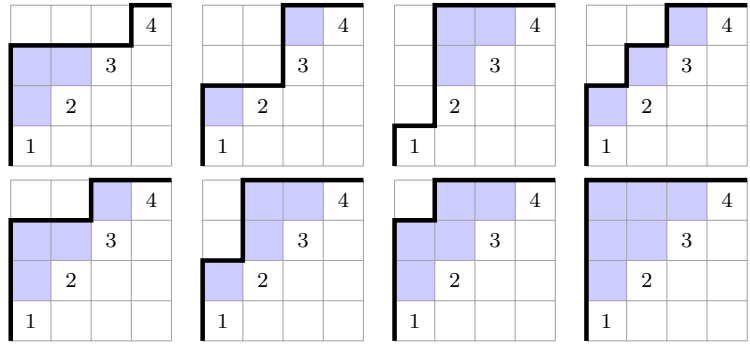

    \begin{tabular}{*{4}{c}}
        \dyckDiagram[0.03125\textwidth]{0,1,2,0} &
        \dyckDiagram[0.03125\textwidth]{0,1,0,1} &
        \dyckDiagram[0.03125\textwidth]{0,0,1,2} &
        \dyckDiagram[0.03125\textwidth]{0,1,1,1} \\
        \dyckDiagram[0.03125\textwidth]{0,1,2,1} &
        \dyckDiagram[0.03125\textwidth]{0,1,1,2} &
        \dyckDiagram[0.03125\textwidth]{0,1,2,2} &
        \dyckDiagram[0.03125\textwidth]{0,1,2,3} \\
    \end{tabular}
    \caption{All Abelian area sequences of size $4$.}
    \label{fig:all-abelians4}
\end{figure}

The following lemma follows directly from the definition of Abelian area sequences:

\begin{lemma}
\label{charac-abelian}
    An area sequence $\avec$ of size $n$ is Abelian if and only if there exists $k \in \{1, \ldots, n\}$ such that $\{1, \ldots, k\}$ and $\{k+1, \ldots, n\}$ are cliques of $\Gamma_\avec$.
\end{lemma}

\begin{proof}
    The sufficiency of the claim is clear, therefore we only need to prove its necessity. In fact, let
    \[
        k = \max \left\{ \{ i \in [n] \colon \{1,i\} \mbox{ is an edge of } \Gamma_\avec \} \cup \{1\} \right\} .
    \]
    By definition of $k$ and by definition of Abelian area sequences, the claim follows.
\end{proof}

By Lemma \ref{charac-abelian}, an Abelian Dyck path is a Dyck path such that, for some $k \in \{1, \ldots, n\}$, its first $k$ steps are $N$ steps, and its $(k+1)$th step and its last $n-k$ steps are $E$ steps. Hence there is an $NE$ lattice path from $(0,k)$ to $(k,n)$ whose first step is an $E$ step. Hence, an Abelian Dyck path can be uniquely represented by the pair $(k, L)$, where $k \in \{1, \ldots, n\}$ is as given in Lemma \ref{charac-abelian}, and $L$ is an $NE$ lattice path from $(0,k)$ to $(k,n)$ whose first step is an $E$ step.

\begin{proposition}
\label{prop-number-abelian}
    The number of Abelian unit interval graphs on $[n]$ is $2^{n-1}$. The number of connected Abelian unit interval graphs on $[n]$ is $2^{n-1} - (n-1)$.
\end{proposition}

\begin{proof}
    We start by proving the first part of the proposition. For each $k \in \{1, \ldots, k\}$, we will count the number of Abelian Dyck paths $(k, L)$ of size $n$. Observe that each Abelian Dyck path $(k, L)$ determines uniquely a weakly increasing sequence $0 = b_1 \leq b_2 \leq \cdots \leq b_k \leq n-k$, where $b_i$ is the number of boxes under the $i$th $E$ step of $L$ and above the horizontal line $y=k$, see Figure \ref{fig:Abelian}.

    \begin{figure}[H]
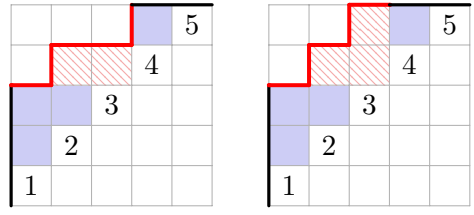

        \begin{tabular}{*{2}{c}}
            \AbelianDyckDiagram[0.03125\textwidth]{5}{3}{0,1,1} &
            \AbelianDyckDiagram[0.03125\textwidth]{5}{3}{0,1,2}
        \end{tabular}
        \caption{The Abelian unit interval graphs on $[5]$ given by $(3; 0, 1, 1)$, on the left, and $(3; 0, 1, 2)$, on the right.}
        \label{fig:Abelian}
    \end{figure}

    Let $x_1 \coloneqq 0$ and $x_i \coloneqq b_i - b_{i-1}$ for each $i = 2, \ldots, k$. Observe that
    \begin{equation*}
        b_k = x_2 + \ldots + x_k \leq n-k.
    \end{equation*}
    Hence, the number of weakly increasing sequences $0 = b_1 \leq b_2 \leq \cdots \leq b_k \leq n-k$ equals the number of integer solutions of the inequality above, which corresponds to the number of integer solutions of
    \begin{equation}
    \label{eq:number-abelian}
        x_2 + \ldots + x_k + x_{k+1} = n-k, \quad x_2, \ldots, x_{k+1} \geq 0,
    \end{equation}
    which is $\binom{n-1}{k-1}$. Summing over all $k \in [n]$, the claim follows.

    From the proof of the first part, it follows that each Abelian Dyck path of semilength $n$ can be uniquely represented by the tuple $(k; 0, b_2, \ldots, b_k)$, where $0 \leq b_2 \leq b_3 \leq \cdots \leq b_k \leq n-k$. So, for the second part, observe that the Abelian unit interval graph determined by $(k; 0, b_2, \ldots, b_k)$ is disconnected if and only if $k \neq n$ and $b_2 = \ldots b_k = 0$. Hence, there are $n-1$ disconnected Abelian unit interval graphs with $n$ vertices, from which the second part of the proposition follows.
\end{proof}

\subsection{The peak-nesting polynomial of Abelian unit interval graphs}

Set $\Pnest_{\mathcal{A}(0)}(t) \coloneqq 1$ and $\Pnest_{\mathcal{A}(0)}^c(t) \coloneqq 0$, and let
\[
    \Pnest_{\mathcal{A}(n)}(t) \coloneqq \sum_{\avec \in \abelian{n}} t^{\pnest(\avec)} \quad \mbox{and} \quad \Pnest_{\mathcal{A}(n)}^c(t) \coloneqq \sum_{\substack{\avec \in \abelian{n} \\ \Gamma_\avec \mbox{ \scriptsize is connected}}} t^{\pnest(\avec)},
\]
for $n \geq 1$, see Table \ref{tab:pnest-Abelian}. The following theorem gives a new interpretation of \cite[\oeis{A034867}]{OEIS}:

\begin{theorem}
\label{thm:pnest-pol-abelian}
    For $n \in \mathbb{N}$,
    \[
        \Pnest_{\mathcal{A}(n)}(t) = \sum_{i=1}^{\lceil n/2 \rceil} \binom{n}{2i - 1}t^i  = \frac{\sqrt{t}}{2} \left[ (1+\sqrt{t})^n - (1-\sqrt{t})^n \right] .
    \]
\end{theorem}

To prove Theorem~\ref{thm:pnest-pol-abelian}, we use the following lemma:
\begin{lemma}
\label{lemma:pnest(k+1)}
    Let $\avec = (k; 0, b_2, \ldots, b_k) \in \abelian{n}$. Then $\pnest(\avec) = \pnest(k)$.
\end{lemma}
\begin{proof}
    Let $p$ be the number of peak-edges of $\avec$, and let $b_{k+1} = n-k$. It is not hard to see that
    \[
        p = 1 + \# \{i \in [k] \colon b_i < b_{i+1} \},
    \]
    where $b_1 = 0$, see Figure \ref{fig:Abelian}. We will compute $\pnest(j)$ for each $j \in [n]$. In fact,
    \begin{equation}
    \label{eq:pnest-abelian}
        \pnest(j) = 1 + \# \{i \in [j-1] \colon b_i < b_{i+1} \}
    \end{equation}
    if $1 \leq j \leq k$, and $\pnest(j) \leq p-1$ if $k < j \leq n$. We now show that $\pnest(\avec) = \pnest(k)$. In fact, $\pnest(k) \geq \pnest(j)$ for $1 \leq j \leq k-1$ by \eqref{eq:pnest-abelian}. Moreover,
    \[
        \pnest(k) =
        \begin{cases}
            p \quad &\mbox{if } b_k = n-k, \\
            p-1 \quad &\mbox{if } b_k < n-k,
        \end{cases}
    \]
    see Figure \ref{fig:Abelian}, from which follows that $\pnest(k) \geq \pnest(j)$ for $k+1 \leq j \leq n$. Hence, $\pnest(\avec) = \pnest(k)$.
\end{proof}

A \defin{binary word} of length $n$ is a sequence $w = w_1 w_2 \cdots w_n$ where $w_i \in \{0,1\}$ for all $i \in [n]$ and each $w_i$ is called a \defin{letter}.
\begin{proof}[Proof of Theorem \ref{thm:pnest-pol-abelian}]
    We start by proving the first equality. First, consider $i=1$. By the proofs of Proposition \ref{prop-number-abelian} and Lemma \ref{lemma:pnest(k+1)}, for each $1 \leq k \leq n$ there exists only one $(k; 0, b_2, \ldots, b_k) \in \abelian{n}$ such that $\pnest((k; 0, b_2, \ldots, b_k)) = 1$, which is given by $(k; 0, 0, \ldots, 0)$. Hence, $[t]\Pnest_{\abelian{n}}(t) = n$.

    Now, let $i > 1$. By the proof of Lemma \ref{lemma:pnest(k+1)}, the number of Abelian unit interval graphs $(k, L)$ such that $\pnest((k,L)) = i$ equals the number of $NE$ lattice paths $L$ from $(0,k)$ to $(k,n)$ with $i-1$ peaks. We count the number of such paths. Given an Abelian Dyck path $(k, L)$, recall that $L$ is an $NE$ lattice path with $k$ $E$ steps and $n-k$ $N$ steps, where the first step of $L$ is an $E$ step. Hence, there is a bijection between the set of $NE$ lattice paths $L$ and the set of binary words of length $n$ whose first letter is $0$, where the $E$ steps are encoded by zeros and the $N$ steps are encoded by ones. For example, the lattice path $L = ENEEN$ is encoded by the binary word $B_L = 01001$. Therefore, the number of peaks in a lattice path $L$ corresponds to the number of ascents in the binary word $B_L$. Observe that $k$ corresponds to the number of zeros in the binary word $B_L$. Since $1 \leq k \leq n$ and the first step of $L$ is $E$, the number of area sequences $(k, L)$ such that $\pnest((k,L)) = i$ equals the number of binary words of length $n-1$ and with $i-1$ ascents, which is given by $\binom{n}{2(i-1)+1}$ \cite[\oeis{A034867}]{OEIS}.

    The second equality is easily verified by induction on $n$.
\end{proof}

\begin{corollary}
\label{pnest-pol-abelian-rr}
    The polynomials $\{ \Pnest_{\abelian{n}}(t) \}_{n \geq 1}$ satisfy the recursion
    \begin{equation}
    \label{eq1}
        \Pnest_{\abelian{n}}(t) =
        \begin{cases}
            t \quad &\mbox{if } n = 1, \\
            2t \quad &\mbox{if } n = 2, \\
            2 \cdot \Pnest_{\abelian{n-1}}(t) + (t-1) \cdot \Pnest_{\abelian{n-2}}(t) \quad &\mbox{if } n \geq 3 .
        \end{cases}
    \end{equation}
    Moreover, $\Pnest_{\abelian{n}}(t)$ is real-rooted for every $n \geq 0$, and $\{\Pnest_{\abelian{n}}(t)\}_{n\geq 0}$ is a generalized Sturm sequence.
\end{corollary}

\begin{proof}
    Equation \ref{eq1} is easily verified by induction on $n \geq 3$. We also use induction to prove that $\{\Pnest_{\abelian{n}}(t)\}_{n\geq 0}$ is a generalized Sturm sequence.
    
    Recall that $\Pnest_{\abelian{0}}(t) = 1$. By definition, $\Pnest_{\abelian{0}}(t) \interl \Pnest_{\abelian{1}}(t)$, so assume by induction that $\Pnest_{\abelian{n-2}}(t) \interl \Pnest_{\abelian{n-1}}(t)$ for some $n > 2$, and let us show that $\Pnest_{\abelian{n-1}}(t) \interl \Pnest_{\abelian{n}}(t)$. In fact, $2 \cdot \Pnest_{\abelian{n-1}}(t) \interl (t-1) \cdot \Pnest_{\abelian{n-2}}(t)$, which implies that $\Pnest_{\abelian{n-1}}(t) \interl \Pnest_{\abelian{n}}(t)$ by Lemma \ref{lem:interlacing-properties}(i). Hence, the claim follows.
\end{proof}

Different from $\Pnest_{\mathcal{A}(n)}(t)$, the polynomials $\Pnest_{\mathcal{A}(n)}^c(t)$ are not real-rooted in general. In fact, $\Pnest_{\mathcal{A}(n)}^c(t)$ is real-rooted for $0 \leq n \leq 16$, but
\[
    \Pnest_{\mathcal{A}(17)}^c(t) = t + 680t^{2} + 6188t^{3} + 19448t^{4} + 24310t^{5} + 12376t^{6} + 2380t^{7} + 136t^{8} + t^{9}
\]
is not real-rooted. However, the sequence of coefficients of $\Pnest_{\mathcal{A}(n)}^c(t)$ is ultra log-concave.

\begin{corollary}
\label{cor:ulc-pnest}
    For each $n \in \mathbb{N}$,
    \begin{equation}
    \label{eq2}
        \Pnest_{\mathcal{A}(n)}^c(t) = \Pnest_{\mathcal{A}(n)}(t) - (n-1)t.
    \end{equation}
    Moreover, the sequence of coefficients of $\Pnest_{\mathcal{A}(n)}^c(t)$ is an ultra log-concave sequence.
\end{corollary}

\begin{proof}
    Equation \ref{eq2} follows trivially from the proofs of Proposition \ref{prop-number-abelian} and Theorem \ref{thm:pnest-pol-abelian}. The ultra log-concavity of the coefficients of $\Pnest_{\mathcal{A}(n)}^c(t)$ follows from the ultra log-concavity of the coefficients of $\Pnest_{\mathcal{A}(n)}(t)$.
\end{proof}

\begin{table}[H]
\centering
\begin{tabular}{@{}r r r@{}}
\toprule
$n$ & $\Pnest_{\mathcal{A}(n)}(t)$ & $\Pnest_{\mathcal{A}(n)}^c(t)$ \\
\midrule
0  & $1$ & $0$ \\
1  & $t$ & $t$ \\
2  & $2t$ & $t$ \\
3  & $3t+t^{2}$ & $t+t^{2}$ \\
4  & $4t+4t^{2}$ & $t+4t^{2}$ \\
5  & $5t+10t^{2}+t^{3}$ & $t+10t^{2}+t^{3}$ \\
6  & $6t+20t^{2}+6t^{3}$ & $t+20t^{2}+6t^{3}$ \\
7  & $7t+35t^{2}+21t^{3}+t^{4}$ & $t+35t^{2}+21t^{3}+t^{4}$ \\
8  & $8t+56t^{2}+56t^{3}+8t^{4}$ & $t+56t^{2}+56t^{3}+8t^{4}$ \\
9  & $9t+84t^{2}+126t^{3}+36t^{4}+t^{5}$ & $t+84t^{2}+126t^{3}+36t^{4}+t^{5}$ \\
10 & $10t+120t^{2}+252t^{3}+120t^{4}+10t^{5}$ & $t+120t^{2}+252t^{3}+120t^{4}+10t^{5}$\\
\bottomrule
\end{tabular}
\caption{The distribution of peak-nestings over Abelian unit interval graphs and connected unit interval graphs with $n$ vertices, $0 \leq n \leq 10$. The coefficients of $\Pnest_{\mathcal{A}(n)}(t)$ are given by \cite[\oeis{A034867}]{OEIS}.}
\label{tab:pnest-Abelian}
\end{table}

By Theorem \ref{thm:pnest-pol-abelian}, one easily verifies that $\mbox{\Large $\nicefrac{\Pnest_{\abelian{2n}}(t)}{t}$}$ is symmetric for every $n \geq 1$, see Table \ref{tab:pnest-Abelian}. By Corollary \ref{pnest-pol-abelian-rr}, they are also real-rooted, from which follows that the gamma-polynomials associated to them, $\gamma_{\abelian{2n}}^{\pnest}(t)$, have nonnegative coefficients and are real-rooted by Lemmas \ref{lem:gamma1} and \ref{lem:gamma2}, respectively, see Table \ref{tab:gamma-pnest-Abelian}.

We finish this section by finding a closed formula for the coefficients of $\gamma_{\abelian{2n}}^{\pnest}(t)$, and by presenting a combinatorial interpretation for them. Given an $NE$ lattice path $L = S_1 \cdots S_m$ and $i \in [m-1]$, we say that $i$ is a \defin{turn} if $S_i \neq S_{i+1}$, and we denote the set of all turns of $L$ by $T(L)$. Recall that any Abelian Dyck path of semilength $n$ has a unique representation $(k, L)$, where $1 \leq k \leq n$ and $L$ is an $NE$ lattice path from $(0,k)$ to $(k,n)$ whose first step is an $E$ step.

\begin{proposition}
\label{th:gamma-Abelian}
    For each $n \in \mathbb{N}$,
    \begin{equation}
    \label{eq:pnest-gamma-Abelian}
        \gamma_{\abelian{2n}}^{\pnest}(t) = \sum_{i = 0}^{\lfloor \frac{n-1}{2} \rfloor} 2^{2i + 1} \binom{n}{2i + 1} t^i = \frac{(1 + 2\sqrt{t})^n - (1 - 2\sqrt{t})^n}{2 \sqrt{t}}.
    \end{equation}
    Moreover, $\displaystyle 2^{2i + 1} \binom{n}{2i + 1}$ is the number of Abelian Dyck paths $(k, L)$ of semilength $2n$ with peak-nesting $i+1$ for which there is no $j \in [n]$ such that $S_{2j-1} S_{2j} S_{2j+1} = ENE$ or $S_{2j-1} S_{2j} S_{2j+1} = NEN$ are subwords of $LN$, the augmented $NE$ lattice path obtained by appending an $N$ step to $L$.
\end{proposition}

\begin{proof}
    Let $n \in \mathbb{N}$. By Theorem \ref{thm:pnest-pol-abelian},
    \begin{align*}
        \frac{\Pnest_{\abelian{2n}}(t)}{t} &= \frac{(1+\sqrt{t})^{2n} - (1-\sqrt{t})^{2n}}{2 \sqrt{t}} \\
                                           &= \frac{\left( (1+t) + 2\sqrt{t} \right)^{n} - \left( (1+t) - 2\sqrt{t} \right)^{n}}{2 \sqrt{t}} \\
                                           &= \frac{1}{2\sqrt{t}} \left[ \sum_{i=0}^{n} \left( 2\sqrt{t} \right)^i \binom{n}{i} (1+t)^{n-i} - \sum_{i=0}^{n} \left( -2\sqrt{t} \right)^i \binom{n}{i} (1+t)^{n-i} \right] \\
                                           &= \sum_{i = 0}^{\lfloor \frac{n-1}{2} \rfloor} 2^{2i + 1} \binom{n}{2i + 1} t^i (1 + t)^{n-1-2i}.
    \end{align*}
    Hence, \eqref{eq:pnest-gamma-Abelian} holds.

    To prove the second part of the theorem, let
    \[
        [2n] = \bigsqcup_{j=1}^n P_j ,
    \]
    where $P_j = \{2j - 1, 2j\}$, $j = 1, \ldots, n$. Let $(k, L)$ be an Abelian Dyck path of semilength $2n$. Hence, $L = S_1 \cdots S_{2n}$ is an $NE$ lattice path with $2n$ steps where $S_1 = E$. We shall identify each subword $S_{2j-1} S_{2j}$ with the set $P_j$. Clearly, $|T(LN) \cap P_j| \in \{0, 1, 2\}$ for each $j \in [n]$. Observe that, if $|T(LN) \cap P_j| = 2$ for some $j \in [n]$, then $ENE$ or $NEN$ is a subword of $LN$. Hence, if there are $2i + 1$ indices $j$ for which $|T(LN) \cap P_j| = 1$, then there are at most $n - 1 - 2i$ indices $j$ for which $|T(LN) \cap P_j| = 2$. We may remove the sets $P_j$ for which $|T(L) \cap P_j| = 2$ as follows: if $S_{2j-1} S_{2j} S_{2j+1} = ENE$, replace $S_{2j}$ by $E$; if $S_{2j-1} S_{2j} S_{2j+1} = NEN$, replace $S_{2j}$ by $N$. Then, we obtain a new lattice path $\tilde{L}N$ with $2i + 1$ turns, each of them in a distinct $P_j$. Moreover, since the first step of $L$ is $E$, it follows that $\tilde{L}N$ has exactly $i$ subwords $NE$ and $i+1$ subwords $EN$, from which follows that $\pnest( (k, \tilde{L}) ) = i+1$.
    
    Observe that adding back a triple $ENE$ or $NEN$ to $\tilde{L}$ increases the peak-nesting by $1$, so the entire class of Abelian Dyck paths obtained by $(k, \tilde{L})$ by mapping $EEE \mapsto ENE$ and $NNN \mapsto NEN$ contributes
    \[
        t^{i+1} (1 + t)^{n - 1 - 2i}
    \]
    to the peak-nesting polynomial $\Pnest_{\abelian{2n}}(t)$. Hence, $[t^i] \gamma_{\abelian{2n}}^{\pnest}(t)$ equals the number of Abelian Dyck satisfying the all the properties stated in the proposition, which finishes the proof.
\end{proof}

\begin{remark}
\label{rmk:gamma-pnest-abelian}
    The polynomials $\{ \gamma^{\pnest}_{\abelian{2n}}(t) \}_{n \geq 1}$ satisfy the recursion
    \begin{equation*}
        \gamma^{\pnest}_{\abelian{2n}}(t) =
        \begin{cases}
            2 &\mbox{if } n = 1, \\
            4 &\mbox{if } n = 2, \\
            2 \cdot \gamma^{\pnest}_{\abelian{2n-2}}(t) + (4t - 1) \cdot \gamma^{\pnest}_{\abelian{2n-4}}(t) &\mbox{if } n \geq 3,
        \end{cases}
    \end{equation*}
    from which one can prove that $\{\gamma^{\pnest}_{\abelian{2n}}(t)\}_{n\geq 1}$ is a generalized Sturm sequence.
\end{remark}

\begin{table}[H]
\centering
\begin{tabular}{@{}r r@{}}
\toprule
$n$ & $\gamma^{\pnest}_{\abelian{2n}}(t)$\\
\midrule
1  & $2$ \\
2  & $4$ \\
3  & $6+8t$ \\
4  & $8+32t$ \\
5  & $10+80t+32t^{2}$ \\
6  & $12+160t+192t^{2}$ \\
7  & $14+280t+672t^{2}+128t^{3}$ \\
8  & $16+448t+1792t^{2}+1024t^{3}$ \\
9  & $18+672t+4032t^{2}+4608t^{3}+512t^{4}$ \\
10 & $20+960t+8064t^{2}+15360t^{3}+5120t^{4}$ \\
\bottomrule
\end{tabular}
\caption{The polynomials $\gamma^{\pnest}_{\abelian{2n}}(t)$, $1 \leq n \leq 10$.}
\label{tab:gamma-pnest-Abelian}
\end{table}

\subsection{The peak polynomial of Abelian unit interval graphs}
Set $\Peak_{\abelian{0}}(t) \coloneqq 1$ and $\Peak_{\abelian{0}}^c(t) \coloneqq 0$, and let
\[
    \Peak_{\abelian{n}}(t) \coloneqq \sum_{\avec \in \abelian{n}} t^{\peak(\avec)} \qquad \mbox{and} \qquad \Peak_{\abelian{n}}^c(t) \coloneqq \sum_{\substack{\avec \in \abelian{n} \\ \Gamma_\avec \mbox{ \scriptsize is connected}}} t^{\peak(\avec)}
\]
for $n \in \mathbb{N}$, see Table \ref{tab:peak-Abelian}.

A \defin{run} in a binary word $w$ is an uninterrupted sequence of ones (or zeros) that cannot be extended, that is, flanked on each side either by zero (or one) or by the start or the end of $w$. For example, $00$, $1$ and $11$ are runs in $0010011$, but $0$ is not.
\begin{theorem}
\label{th:useful}
    For each $n \in \mathbb{N}$,
    \begin{equation}
    \label{eq:peak-pol-abelian}
        \Peak_{\abelian{n}}(t) = \sum_{i=0}^{\lfloor \frac{n}{2} \rfloor} \binom{n}{2i} t^{i+1} = \frac{t}{2} \left[ (1+\sqrt{t})^n + (1-\sqrt{t})^n \right]
    \end{equation}
    and
    \begin{equation}
    \label{eq:peak-pol-connected-abelian}
        \Peak_{\abelian{n}}^c(t) = \Peak_{\abelian{n}}(t) - (n-1)t^2.
    \end{equation}
\end{theorem}

\begin{proof}
    Let $n \in \mathbb{N}$. We start by proving \eqref{eq:peak-pol-abelian}. First, we prove that $\peak(\avec) \leq 1 + \lfloor \frac{n}{2} \rfloor$ for all $\avec \in \abelian{n}$. In fact, let $\avec = (k, L)$. As in the proof of Theorem \ref{thm:pnest-pol-abelian} we encode $L$ by a binary word $B_L$ of length $n$ where the first letter is $0$. Observe that, by the proof of Lemma \ref{lemma:pnest(k+1)}, $\peak(\avec) =  1 + \peak(L)$. Observe that the number of peaks of $L$ corresponds to the number of runs of $1$s in $L$, which is maximized by $01010\cdots101$ if $n$ is even and $01010\cdots1010$ if $n$ is odd. Hence, $\peak(\avec) \leq 1 + \lfloor \frac{n}{2} \rfloor$. Now, we show that
    \begin{equation}
    \label{eq:peaks=runs}
        [t^{i+1}]\Peak_{\abelian{n}}(t) = \binom{n}{2i}
    \end{equation}
    for all $i = 0, 1, \ldots, \lfloor \frac{n}{2} \rfloor$. In fact, there is only one Abelian area sequence of length $n$ with only one peak, which is $\avec = (0, 1, 2, \ldots, n-1)$, from which follows that $[t]\peak_{\abelian{n}}(t) = 1$, and the number of Abelian area sequences with $i+1$ peaks, $i + 1 \geq 2$, equals the number of binary words of length $n$ starting at $0$ with $i$ runs of consecutive $1$s, which is given by $\binom{n}{2i}$ \cite[\oeis{A034839}]{OEIS}. The second equality of \eqref{eq:peak-pol-abelian} follows by induction on $n \geq 1$.
    
    Now, to verify Equation \ref{eq:peak-pol-connected-abelian}, observe that, by the proof of Proposition \ref{prop-number-abelian}, all the disconnected Abelian unit interval graphs on $[n]$ have two peaks. Since there are $(n-1)$ many of them, \eqref{eq:peak-pol-connected-abelian} follows.
\end{proof}

\begin{corollary}
\label{cor:peak-pol-albelian-rr}
    The polynomials $\{ \Peak_{\abelian{n}}(t) \}_{n \geq 1}$ satisfy the recursion
    \begin{equation}
    \label{eq:rec-peak-abelian}
        \Peak_{\abelian{n}}(t) =
        \begin{cases}
            t \quad &\mbox{if } n = 1, \\
            t + t^2 &\mbox{if } n = 2, \\
            2 \cdot \Peak_{\abelian{n-1}}(t) + (t-1) \cdot \Peak_{\abelian{n-2}}(t) \quad &\mbox{if } n \geq 3.
        \end{cases}
    \end{equation}
    Moreover, $\Peak_{\abelian{n}}$ is real-rooted for every $n \geq 0$, and $\{\Peak_{\abelian{n}}\}_{n\geq 0}$ is a generalized Sturm sequence.
\end{corollary}

\begin{proof}
    Equation \ref{eq:rec-peak-abelian} is easily verified by induction on $n \geq 3$. The proof that $\{\Peak_{\abelian{n}}(t)\}_{n\geq 0}$ is a generalized Sturm sequence is similar to the proof of Corollary \ref{pnest-pol-abelian-rr}.
\end{proof}

\begin{corollary}\label{cor:peak-connected-abelian-rr}
    The polynomials $\Peak_{\abelian{n}}^c(t)$ are real-rooted for all $n \geq 0$.
\end{corollary}

\begin{proof}
    We use a strategy similar to the one used in the proof of \cite[Theorem 1.3]{rank2matroids} and in the proof of \cite[Theorem 1]{liu2026interlacing}. Set $m=\lfloor n/2\rfloor$. Since $t | \Peak_{\abelian{n}}^c(t)$ for all $n \geq 0$, it is sufficient to prove that 
    \[
        p_n(t) \coloneqq \sum_{j = 0}^{m} \binom{n}{2j}t^j-(n-1)t
    \]
    is real-rooted.  The cases $n = 0, \ldots, 5$ are immediate, since $p_0(t) = 0$, $p_1(t) = p_2(t) = 1$, $p_3(t) = 1 + t$, $p_4(t) = 1 + 3t + t^2$ and $p_5(t) = 1 + 6t + 5t^2$, so assume $n\geq 6$. We will use the Intermediate Value Theorem to show that $p_n(t)$ has $m$ distinct zeros.
    
    Let $\operatorname{Re} \colon \mathbb{C} \to \mathbb{R}$ denote the real part of a complex number, that is, if $a, b \in \mathbb{R}$, then $\operatorname{Re}(a + ib) = a$. Since
    \[
            \operatorname{Re}(1+i t)^n = \sum_{j = 0}^m \binom{n}{2j}(-t^2)^j \quad \mbox{and} \quad \operatorname{Re}(1 + i \tan(\theta))^n = \operatorname{Re}\left(  \frac{e^{i n\theta}}{\cos^n (\theta)}  \right) = \frac{\cos(n \theta)}{\cos^n (\theta)},
    \]
    it follows that
    \[
        p_n( - \tan^2 (\theta)) = \frac{\cos(n \theta)}{\cos^n (\theta)} + (n-1)\tan^2(\theta).
    \]
    We will show that $p_n( - \tan^2 (\theta))$ has $m$ distinct zeros in $(0, \pi/2)$.  Since $\cos (\theta)>0$ on this interval, this is equivalent to studying the zeros of
    \begin{equation}
    \label{eq:trigSigns}
        F_n(\theta)\coloneqq \cos(n\theta) + (n-1)\sin^2(\theta)\cos^{n-2}(\theta) = \cos^n (\theta) \, p_n(-\tan^2 (\theta)).
    \end{equation}
    We shall use that
    \begin{equation}
    \label{eq:limits-Fn}
        0 \leq (n-1)\sin^2 (\theta) \cos^{n-2} (\theta) < 1, \qquad 0<\theta<\pi/2,
    \end{equation}
    where the leftmost inequality is obvious. For the rightmost inequality, set $u = \cos^2 (\theta)$. The left hand side becomes
    $(n-1)(1-u)u^{(n-2)/2}$, whose maximum on $0 < u < 1$ is
    \[
        \frac{2(n-1)}{n}\left(1-\frac{2}{n}\right)^{(n-2)/2}<1
    \]
    for $n \geq 2$. For $k=0,1,\ldots,m$, set $\theta_k=k\pi/n$. Since
    \[
        \cos(n \theta_k) =
        \begin{cases}
            \phantom{-}1 &\mbox{if } k \mbox{ is even},\\
            -1 &\mbox{if } k \mbox{ is odd},
        \end{cases} 
    \]
    it follows that
    \[
        F_n (\theta_k) 
        \begin{cases}
            > 0 \quad \mbox{ if } k \mbox{ is even},\\
            < 0 \quad \mbox{ if } k \mbox{ is odd},
        \end{cases}
    \]
    by definition of $F_n(\theta)$ and \eqref{eq:limits-Fn}. Hence, by continuity of $F_n(\theta)$ and by the Intermediate Value Theorem, $F_n(\theta)$ has a zero $x_k$ in each interval $(\theta_{k-1}, \theta_k)$, $k = 1, 2, \ldots, m$. Then, the corresponding numbers $-\tan^2(x_k)$ are $m$ distinct negative zeros of $p_n(t)$.  Since $\deg p_n (t) =m$, all zeros of $p_n(t)$ are real.
\end{proof}


\begin{table}[H]
\centering
\begin{tabular}{@{}r r r@{}}
\toprule
$n$ & $\Peak_{\mathcal{A}(n)}(t)$ & $\Peak_{\mathcal{A}(n)}^c(t)$ \\
\midrule
0  & $1$ & $0$ \\
1  & $t$ & $t$ \\
2  & $t+t^2$ & $t$ \\
3  & $t+3t^2$ & $t+t^2$ \\
4  & $t+6t^2+t^3$ & $t+3t^2+t^3$ \\
5  & $t+10t^2+5t^3$ & $t+6t^2+5t^3$ \\
6  & $t+15t^2+15t^3+t^4$ & $t+10t^2+15t^3+t^4$ \\
7  & $t+21t^2+35t^3+7t^4$ & $t+15t^2+35t^3+7t^4$ \\
8  & $t+28t^2+70t^3+28t^4+t^5$ & $t+21t^2+70t^3+28t^4+t^5$ \\
9  & $t+36t^2+126t^3+84t^4+9t^5$ & $t+28t^2+126t^3+84t^4+9t^5$ \\
10 & $t+45t^2+210t^3+210t^4+45t^5+t^6$ & $t+36t^2+210t^3+210t^4+45t^5+t^6$\\
\bottomrule
\end{tabular}
\caption{The distribution of peaks over Abelian unit interval graphs and connected Abelian unit interval graphs with $n$ vertices, $1 \leq n \leq 10$. The coefficients of $\Peak_{\mathcal{A}(n)}(t)$ are given by \cite[\oeis{A034839}]{OEIS}.}
\label{tab:peak-Abelian}
\end{table}

By Theorem \ref{th:useful}, one easily verifies that $\mbox{\Large $\nicefrac{\Peak_{\abelian{2n}}(t)}{t}$}$ is symmetric for every $n \geq 1$. Since they are also real-rooted by Corollary \ref{cor:peak-pol-albelian-rr}, it follows that the gamma-polynomials associated to them, $\gamma_{\abelian{2n}}^{\peak}(t)$, have nonnegative coefficients and are real-rooted by Lemmas \ref{lem:gamma1} and \ref{lem:gamma2}, see Table \ref{tab:gamma-peak-Abelian}.

We finish this section by finding a closed formula for the coefficients of $\gamma_{\abelian{2n}}^{\peak}(t)$, and by presenting a combinatorial interpretation for them.

\begin{proposition}
\label{prop:gamma-peak-abelian}
    For each $n \in \mathbb{N}$,
    \begin{equation}
    \label{eq:peak-gamma-Abelian}
        \gamma_{\abelian{2n}}^{\peak}(t) = \sum_{i=0}^{\lfloor \frac{n}{2} \rfloor} 2^{2i} \binom{n}{2i} t^i = \frac{(1 + 2\sqrt{t})^n + (1 - 2\sqrt{t})^n}{2} .
    \end{equation}
    Moreover, $\displaystyle 2^{2i} \binom{n}{2i}$ is the number of Abelian Dyck paths $(k, L)$ of semilength $2n$ with $i+1$ peaks for which there is no $j \in [n]$ such that $S_{2j-1} S_{2j} S_{2j+1} = ENE$ or $S_{2j-1} S_{2j} S_{2j+1} = NEN$ are subwords of $LE$, the augmented $NE$ lattice path obtained by appending an $E$ step to $L$.
\end{proposition}

\begin{proof}
    Let $n \in \mathbb{N}$. By \eqref{eq:peak-pol-abelian},
    \begin{align*}
        \frac{\Peak_{\abelian{2n}}(t)}{t} &= \frac{(1+\sqrt{t})^{2n} + (1-\sqrt{t})^{2n}}{2} \\
                                           &= \frac{\left( (1+t) + 2\sqrt{t} \right)^{n} + \left( (1+t) - 2\sqrt{t} \right)^{n}}{2} \\
                                           &= \frac{1}{2} \left[ \sum_{i=0}^{n} \left( 2\sqrt{t} \right)^i \binom{n}{i} (1+t)^{n-i} + \sum_{i=0}^{n} \left( -2\sqrt{t} \right)^i \binom{n}{i} (1+t)^{n-i} \right] \\
                                           &= \sum_{i = 0}^{\lfloor \frac{n}{2} \rfloor} 2^{2i} \binom{n}{2i} t^i (1 + t)^{n-2i}.
    \end{align*}
    Hence, \eqref{eq:peak-gamma-Abelian} holds.

    The proof of the second part of the result is similar to the proof of the second part of Theorem \ref{th:gamma-Abelian}, and we leave it for the reader.
\end{proof}

\begin{remark}
\label{rmk:gamma-peak-abelian}
    The polynomials $\{ \gamma^{\peak}_{\abelian{2n}}(t) \}_{n \geq 1}$ satisfy the recursion
    \begin{equation*}
        \gamma^{\peak}_{\abelian{2n}}(t) =
        \begin{cases}
            1 &\mbox{if } n = 1, \\
            1+4t &\mbox{if } n = 2, \\
            2 \cdot \gamma^{\peak}_{\abelian{2n-2}}(t) + (4t - 1) \cdot \gamma^{\peak}_{\abelian{2n-4}}(t) &\mbox{if } n \geq 3,
        \end{cases}
    \end{equation*}
    from which one can prove that $\{\gamma^{\peak}_{\abelian{2n}}(t)\}_{n\geq 1}$ is a generalized Sturm sequence.
\end{remark}

\begin{table}[H]
\centering
\begin{tabular}{@{}r r@{}}
\toprule
$n$ & $\gamma^{\peak}_{\abelian{2n}}(t)$\\
\midrule
1  & $1$ \\
2  & $1 + 4t$ \\
3  & $1 + 12t$ \\
4  & $1 + 24t + 16t^2$ \\
5  & $1 + 40t + 80t^2$ \\
6  & $1 + 60t + 240t^2 + 64t^3$ \\
7  & $1 + 84t + 560t^2 + 448t^3$ \\
8  & $1 + 112t + 1120t^2 + 1792t^3 + 256t^4$ \\
9  & $1 + 144t + 2016t^2 + 5376t^3 + 2304t^4$ \\
10  & $1 + 180t + 3360t^2 + 13440t^3 + 11520t^4 + 1024t^5$ \\
\bottomrule
\end{tabular}
\caption{The polynomials $\gamma^{\peak}_{\abelian{2n}}(t)$, $1 \leq n \leq 10$.}
\label{tab:gamma-peak-Abelian}
\end{table}

\section{2-nested unit interval graphs}
\label{sec:line-graphs}

In this section, we define $m$-nested unit interval graphs, and study the case $m = 2$. We characterize $2$-nested unit interval graphs in terms of a subclass of line graphs, and study the peak-nesting and the peak-polynomials associated to them, proving that they are all generalized Sturm sequences. In particular, we find new combinatorial interpretations for some sequences in OEIS.

A unit interval graph $\Gamma_\avec$ is called \defin{$m$-nested} if its peak-nesting is at most $m$, and we denote the set of all $m$-nested unit interval graphs on $[n]$ by $\mathcal{N}_m(n)$. The \defin{line graph} of a graph $G = (V,E)$ is a graph $L(G)$ such that its vertex set is $E$, and two vertices of $L(G)$ are adjacent if and only if their corresponding edges share a common vertex in $G$. A \defin{multigraph} is a pair $H = (V, F)$ where $V$ is a set of elements called \defin{vertex set} and $F \subseteq \binom{V}{1} \cup \binom{V}{2}$ is a multiset called \defin{edge set}. In particular, every graph is a multigraph.

\begin{lemma}\cite[Theorem 8.4]{hararyGT1969}
\label{charac-line}
    A graph $G$ is the line graph of some other graph or multigraph if and only if it is possible to find a collection of cliques in $G$ (allowing some of the cliques to be single vertices) that partition the edges of $G$, such that each vertex of $G$ belongs to exactly two of the cliques.
\end{lemma}

A multigraph is called \defin{triangle-free} if it does not contain a triangle as an induced subgraph.

\begin{proposition}
\label{charac-unit-interval-line}
    A unit interval graph is $2$-nested if and only if it is the line graph of a triangle-free multigraph.
\end{proposition}

\begin{proof}
    Assume that $\Gamma_\avec$ is a $2$-nested unit interval graph, and let $\mathcal{P}_\avec = \{P_1, \ldots, P_m\}$ be its set of peak-cliques, where $i < j \implies \min (P_i) < \min (P_j)$. We construct a graph $G$ as follows: its vertex set is $\{P_1, \ldots, P_m\}$, and for each vertex $k$ of $\Gamma_\avec$, if $\pnest(k) = 1$ and $k \in P_i$, let $e$ be an edge whose endpoints are $P_i$; otherwise, let $e$ be the edge whose endpoints are the peak-cliques to which $k$ belongs. By construction, $G$ is triangle-free and $L(G) \cong \Gamma_\avec$.

    Conversely, let $H$ be a triangle free multigraph such that $L(H)$ is a unit interval graph. Then every maximal clique of $L(H)$ arises from the edges incident with a vertex of $H$. Since every edge of $G$ is incident to at most two of its vertices, it follows that every vertex of $L(G)$ belongs to at most two maximal cliques. Hence, the peak-nesting of $L(G)$ is at most $2$.
\end{proof}

Set $g_0 \coloneqq 1$, $h_0 \coloneqq 0$ and, for each $n \in \mathbb{N}$, let $g_n = |\mathcal{N}_2(n)|$ and $h_n = | \{ \Gamma_\avec \in \mathcal{N}_2(n) \colon \Gamma_\avec \mbox{ is connected} \} |$, see Table \ref{tab:unit-line}. By Corollary \ref{prop:pnest-1-and-2}(i) and (ii), and \cite[\oeis{A007051},\oeis{A001519}]{OEIS}, we have the following:

\begin{proposition}
\label{prop:recursions}
    The numbers $g_n$ and $h_n$ satisfy the recursions
    \begin{equation}
    \label{rec1}
        g_n = 
        \begin{cases}
            1 &\mbox{ if } n = 0 \mbox{ or } n = 1, \\
            2 &\mbox{ if } n = 2, \\
            4g_{n-1} - 3g_{n-2} &\mbox{ if } n \geq 3,
        \end{cases}
    \end{equation}
    and 
    \begin{equation}
    \label{rec2}
        h_n = 
        \begin{cases}
            0 &\mbox{ if } n = 0, \\
            1 &\mbox{ if } n=1 \mbox{ or } n = 2, \\
            3h_{n-1} - h_{n-2} &\mbox{ if } n \geq 3 .
        \end{cases}
    \end{equation}
    Hence,
    \[
        g_n = \frac{3^{n-1} + 1}{2} \quad \mbox{and} \quad h_n = F_{2n-3} = \frac{1}{\sqrt{5}} \left( \frac{1 + \sqrt{5}}{2} \right)^{2n-3} - \frac{1}{\sqrt{5}} \left( \frac{1 - \sqrt{5}}{2} \right)^{2n-3}
    \]
    for $n \geq 3$. The generating functions for $\{g_n\}_{n \geq 0}$ and $\{h_n\}_{n \geq 0}$ are, respectively,
    \[
        g(z) = \sum_{n \geq 0} g_n z^n = \frac{1-3z+z^2}{(1-z)(1-3z)} \quad \mbox{and} \quad h(z) = \sum_{n \geq 0} h_n z^n = \frac{z(1 - 2z)}{1 - 3z + z^2}.
    \]
\end{proposition}

\begin{table}[H]
\centering
    \begin{tabular}{ c c  c  c  c  c  c  c  c  c  c  c } 
        \toprule
        $\mathbf{n}$ & $\mathbf{0}$ & $\mathbf{1}$ & $\mathbf{2}$ & $\mathbf{3}$ & $\mathbf{4}$ & $\mathbf{5}$ & $\mathbf{6}$ & $\mathbf{7}$ & $\mathbf{8}$ & $\mathbf{9}$ & $\mathbf{10}$ \\
        \midrule
        $\mathbf{g_n}$ & 1 & 1 & 2 & 5 & 14 & 41 & 122 & 365 & 1094 & 3281 & 9842 \\
        \midrule
        $\mathbf{h_n}$ & 0 & 1 & 1 & 2 & 5 & 13 & 34 & 89 & 233 & 610 & 1597 \\
        \bottomrule
    \end{tabular}
    \caption{Table of values of $g_n$ and $h_n$ for $0 \leq n \leq 10$. These sequences are \cite[\oeis{A007051}]{OEIS} and \cite[\oeis{A001519}]{OEIS}, respectively.}
    \label{tab:unit-line}
\end{table}

\begin{corollary}
    The number of 2-nested Abelian unit interval graphs on $[n]$ is given by the Cake numbers, \cite[\oeis{A000125}]{OEIS}. An explicit formula is $\binom{n-1}{0}+\binom{n-1}{1}+\binom{n-1}{2}+\binom{n-1}{3}$.
\end{corollary}

\begin{proof}
    By Proposition \ref{charac-unit-interval-line}, a unit interval graph is a 2-nested if and only if its peak-nesting is at most $2$. Hence, by Theorem \ref{thm:pnest-pol-abelian} the number of 2-nested Abelian unit interval graphs is
    \begin{align*}
        [t]\pnest_{\mathcal{A}(n)}(t) + [t^2]\pnest_{\mathcal{A}(n)}(t) &= \binom{n}{1} + \binom{n}{3} \\
                                                                        &= \binom{n-1}{0} + \binom{n-1}{1}  + \binom{n-1}{2} + \binom{n-1}{3}.
    \end{align*}
\end{proof}

\subsection{The peak-nesting polynomial of 2-nested unit interval graphs}
\label{sub:pnest-2nested}

Set $\Pnest_{\mathcal{N}_2(0)}(t) \coloneqq 1$ and $\Pnest_{\mathcal{N}_2(0)}^c(t) \coloneqq 0$. For each $n \in \mathbb{N}$, let
\[
    \Pnest_{\mathcal{N}_2(n)}(t) \coloneqq \sum_{\avec \in \mathcal{N}_2(n)} t^{\pnest(\avec)} \qquad \mbox{and} \qquad \Pnest_{\mathcal{N}_2(n)}^c(t) \coloneqq \sum_{\substack{\avec \in \mathcal{N}_2(n) \\ \Gamma_\avec \mbox{ \scriptsize is connected}}} t^{\pnest(\avec)},
\]
see Table \ref{tab:pnest-line}.

By Corollary~\ref{prop:pnest-1-and-2},
\[
    \Pnest_{\mathcal{N}_2(n)}(t) = 2^{n-1}t + S(n,3) t^2 \qquad \mbox{and} \qquad \Pnest_{\mathcal{N}_2(n)}^c(t) = t + (F_{2n-3} - 1) t^2,
\]
for $n \geq 1$, from which we get the following:

\begin{corollary}\label{cor:pnest-line-rr}
    The polynomials $\Pnest_{\mathcal{N}_2(n)}(t)$ and $\Pnest_{\mathcal{N}_2(n)}^c(t)$ are real-rooted for every $n \geq 0$, and the sequences $\left\{ \Pnest_{\mathcal{N}_2(n)}(t) \right\}_{n \geq 0}$ and $\left\{ \Pnest_{\mathcal{N}_2(n)}^c(t) \right\}_{n \geq 0}$ are generalized Sturm sequences.
\end{corollary}

\begin{table}[H]
\centering
\begin{tabular}{@{}r r r@{}}
\toprule
$n$ & $\Pnest_{\mathcal{N}_2(n)}(t)$ & $\Pnest_{\mathcal{N}_2(n)}^c(t)$ \\
\midrule
0  & $1$ & $0$ \\
1  & $t$ & $t$ \\
2  & $2t$ & $t$ \\
3  & $4t + t^2$ & $t + t^2$ \\
4  & $8t + 6t^2$ & $t + 4t^2$ \\
5  & $16t + 25t^2$ & $t + 12t^2$ \\
6  & $32t + 90t^2$ & $t + 33t^2$ \\
7  & $64t + 301t^2$ & $t + 88t^2$ \\
8  & $128t + 966t^2$ & $t + 232t^2$ \\
9  & $256t + 3025t^2$ & $t + 609t^2$ \\
10 & $512t + 9330t^2$ & $t + 1596t^2$\\
\bottomrule
\end{tabular}
\caption{The distribution of peak-nestings over 2-nested unit interval graphs and connected 2-nested unit interval graphs with $n$ vertices, $0 \leq n \leq 10$.}
\label{tab:pnest-line}
\end{table}

\subsection{The peak polynomial of 2-nested unit interval graphs}

Set $\Peak_{\mathcal{N}_2(0)}(t) \coloneqq 1$ and $\Peak_{\mathcal{N}_2(0)}^c(t) \coloneqq 0$. For each $n \in \mathbb{N}$, let
\[
    \Peak_{\mathcal{N}_2(n)}(t) \coloneqq \sum_{\avec \in \mathcal{N}_2(n)} t^{\peak(\avec)} \qquad \mbox{and} \qquad \Peak_{\mathcal{N}_2(n)}^c(t) \coloneqq \sum_{\substack{\avec \in \mathcal{N}_2(n) \\ \Gamma_\avec \mbox{ \scriptsize is connected}}} t^{\peak(\avec)} ,
\]
see Table \ref{tab:peak-line}.

\begin{theorem}
\label{th:peak-line-formula}
    For each $n \in \mathbb{N}$,
    \begin{equation}
    \label{eq:peak-pol-line}
        \Peak_{\mathcal{N}_2(n)}(t) = \sum_{i=0}^{n-1} \left[ \sum_{k=0}^i \binom{n-1}{k} \binom{n-1-k}{2(i-k)} \right] t^{i+1} = \frac{t \cdot \left[ (1 + t + \sqrt{t})^{n-1} + (1 + t - \sqrt{t})^{n-1} \right]}{2} .
    \end{equation}
\end{theorem}

\begin{proof}
    Fix $n \in \mathbb{N}$. Given $\avec \in \mathcal{N}_2(n)$, let $M^*_\avec$ be its associated bicolored Motzkin path. By Theorem \ref{th:bij-dyck-decorated-motzkin}(ii) and by \eqref{eq:bij-dmotzkin},
    \begin{equation}
    \label{eq:peak-path}
        \peak(\avec) = 1 + \mathrm{u}(M^*_\avec) + \mathrm{h}^*(M^*_\avec).
    \end{equation}
    Moreover, since $\Gamma_\avec$ is 2-nested, it follows that $\height(M^*_\avec) \leq 1$, which implies that $M^*_\avec$ can be decomposed as
    \[
        W_1 (U W_2 D) W_3 (U W_4 D) \cdots (U W_{2m} D) W_{2m+1},
    \]
    where $W_1, \ldots, W_{2m+1}$ are (possibly empty) words in $\{H, H^*\}$.

    Assume that $\peak(\avec) = i+1$ and that $h^*(M^*_\avec) = k$. Recall that $M^*_\avec$ has $n-1$ steps. By \eqref{eq:peak-path}, $\mathrm{u}(M^*_\avec) = i - k$. Hence, there are
    \[
        \binom{n-1}{k} \binom{n-1-k}{2(i-k)}
    \]
    many such paths. Summing over all $0 \leq k \leq i$ gives the first equality of \eqref{eq:peak-pol-line}.

    The second equality of \eqref{eq:peak-pol-line} follows by induction on $n \geq 1$.
\end{proof}

\begin{corollary}
\label{cor:peak-line-rr}
    The polynomials $\{\Peak_{\mathcal{N}_2(n)}(t)\}_{n \geq 1}$ satisfy the recursion
    \begin{equation}
    \label{eq:rec-peak-pol-line}
        \Peak_{\mathcal{N}_2(n)}(t) =
        \begin{cases}
            t \quad &\mbox{if } n = 1, \\
            t + t^2 \quad &\mbox{if } n = 2, \\
            2 \left( t+1 \right) \cdot \Peak_{\mathcal{N}_2(n-1)}(t) - (1+t+t^2) \cdot \Peak_{\mathcal{N}_2(n-2)}(t)  \quad &\mbox{if } n \geq 3.
        \end{cases}
    \end{equation}
    Moreover, $\Peak_{\mathcal{N}_2(n)}(t)$ is real-rooted for every $n \geq 0$ and $\left\{ \Peak_{\mathcal{N}_2(n)}(t) \right\}_{n \geq 0}$ is a Sturm sequence.
\end{corollary}

\begin{proof}
    The recursion given in \eqref{eq:rec-peak-pol-line} is easily verified by induction on $n \geq 3$. To prove that $\left\{ \Peak_{\mathcal{N}_2(n)}(t) \right\}_{n \geq 0}$ is a Sturm sequence, first observe that $\Peak_{\mathcal{N}_2(0)}(t) \interl \Peak_{\mathcal{N}_2(1)}(t)$, and assume by induction on $n$ that $\Peak_{\mathcal{N}_2(n-2)}(t) \interl \Peak_{\mathcal{N}_2(n-1)}(t)$ for some $n-1 > 1$. Since $\Peak_{\mathcal{N}_2(n)}(t)$ and $\Peak_{\mathcal{N}_2(n-2)}(t)$ have positive leading coefficients by Theorem \ref{th:peak-line-formula}, and $-(1+t+t^2) \leq 0$ for all $t \in \mathbb{R}$, Lemma \ref{recursion-rr} implies that $\Peak_{\mathcal{N}_2(n-1)}(t) \interl \Peak_{\mathcal{N}_2(n)}(t)$, which concludes the proof.
\end{proof}

\begin{theorem}
    For each $n \geq 2$,
    \begin{equation}
    \label{eq:cpeak-pol-line}
        \Peak_{\mathcal{N}_2(n)}^c(t) = \sum_{i=0}^{n-2} \binom{n-2+i}{2i} t^{i+1} = t \cdot \leftidx{_2}F_1 \left( 2-n, n-1 ; \frac{1}{2}; -\frac{t}{4} \right).
    \end{equation}
\end{theorem}

\begin{proof}
    Fix $n \geq 2$. Let $\mathcal{C}_{n,i}$ be the set of all bicolored Motzkin paths of length $n-1$ associated to connected 2-nested unit interval graphs with $i+1$ peaks and $\mathcal{D}_{n,i}$ be the set of all Motzkin paths of length $n-2+i$ with $i$ $U$ steps and height at most $1$. We will show that there is a bijection between $\mathcal{C}_{n,i}$ and $\mathcal{D}_{n,i}$.

    In fact, let $\Gamma_\avec \in \mathcal{N}_2(n)$ be connected and let $M^*_\avec$ be its associated bicolored Motzkin path. Then $M^*_\avec$ can be decomposed as
    \[
        W_1 (U W_2 D) W_3 (U W_4 D) \cdots (U W_{2m} D) W_{2m+1},
    \]
    where $W_{2j}$ is a word in $\{H, H^*\}$ and $W_{2j+1} = H^k$ for some $k \geq 0$, by connectedness of $\Gamma_\avec$.
    
    Assume that $\peak(\avec) = i+1$. Using the notation of the proof of Theorem \ref{th:peak-line-formula}, observe that $\mathrm{u}(M^*_\avec) = m$ and let $\mathrm{h}^*(M^*_\avec) = h^*$. Then $m + h^* = i$ by Equation \ref{eq:peak-path}. If $i = 0$, then $m = h^* = 0$, and $M^*_\avec = H^{n-1}$. Hence $[1]\Peak_{\mathcal{N}_2(n)}^c(t) = 1$; if $i \geq 1$, let $\tilde{M}_\avec$ be the (uncolored) Motzkin path obtained from $M^*_\avec$ by replacing each $H^*$ by a $DU$ word and by adding an $H$ step between each word $U W_{2j} D$ and $U W_{2j+2} D$, $j \in [m-1]$. Since $\height(\tilde{M}_\avec) \leq 1$, and $\tilde{M}_\avec$ has $(n-1) + h^* + (m-1) = n-2+i$ steps, of which $m + h^* = i$ are $U$ steps, it follows that $\tilde{M}_\avec$ is a Motzkin path in $\mathcal{D}_{n,i}$. Hence the map $M^*_\avec \mapsto \tilde{M}_\avec$ is clearly a bijection between $\mathcal{C}_{n,i}$ and $\mathcal{D}_{n,i}$.
    
    We now observe that
    \[
        |\mathcal{D}_{n,i}| = \binom{n-2+i}{2i},
    \]
    since $\height(M) \leq 1$ for all $M \in \mathcal{D}_{n,i}$, which finishes the proof.

    The second equality of \eqref{eq:cpeak-pol-line} is verified by induction on $n \geq 2$.
\end{proof}

\begin{corollary}
\label{cor:cpeak-line-rr}
    The polynomials $\{\Peak_{\mathcal{N}_2(n)}^c(t)\}_{n \geq 1}$ satisfy the recursion
    \begin{equation}
    \label{eq:crec-peak-pol-line}
        \Peak_{\mathcal{N}_2(n)}^c(t) =
        \begin{cases}
            t \quad &\mbox{if } n = 1 \mbox{ or } n = 2, \\
            \left( t+2 \right) \cdot \Peak_{\mathcal{N}_2(n-1)}^c(t) - \Peak_{\mathcal{N}_2(n-2)}^c(t)  \quad &\mbox{if } n \geq 3.
        \end{cases}
    \end{equation}
    Moreover, $\Peak_{\mathcal{N}_2(n)}^c(t)$ is real-rooted for every $n \geq 0$ and $\left\{ \Peak_{\mathcal{N}_2(n)}^c(t) \right\}_{n \geq 0}$ is a generalized Sturm sequence.
\end{corollary}

\begin{proof}
    The recursion given in \eqref{eq:crec-peak-pol-line} is easily verified by induction on $n \geq 3$. The proof that $\{\Peak_{\mathcal{N}_2(n)}^c(t)\}_{n\geq 0}$ is a generalized Sturm sequence is similar to the proof of Corollary \ref{cor:peak-line-rr}.
\end{proof}

\begin{table}[H]
\centering
\begin{tabular}{@{}r r r@{}}
\toprule
$n$ & $\Peak_{\mathcal{N}_2(n)}(t)$ & $\Peak_{\mathcal{N}_2(n)}^c(t)$ \\
\midrule
0  & $1$ & $0$ \\
1  & $t$ & $t$ \\
2  & $t + t^2$ & $t$ \\
3  & $t + 3t^2 + t^3$ & $t + t^2$ \\
4  & $t + 6t^2 + 6t^3 + t^4$ & $t + 3t^2 + t^3$ \\
5  & $t + 10t^2 + 19t^3 + 10t^4 + t^5$ & $t + 6t^2 + 5t^3 + t^4$ \\
6  & $t + 15t^2 + 45t^3 + 45t^4 + 15t^5 + t^6$ & $t + 10t^2 + 15t^3 + 7t^4 + t^5$ \\
7  & $t + 21t^2 + 90t^3 + 141t^4 + 90t^5 + 21t^6 + t^7$ & $t + 15t^2 + 35t^3 + 28t^4 + 9t^5 + t^6$ \\
8  & $t + 28t^2 + 161t^3 + 357t^4 + 357t^5 + 161t^6 + 28t^7 + t^8$ & $t + 21t^2 + 70t^3 + 84t^4 + 45t^5 + 11t^6 + t^7$ \\
\bottomrule
\end{tabular}
\caption{The distribution of peaks over 2-nested unit interval graphs and connected 2-nested unit interval graphs with $n$ vertices, $0 \leq n \leq 8$. The coefficients of $\Peak_{\mathcal{N}_2(n)}(t)$ are given by \cite[\oeis{A056241}]{OEIS}, while the coefficients of $\Peak_{\mathcal{N}_2(n)}^c(t)$ are given by \cite[\oeis{A085478}]{OEIS}.}
\label{tab:peak-line}
\end{table}

By Table \ref{tab:peak-line}, one verifies that $\mbox{\Large $\nicefrac{\Peak_{\mathcal{N}_2(n)}(t)}{t}$}$ is symmetric for every $n \in \mathbb{N}$, which we prove in Proposition \ref{prop:gamma-line}. Since they are also real-rooted by Corollary \ref{cor:peak-line-rr}, it follows that the gamma-polynomials associated to them, $\gamma_{\mathcal{N}_2(n)}(t)$, have nonnegative coefficients and are real-rooted by Lemmas \ref{lem:gamma1} and \ref{lem:gamma2}.

We finish this section by finding a closed formula for the coefficients of $\gamma_{\mathcal{N}_2(n)}(t)$, and by presenting a combinatorial interpretation for them.

\begin{proposition}
\label{prop:gamma-line}
    For each $n \in \mathbb{N}$,
    \begin{equation}
    \label{eq:gamma-line}
        \gamma_{\mathcal{N}_2(n)}(t) = \sum_{i=0}^{\lfloor \frac{n-1}{2} \rfloor} \binom{n-1}{2i} t^i = \frac{(1+\sqrt{t})^{n-1} + (1-\sqrt{t})^{n-1}}{2} .
    \end{equation}
    Moreover, $\displaystyle \binom{n-1}{2i}$ is the number of Motzkin paths of length $n-1$ with $i$ $U$ steps and height at most $1$ (or, equivalently, it is the number of Motzkin sequences of size $n$ with $i$ ascents whose largest entry is at most $2$).
\end{proposition}

\begin{proof}
    Let $n \in \mathbb{N}$. First, we will prove that $\mbox{\Large $\nicefrac{\Peak_{\mathcal{N}_2(n)}(t)}{t}$}$ is symmetric for every $n \in \mathbb{N}$. In fact, by \eqref{eq:peak-pol-line},
    \begin{align*}
        t^{n-1} \cdot \frac{\Peak_{\mathcal{N}_2(n)}(1/t)}{1/t} &= \frac{t^{n-1}}{2} \left[ \left( 1 + \frac{1}{t} + \frac{1}{\sqrt{t}} \right)^{n-1} + \left( 1 + \frac{1}{t} - \frac{1}{\sqrt{t}} \right)^{n-1} \right] \\
        &= \frac{t^{n-1}}{2} \left[ \left( \frac{t + 1 + \sqrt{t}}{t} \right)^{n-1} + \left( \frac{t + 1 - \sqrt{t}}{t} \right)^{n-1} \right] \\
        &= \frac{\Peak_{\mathcal{N}_2(n)}(t)}{t}.
    \end{align*}
    Equation \ref{eq:peak-pol-line} also implies
    \begin{align*}
        \frac{\Peak_{\mathcal{N}_2(n)}(t)}{t} &= \frac{(1 + t + \sqrt{t})^{n-1} + (1 + t - \sqrt{t})^{n-1}}{2} \\
                                            &= \frac{1}{2} \left[ \sum_{i=0}^{n-1} \binom{n-1}{i} (\sqrt{t})^i (1+t)^{n-1-i} + \sum_{i=0}^{n-1} \binom{n-1}{i} (-\sqrt{t})^i (1+t)^{n-1-i} \right] \\
                                            &= \sum_{i=0}^{\lfloor \frac{n-1}{2} \rfloor} \binom{n-1}{2i} t^i (1+t)^{n-1-2i}. \\
    \end{align*}
    Hence, \eqref{eq:gamma-line} holds.

    The second part of the claim follows from Corollary \ref{cor:peak-and-pnest}. In fact,
    \begin{align*}
        \frac{\Peak_{\mathcal{N}_2(n)}(t)}{t} &= \frac{1}{t} \sum_{\avec \in \mathcal{N}_2(n)} t^{\peak(\avec)} \\
                                            &= \frac{1}{t} \sum_{\substack{\wvec = (w_1, \ldots, w_n) \in \mathcal{MP}(n) \\ w_i \in \{1,2\} \mbox{ } \forall i \in [n-1]}} \sum_{\substack{\avec \in \mathcal{N}_2(n) \\ \pnv(\avec)=\wvec}} t^{\peak(\avec)} \\
                                            &= \sum_{\substack{\wvec = (w_1, \ldots, w_n) \in \mathcal{MP}(n) \\ w_i \in \{1,2\} \mbox{ } \forall i \in [n-1]}} t^{\asc(\wvec)}(1+t)^{n-1-2\asc(\wvec)} \\
                                            &= \sum_{i=0}^{\lfloor \frac{n-1}{2} \rfloor} \binom{n-1}{2i} t^i (1+t)^{n-1-2i},
    \end{align*}
    which concludes the proof.
\end{proof}

\begin{remark}
\label{rmk:gamma-peak-line}
    For each $n \in \mathbb{N}$, we have
    \[
        \gamma_{\mathcal{N}_2(n+1)}(t) = \frac{\Peak_{\abelian{n}}(t)}{t},
    \]
    see Table \ref{tab:peak-Abelian}. Hence, the polynomials $\{ \gamma_{\mathcal{N}_2(n)}(t) \}_{n \geq 1}$ satisfy the recursion
    \begin{equation}
    \label{rec:gamma-line}
        \gamma_{\mathcal{N}_2(n)}(t) =
        \begin{cases}
            1 &\mbox{if } n = 1 \mbox{ or } n=2, \\
            2 \cdot \gamma_{\mathcal{N}_2(n-1)}(t) + (t-1) \cdot \gamma_{\mathcal{N}_2(n-2)}(t) &\mbox{if } n \geq 3,
        \end{cases}
    \end{equation}
    from which one can prove that $\{\gamma_{\mathcal{N}_2(n)}(t)\}_{n\geq 1}$ is a generalized Sturm sequence.

    Moreover, $\gamma_{\mathcal{N}_2(2n+1)}(t)$ is symmetric for all $n \geq 0$, and
    \[
        \gamma_{\mathcal{N}_2(2n+1)}(t) = \sum_{i=0}^{n} 2^{2i} \binom{n}{2i} t^i (1+t)^{n-2i},
    \]
    where $\displaystyle 2^{2i} \binom{n}{2i}$ is the number of Motzkin paths of length $n$ and height at most $1$ in which every $U$ step and $D$ step is independently colored with one of two possible colors.
\end{remark}

\section{Reduced unit interval graphs}
\label{sec:reduced}

Let $N(i)$ be the set of neighbors of a vertex $i$ in a graph. We say that a unit interval graph $\Gamma_\avec$ is \defin{reduced} \cite{HanlonCounting1982} if there are no two adjacent vertices $j$ and $j+1$ such that $N(j) \setminus \{j+1\} = N(j+1) \setminus \{j\}$, that is, $j$ and $j+1$ have the same set of neighbors. The set of reduced unit interval graphs on $[n]$ will be denoted by $\mathcal{R}(n)$.

The \defin{twin number} of $\Gamma_\avec$ is the number of pairs $(j, j+1)$ of adjacent vertices such that $N(j) \setminus \{j+1\} = N(j+1) \setminus \{j\}$ in $\Gamma_\avec$, and it is denoted by $\twin(\Gamma_\avec)$. The reduced unit interval graphs are exactly those whose twin number is $0$.

\begin{example}
    The following figure shows the 14 area sequences $\avec$ of length 4, with the twin number and the set of neighbors of each vertex of $\Gamma_\avec$.
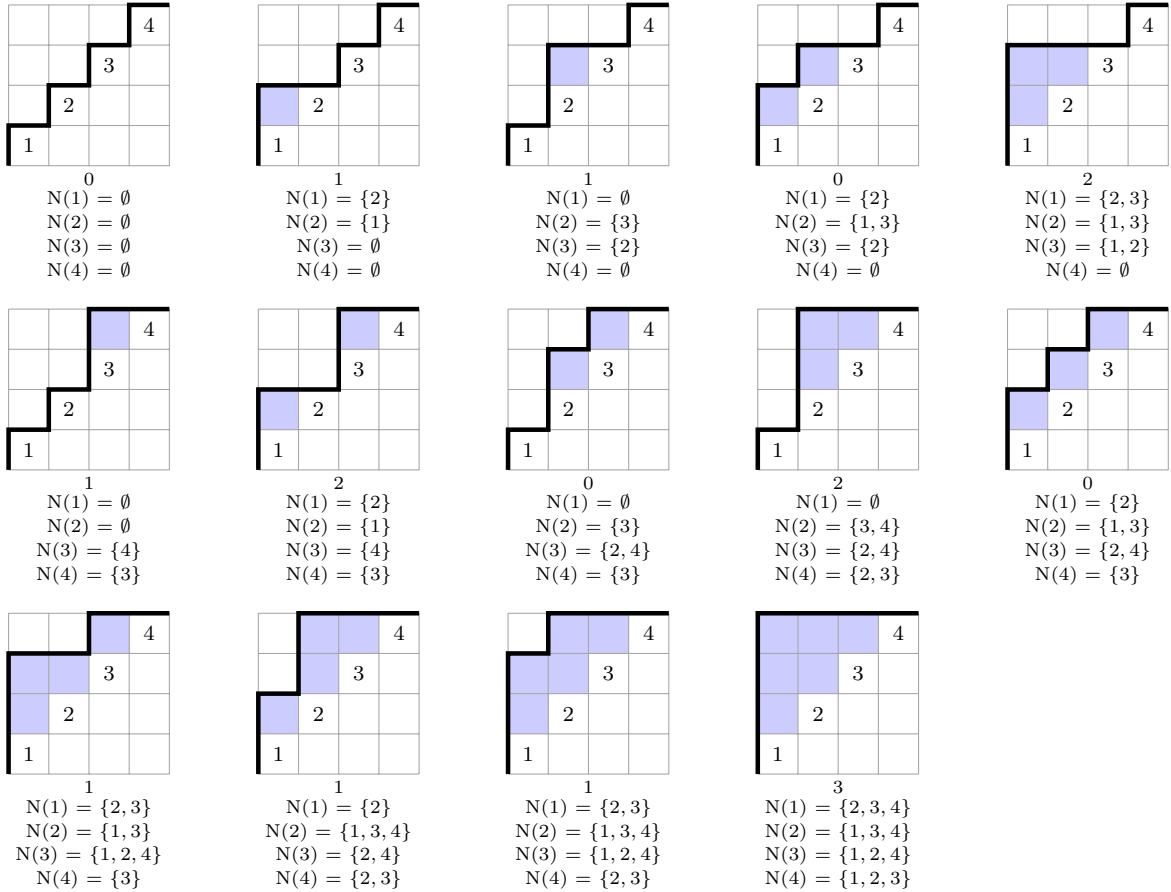
    \begin{center}
        \setlength{\tabcolsep}{1pt}
        \begin{tabular}{p{0.19\linewidth} p{0.19\linewidth} p{0.19\linewidth} p{0.19\linewidth} p{0.19\linewidth}}
            \shortstack{\dyckDiagram[0.03125\textwidth]{0,0,0,0}\\[-1pt]{\scriptsize 0}\\[-1pt]{\scriptsize N(1) = $\emptyset$}\\[-1pt]{\scriptsize N(2) = $\emptyset$}\\[-1pt]{\scriptsize N(3) = $\emptyset$}\\[-1pt]{\scriptsize N(4) = $\emptyset$}} &
            \shortstack{\dyckDiagram[0.03125\textwidth]{0,1,0,0}\\[-1pt]{\scriptsize 1}\\[-1pt]{\scriptsize N(1) = $\{2\}$}\\[-1pt]{\scriptsize N(2) = $\{1\}$}\\[-1pt]{\scriptsize N(3) = $\emptyset$}\\[-1pt]{\scriptsize N(4) = $\emptyset$}} &
            \shortstack{\dyckDiagram[0.03125\textwidth]{0,0,1,0}\\[-1pt]{\scriptsize 1}\\[-1pt]{\scriptsize N(1) = $\emptyset$}\\[-1pt]{\scriptsize N(2) = $\{3\}$}\\[-1pt]{\scriptsize N(3) = $\{2\}$}\\[-1pt]{\scriptsize N(4) = $\emptyset$}} &
            \shortstack{\dyckDiagram[0.03125\textwidth]{0,1,1,0}\\[-1pt]{\scriptsize 0}\\[-1pt]{\scriptsize N(1) = $\{2\}$}\\[-1pt]{\scriptsize N(2) = $\{1,3\}$}\\[-1pt]{\scriptsize N(3) = $\{2\}$}\\[-1pt]{\scriptsize N(4) = $\emptyset$}} &
            \shortstack{\dyckDiagram[0.03125\textwidth]{0,1,2,0}\\[-1pt]{\scriptsize 2}\\[-1pt]{\scriptsize N(1) = $\{2,3\}$}\\[-1pt]{\scriptsize N(2) = $\{1,3\}$}\\[-1pt]{\scriptsize N(3) = $\{1,2\}$}\\[-1pt]{\scriptsize N(4) = $\emptyset$}}
            \\[8pt]
            \shortstack{\dyckDiagram[0.03125\textwidth]{0,0,0,1}\\[-1pt]{\scriptsize 1}\\[-1pt]{\scriptsize N(1) = $\emptyset$}\\[-1pt]{\scriptsize N(2) = $\emptyset$}\\[-1pt]{\scriptsize N(3) = $\{4\}$}\\[-1pt]{\scriptsize N(4) = $\{3\}$}} &
            \shortstack{\dyckDiagram[0.03125\textwidth]{0,1,0,1}\\[-1pt]{\scriptsize 2}\\[-1pt]{\scriptsize N(1) = $\{2\}$}\\[-1pt]{\scriptsize N(2) = $\{1\}$}\\[-1pt]{\scriptsize N(3) = $\{4\}$}\\[-1pt]{\scriptsize N(4) = $\{3\}$}} &
            \shortstack{\dyckDiagram[0.03125\textwidth]{0,0,1,1}\\[-1pt]{\scriptsize 0}\\[-1pt]{\scriptsize N(1) = $\emptyset$}\\[-1pt]{\scriptsize N(2) = $\{3\}$}\\[-1pt]{\scriptsize N(3) = $\{2,4\}$}\\[-1pt]{\scriptsize N(4) = $\{3\}$}} &
            \shortstack{\dyckDiagram[0.03125\textwidth]{0,0,1,2}\\[-1pt]{\scriptsize 2}\\[-1pt]{\scriptsize N(1) = $\emptyset$}\\[-1pt]{\scriptsize N(2) = $\{3,4\}$}\\[-1pt]{\scriptsize N(3) = $\{2,4\}$}\\[-1pt]{\scriptsize N(4) = $\{2,3\}$}} &
            \shortstack{\dyckDiagram[0.03125\textwidth]{0,1,1,1}\\[-1pt]{\scriptsize 0}\\[-1pt]{\scriptsize N(1) = $\{2\}$}\\[-1pt]{\scriptsize N(2) = $\{1,3\}$}\\[-1pt]{\scriptsize N(3) = $\{2,4\}$}\\[-1pt]{\scriptsize N(4) = $\{3\}$}}
            \\[8pt]
            \shortstack{\dyckDiagram[0.03125\textwidth]{0,1,2,1}\\[-1pt]{\scriptsize 1}\\[-1pt]{\scriptsize N(1) = $\{2,3\}$}\\[-1pt]{\scriptsize N(2) = $\{1,3\}$}\\[-1pt]{\scriptsize N(3) = $\{1,2,4\}$}\\[-1pt]{\scriptsize N(4) = $\{3\}$}} &
            \shortstack{\dyckDiagram[0.03125\textwidth]{0,1,1,2}\\[-1pt]{\scriptsize 1}\\[-1pt]{\scriptsize N(1) = $\{2\}$}\\[-1pt]{\scriptsize N(2) = $\{1,3,4\}$}\\[-1pt]{\scriptsize N(3) = $\{2,4\}$}\\[-1pt]{\scriptsize N(4) = $\{2,3\}$}} &
            \shortstack{\dyckDiagram[0.03125\textwidth]{0,1,2,2}\\[-1pt]{\scriptsize 1}\\[-1pt]{\scriptsize N(1) = $\{2,3\}$}\\[-1pt]{\scriptsize N(2) = $\{1,3,4\}$}\\[-1pt]{\scriptsize N(3) = $\{1,2,4\}$}\\[-1pt]{\scriptsize N(4) = $\{2,3\}$}} &
            \shortstack{\dyckDiagram[0.03125\textwidth]{0,1,2,3}\\[-1pt]{\scriptsize 3}\\[-1pt]{\scriptsize N(1) = $\{2,3,4\}$}\\[-1pt]{\scriptsize N(2) = $\{1,3,4\}$}\\[-1pt]{\scriptsize N(3) = $\{1,2,4\}$}\\[-1pt]{\scriptsize N(4) = $\{1,2,3\}$}}
        \end{tabular}
    \end{center}
\end{example} 

\begin{proposition}
\label{prop:reduced-sequences}
    For each $n \geq 0$, let $m_n = |\mathcal{R}(n)|$, $r_n = |\{ \Gamma_\avec \in \mathcal{R}(n) \colon \Gamma_\avec \mbox{ is connected} \}$, $t_n = |\mathcal{R}(n) \cap \mathcal{N}_2(n)|$ and $f_n = |\{ \Gamma_\avec \in \mathcal{R}(n) \cap \mathcal{N}_2(n) \colon \Gamma_\avec \mbox{ is connected} \}$. Then
    \begin{enumerate}[(i)]
        \item $m_0 = 1$ and $m_n = M_{n-1}$ for $n \geq 1$, where $M_n$ is the $n$th Motzkin number \cite[\oeis{A001006}]{OEIS};
        \item $r_0 = 0$ and $r_n = R_{n-1}$ for $n \geq 1$, where $R_n$ is the $n$th Riordan number \cite[\oeis{A005043}]{OEIS};
        \item $t_0 = 1$, $t_1 = 1$ and $t_n = 2^{n-2}$ for $n \geq 2$;
        \item $f_0 = 0$, $f_1 = 1$ and $f_n = F_{n-2}$ for $n \geq 2$.
    \end{enumerate}
\end{proposition}

\begin{table}[H]
\centering
    \begin{tabular}{ c c  c  c  c  c  c  c  c  c  c  c } 
        \toprule
        $\mathbf{n}$ & $\mathbf{0}$ & $\mathbf{1}$ & $\mathbf{2}$ & $\mathbf{3}$ & $\mathbf{4}$ & $\mathbf{5}$ & $\mathbf{6}$ & $\mathbf{7}$ & $\mathbf{8}$ & $\mathbf{9}$ & $\mathbf{10}$ \\
        \midrule
        $\mathbf{m_n}$ & 1 & 1 & 1 & 2 & 4 & 9 & 21 & 51 & 127 & 323 & 835 \\
        \midrule
        $\mathbf{r_n}$ & 0 & 1 & 0 & 1 & 1 & 3 & 6 & 15 & 36 & 91 & 232 \\
        \midrule
        $\mathbf{t_n}$ & 1 & 1 & 1 & 2 & 4 & 8 & 16 & 32 & 64 & 128 & 256 \\
        \bottomrule
        $\mathbf{f_n}$ & 0 & 1 & 0 & 1 & 1 & 2 & 3 & 5 & 8 & 13 & 21 \\
        \midrule
    \end{tabular}
    \caption{Table of values of $m_n$, $r_n$, $t_n$ and $f_n$ for $0 \leq n \leq 10$.}
\end{table}

\begin{proof}
    \hfill
    \begin{enumerate}[(i)]
        \item The empty graph is reduced by definition, so $m_0 = 1$. For $n \geq 1 $, observe that, by the proof of Theorem \ref{th:bij-dyck-decorated-motzkin}, it follows that a unit interval graph is reduced if and only if its associated bicolored Motzkin path is such that all of its plateaus are marked, from which the claim follows;
        \item By definition, the empty graph is not connected, so $r_0 = 0$. For $n \geq 1$, the proof of $(i)$ implies that $r_n$ equals the number of Motzkin sequences $\wvec = (w_1, \ldots, w_n)$ with no index $i \in [n]$ such that $w_i = w_{i+1} = 1$ (also known as \defin{Riordan sequences}), which, by \eqref{eq:bij-motzkin-path-seq} and the proof of Lemma \ref{lem:pnest=height}, is the number of Motzkin paths of length $n-1$ with no horizontal steps at height $0$. By \cite[\oeis{A005043}]{OEIS}, this number is equal to $R_{n-1}$;
        \item The empty graph is reduced and 2-nested by definition, so $t_0 = 1$. For $n \geq 1$, by Proposition \ref{charac-unit-interval-line} and by the proof of $(i)$, the number of reduced 2-nested unit interval graphs equals the number of Motzkin sequences $\wvec = (w_1, \ldots, w_n)$ such that $w_i \in \{1, 2\}$ for all $i \in [n]$. Since $w_1 = w_n = 1$, it follows that $t_n = 2^{n-2}$ for $n \geq 2$;
        \item By definition, the empty graph is not connected, so $f_0 = 0$. For $n \geq 1$, the proofs of $(ii)$ and $(iii)$ imply that the number of reduced connected 2-nested unit interval graphs equals the number of Motzkin sequences $\wvec = (w_1, \ldots, w_n)$ such that $w_i \in \{1, 2\}$ for all $i \in [n]$ and no two consecutive entries of $\wvec$ are equal to $1$, which equals $F_{n-2}$ \cite[\oeis{A000045}]{OEIS} for $n \geq 2$.
    \end{enumerate}
\end{proof}

For each $n \geq 0$, let
\[
    \defin{ \twin_n(t) } \coloneqq \sum_{\avec \in \mathcal{DP}(n)} t^{\twin(\Gamma_\avec)} = \sum_{k=0}^n \tau_{n,k} t^k,
\]
where $\tau_{n,k} \coloneqq |\{ \avec \in \mathcal{DP}(n) \colon \twin(\Gamma_\avec) = k \}|$, see Table \ref{tab:twin-distribution}.

\begin{proposition}
\label{prop:twin-formula}
    For each $n \geq 0$,
    \begin{equation*}
        \tau_{n+1,k} = \binom{n}{k} M_{n-k}
    \end{equation*}
    for all $0 \leq k \leq n+1$, from which follows that the twin statistic is distributed as \cite[\oeis{A091869}]{OEIS} over all Dyck paths. In particular,
    \[
        \twin_n(t) = (1+t)^{n-1} \leftidx{_2}F_1 \left( -\frac{n-1}{2}, \frac{2-n}{2}; 2; \frac{4}{(1+t)^2} \right) .
    \]
\end{proposition}

\begin{proof}
    Fix $n \geq 0$. By Proposition \ref{prop:reduced-sequences}(ii)
    \[
        \tau_{n+1,0} = | \mathcal{R}(n+1) | = M_{n}.
    \]
    Now, let $0 < k \leq n+1$. Let $\Gamma_\avec$ be a reduced unit interval graph on $[n+1-k]$. Among the $n+1-k$ vertices of $\Gamma_\avec$, choose $k$ of them, repetitions are allowed, and for each one of the chosen vertices create a twin. A twin of a vertex $i$ is created as follows: let $C$ be the connected component of $\Gamma_\avec$ that contains $i$ and suppose that $i+1, \ldots, i+m \in C$, but $i+m+1 \notin C$. Set $\avec = (0, a_2, \ldots, a_{n+1-k})$. The function
    \[
        \avec \to (0, a_2, \ldots, a_i, a_i + 1, \ldots, a_{i+m} + 1, a_{i+m+1}, \ldots, a_{n+1-k})
    \]
    maps $\Gamma_\avec$ to a unit interval graph on $[n+2-k]$ with $1$ twin. By iterating this function $k$ times, we obtain a new graph, $\Gamma_\bvec$, which is a unit interval graph on $[n+1]$ whose twin number is $k$.
    
    Reciprocally, every reduced unit interval graph on $[n+1-k]$ can be obtained from a unit interval graph $\bvec$ on $[n+1]$ with $k$ twins by deleting the twins. A twin $(i, i+1)$ is deleted as follows: let $C$ be the connected component of $\Gamma_\bvec$ that contains $i$ and suppose that $i+1, \ldots, i+m \in C$, but $i+m+1 \notin C$. Set $\bvec = (0, b_2, \ldots, b_{n+1})$. The function
    \[
        \bvec \to (0, a_2, \ldots, a_i, a_{i+2} - 1, \ldots, a_{i+m} - 1, a_{i+m+1}, \ldots, a_{n+1})
    \]
    maps $\Gamma_\bvec$ to a unit interval graph on $[n]$ with $k-1$ twins. By iterating this function $k$ times, we obtain a new graph, $\Gamma_\avec$, which is a reduced unit interval graph on $[n+1-k]$.

    Since the number of ways of choosing $k$ vertices among $n+1-k$, repetitions are allowed, equals the number of integer nonnegative solutions of
    \[
        x_1 + \ldots + x_{n-k+1} = k,
    \]
    which is $\binom{n}{n-k}$, it follows that $\tau_{n+1,k} = \binom{n}{n-k} \tau_{n+1-k,0} = \binom{n}{k} M_{n-k}$.
\end{proof}

\begin{table}[H]
\centering
\begin{tabular}{@{}r r @{}}
\toprule
$n$ & $\twin_n(t)$ \\
\midrule
0  & $1$ \\
1  & $1$ \\
2  & $1+t$ \\
3  & $2+2t+t^2$ \\
4  & $4+6t+3t^2+t^3$ \\
5  & $9+16t+12t^2+4t^3+t^4$ \\
6  & $21+45t+40t^2+20t^3+5t^4+t^5$ \\
7  & $51+126t+135t^2+80t^3+30t^4+6t^5+t^6$ \\
8  & $127+357t+441t^2+315t^3+140t^4+42t^5+7t^6+t^7$ \\
9  & $323+1016t+1428t^2+1176t^3+630t^4+224t^5+56t^6+8t^7+t^8$ \\
10 & $835+2907t+4572t^2+4284t^3+2646t^4+1134t^5+336t^6+72t^7+9t^8+t^9$\\
\bottomrule
\end{tabular}
\caption{The distribution of the twin number over all Dyck paths of semilength $n$. The coefficients of $\twin_n(t)$ are given by \cite[\oeis{A091869}]{OEIS}.}
\label{tab:twin-distribution}
\end{table}

The polynomials $\twin_n(t)$ are not real-rooted in general. For example, $\twin_3(t)$ has two complex conjugate zeros. However, their coefficients form sequences of real numbers which are log-concave and ultra log-convex simultaneously.

\begin{corollary}
    For each $n \geq 0$, $\{ \tau_{n+1,k} \}_{k=0}^{n}$ is a log-concave and an ultra log-convex sequence.
\end{corollary}

\begin{proof}
    The ultra log-convexity follows from the log-convexity of $\{ M_k \}_{k \geq 0}$ \cite[Proposition 3]{motzkin-aigner}. To prove the log-concavity, we will show that
    \[
        \frac{\tau_{n+1,k}^2}{\tau_{n+1,k-1} \tau_{n+1,k+1}} > 1, \quad k = 1, \ldots, n-1.
    \]
    Observe that, by \cite[Corollary 2]{motzkin-inequal},
    \begin{equation}
    \label{eq:motzkin-ineq}
        \frac{M_k^2}{M_{k-1} M_{k+1}} \geq \frac{k}{k+1}, \quad k \geq 1
    \end{equation}
    Then
    \[
        \frac{\tau_{n+1,k}^2}{\tau_{n+1,k-1} \tau_{n+1,k+1}} = \frac{(k+1)(n-k+1)}{k(n-k)} \frac{M_{n-k}^2}{M_{n-k-1} M_{n-k+1}} \geq \frac{(k+1)(n-k+1)}{k(n-k)} \frac{n-k}{n-k+1} = \frac{k+1}{k} > 1,
    \]
    from which the claim follows.
\end{proof}

\begin{remark}
    The sequence $\{ \tau_{n+1,k} \}_{k=0}^{n}$ is neither ultra log-concave nor log-convex for $n \geq 3$, see Table \ref{tab:twin-distribution}. 
\end{remark}

\subsection{The peak-nesting polynomial of reduced unit interval graphs}
Set $\Pnest_{\mathcal{R}(0)}(t) \coloneqq 1$ and $\Pnest_{\mathcal{R}(0)}^c(t) \coloneqq 0$, and for each $n \in \mathbb{N}$, let
\[
    \Pnest_{\mathcal{R}(n)}(t) \coloneqq \sum_{\avec \in \mathcal{R}(n)} t^{\pnest(\avec)} \quad \mbox{and} \quad \Pnest_{\mathcal{R}(n)}^c(t) \coloneqq \sum_{\substack{\avec \in \mathcal{R}(n) \\ \Gamma_\avec \mbox{ \scriptsize is connected}}} t^{\pnest(\avec)},
\]
see Table \ref{tab:pnest-reduced}.

A \defin{Riordan path} is a Motzkin path with no horizontal steps at height $0$. Recall that a Riordan sequence is a Motzkin sequence with no two consecutive ones. For each $n \geq 0$, let $\mathcal{RP}(n)$ be the set of all Riordan sequences of size $n$ (or, equivalently, the set of all Riordan paths of length $n-1$ when $n \geq 1$). Observe that, by definition, $\mathcal{RP}(n) \subseteq \mathcal{MP}(n)$ for every $n \geq 0$.

\begin{lemma}
For each $n \in \mathbb{N}$,
    \[
        \Pnest_{\mathcal{R}(n)}(t) = \sum_{M \in \mathcal{MP}(n)} t^{\height(M) + 1} \quad \mbox{and} \quad \Pnest_{\mathcal{R}(n)}^c(t) = \sum_{R \in \mathcal{RP}(n)} t^{\height(R) + 1} .
    \]
\end{lemma}

\begin{proof}
    It follows directly from Lemma \ref{lem:pnest=height} and from the proofs of (i) and (ii) of Proposition \ref{prop:reduced-sequences}.
\end{proof}

Let $f(t) = \sum_{i = 0}^n a_i t^i$ and $g(t) = \sum_{j = 0}^n b_j t^j$ be two polynomials. The \defin{maximum product} of $f(t)$ and $g(t)$ is defined as
\[
   \defin{ (f \star g)(t) } \coloneqq \sum_{i, j} a_i b_j t^{\max \{i, j\}}.
\]

\begin{proposition}
\label{prop:rec-pnest-red}
    The sequences of polynomials $\{\Pnest_{\mathcal{R}(n)}(t)\}_{n \geq 0}$ and $\{\Pnest_{\mathcal{R}(n)}^c(t)\}_{n \geq 0}$ satisfy the recursions
    \begin{equation}
    \label{eq:recur-red}
        \Pnest_{\mathcal{R}(n)}(t) =
        \begin{cases}
            1 &\mbox{if } n = 0, \\
            t &\mbox{if } n = 1, \\
            \Pnest_{\mathcal{R}(n-1)}(t) + \sum_{i=0}^{n-3} \left(t \cdot \Pnest_{\mathcal{R}(i+1)}(t) \right) \star \Pnest_{\mathcal{R}(n-2-i)}(t) &\mbox{if } n \geq 2,
        \end{cases}
    \end{equation}
    and
    \begin{equation}
    \label{eq:recur-cred}
        \Pnest_{\mathcal{R}(n)}^c(t) =
        \begin{cases}
            0 &\mbox{if } n = 0, \\
            t &\mbox{if } n = 1, \\
            \sum_{i = 0}^{n-3} \left(t \cdot \Pnest_{\mathcal{R}(i+1)}(t) \right) \star \Pnest_{\mathcal{R}(n-2-i)}^c(t) &\mbox{if } n \geq 2 .
        \end{cases}
    \end{equation}
\end{proposition}

\begin{proof}
    We begin by proving \eqref{eq:recur-red}. Fix $n \geq 1$. Observe that any Motzkin path $M$ of length $n-1$ can be uniquely decomposed either as $HM'$ or $UADB$, where $A$, $B$ and $M'$ are shorter Motzkin paths, and $D$ is the first down-step of $M$ that returns to height $0$.
    
    In the first case, deleting the initial horizontal step does not change the height, so the contribution is $\Pnest_{\mathcal{R}(n-1)}(t)$. In the second case, suppose $|A| = i$. Then $|B| = n - 3 - i$, and $\height(UAD) = \height(A) + 1$. Since the polynomial $\Pnest_{\mathcal{R}(i+1)}(t)$ records the height $\height(A) + 1$, the polynomial $t \cdot \Pnest_{\mathcal{R}(i+1)}(t)$ records the height $\height(UAD) + 1$. Also, $\height(UADB) = \max \{ \height(UAD), \height(B) \}$. Therefore, the contribution of all pairs of Motzkin paths $(A, B)$ is
    \[
        \left(t \cdot \Pnest_{\mathcal{R}(i+1)}(t) \right) \star \Pnest_{\mathcal{R}(n-2-i)}(t) .
    \]
    Summing over $i = 0, 1, \ldots, n-3$ proves the recursion.

    Now, we prove \eqref{eq:recur-cred}. Since $\mathcal{RP}(1) = \{ (0) \}$ and $\pnest( (0) ) = 1$, we have $\Pnest_{\mathcal{R}(1)}^c(t) = t$. Fix $n \geq 2$. Observe that any Riordan path $R$ of length $n-1$ can be uniquely decomposed as $UADB$, where $D$ is the first down-step of $R$ that returns to height $0$, $A$ is a Motzkin path, and $B$ is a Riordan path. Again, suppose that $|A| = i$. Then $|B| = n - 3 - i$, and $1 + \height(R) = 1 + \max \{ 1 + \height(A), \height(B) \} = \max \{ ( \height(A) + 1 ) + 1, \height(B) + 1 \}$, which contributes with
    \[
        \left(t \cdot \Pnest_{\mathcal{R}(i+1)}(t) \right) \star \Pnest_{\mathcal{R}(n-2-i)}^c(t) .
    \]
    Summing over $i = 0, 1, \ldots, n-3$ proves the recursion.
\end{proof}

\begin{table}[H]
\centering
\begin{tabular}{@{}r r r@{}}
\toprule
$n$ & $\Pnest_{\mathcal{R}(n)}(t)$ & $\Pnest_{\mathcal{R}(n)}^c(t)$ \\
\midrule
0  & $1$ & $0$ \\
1  & $t$ & $t$ \\
2  & $t$ & $0$ \\
3  & $t + t^2$ & $t^2$ \\
4  & $t + 3t^2$ & $t^2$ \\
5  & $t + 7t^2 + t^3$ & $2t^2 + t^3$ \\
6  & $t + 15t^2 + 5t^3$ & $3t^2 + 3t^3$ \\
7  & $t + 31t^2 + 18t^3 + t^4$ & $5t^2 + 9t^3 + t^4$ \\
8  & $t + 63t^2 + 56t^3 + 7t^4$ & $8t^2 + 23t^3 + 5t^4$ \\
9  & $t + 127t^2 + 161t^3 + 33t^4 + t^5$ & $13t^2 + 57t^3 + 20t^4 + t^5$ \\
10  & $t + 255t^2 + 441t^3 + 129t^4 + 9t^5$ & $21t^2 + 136t^3 + 68t^4 + 7t^5$ \\
\bottomrule
\end{tabular}
\caption{The distribution of peak-nestings over reduced unit interval graphs and connected reduced unit interval graphs on $[n]$. The coefficients of $\Pnest_{\mathcal{R}(n)}(t)$ are given by \cite[\oeis{A097862}]{OEIS}.}
\label{tab:pnest-reduced}
\end{table}

Although the polynomials $\{\Pnest_{\mathcal{R}(n)}(t)\}_{n \geq 0}$ and $\{\Pnest_{\mathcal{R}(n)}^c(t)\}_{n \geq 0}$ have nice recursive formulas, 
they are \emph{not} real-rooted in general. In fact,
\[
    \Pnest_{\mathcal{R}(12)}(t) = t + 1023 t^2 + 3036 t^3 + 1485 t^4 + 242 t^5 + 11 t^6
\]
has a pair of non-real zeros, and
\[
    \Pnest_{\mathcal{R}(13)}^c(t) = 89 t^2+1693 t^3+1836 t^4+540 t^5+54 t^6+t^7
\]
is also not real-rooted. They are also not ultra log-concave in general, with $\Pnest_{\mathcal{R}(89)}(t)$ and $\Pnest_{\mathcal{R}(95)}^c(t)$ being the first polynomials of their respective sequences where ultra log-concavity fails. However, log-concavity was verified by computer with Mathematica~\cite{Mathematica} for $0 \leq n \leq 200$, which motivates us to conjecture the following:

\begin{conjecture}
\label{conj:pnest-reduced}
    The coefficients of $\Pnest_{\mathcal{R}(n)}(t)$ and $\Pnest_{\mathcal{R}(n)}^c(t)$ form log-concave sequences for every $n \geq 0$.
\end{conjecture}

\subsection{The peak polynomial of reduced unit interval graphs}

Set $\Peak_{\mathcal{R}(0)}(t) \coloneqq 1$ and $\Peak_{\mathcal{R}(0)}^c(t) \coloneqq 0$, and let
\[
    \Peak_{\mathcal{R}(n)}(t) \coloneqq \sum_{\avec \in \mathcal{R}(n)} t^{\peak(\avec)} \quad \mbox{and} \quad \Peak_{\mathcal{R}(n)}^c(t) \coloneqq \sum_{\substack{\avec \in \mathcal{R}(n) \\ \Gamma_\avec \mbox{ \scriptsize is connected}}} t^{\peak(\avec)}
\]
for each $n \in \mathbb{N}$, see Table \ref{tab:peak-reduced}. In this section, we prove that $\Peak_{\mathcal{R}(n)}(t)$ and $\Peak_{\mathcal{R}(n)}^c(t)$ are real-rooted polynomials for every $n \geq 0$. To do that, we need the following lemma:

\begin{lemma}
\label{lem:peak-asc}
    For each $n \in \mathbb{N}$, let
    \[
        \asc_{\mathcal{MP}(n)}(t) = \sum_{\wvec \in \mathcal{MP}(n)} t^{\asc(\wvec)} \quad \mbox{and} \quad \asc_{\mathcal{RP}(n)}(t) = \sum_{\wvec \in \mathcal{RP}(n)} t^{\asc(\wvec)}
    \]
    be the ascent generating polynomials of Motzkin sequences and Riordan sequences, respectively. Then
    \[
        \Peak_{\mathcal{R}(n)}(t) = t^n \cdot \asc_{\mathcal{MP}(n)}(1/t) \quad \mbox{and} \quad \Peak_{\mathcal{R}(n)}^c(t) = t^n \cdot \asc_{\mathcal{RP}(n)}(1/t) .
    \]
\end{lemma}

\begin{proof}
    Let $\Gamma_\avec \in \mathcal{R}(n)$. We prove that
    \[
        \peak(\avec) + \asc(\wvec_\avec) = n.
    \]
    In fact, recall that, in any Motzkin sequence $\wvec_\avec \in \mathcal{MP}(n)$, $\asc(\wvec_\avec) + \des(\wvec_\avec) + \plateau(\wvec_\avec) = n-1$ and $\asc(\wvec_\avec) = \des(\wvec_\avec)$. Moreover, it is clear that $\asc(\wvec_\avec^*) = \asc(\wvec_\avec)$, $\des(\wvec_\avec^*) = \des(\wvec_\avec)$ and $\plateau(\wvec_\avec^*) = \plateau(\wvec_\avec)$. Finally, by the proof of Proposition \ref{prop:reduced-sequences}, $\decor(\wvec_\avec^*) = \plateau(\wvec_\avec^*)$. Hence,
    \[
        \peak(\avec) + \asc(\wvec_\avec) = (1 + \asc(\wvec_\avec^*) + \decor(\wvec_\avec^*) ) + \des(\wvec_\avec) = n,
    \]
    where the first equality follows from Theorem \ref{th:bij-dyck-decorated-motzkin}(ii).
\end{proof}

Hence, by Lemma~\ref{lem:peak-asc}, $\Peak_{\mathcal{R}(n)}(t)$ (resp. $\Peak_{\mathcal{R}(n)}^c(t)$) is real-rooted if and only if $\asc_{\mathcal{MP}(n)}(t)$ (resp. $\asc_{\mathcal{RP}(n)}(t)$) is real-rooted. We start by proving the real-rootedness of $\asc_{\mathcal{MP}(n)}(t)$:

\begin{theorem}
\label{th:peak-reduced}
    For each $n \in \setN$,
    \[
        \asc_{\mathcal{MP}(n)}(t) = \sum_{i = 0}^{\lfloor \nicefrac{(n-1)}{2} \rfloor} \binom{n-1}{2i} C_i t^i = \leftidx{_2}F_1 \left( \frac{1-n}{2}, \frac{2-n}{2}; 2; 4t \right) .
    \]
\end{theorem}

\begin{proof}
    By \eqref{eq:bij-motzkin-path-seq}, a Motzkin sequence $\wvec$ of size $n$ has $i$ ascents if and only if its corresponding Motzkin path $M_\wvec$ of length $n-1$ has $i$ $U$ steps, and the number of such paths is given by $\displaystyle \binom{n-1}{2i} C_i$ \cite[\oeis{A055151}]{OEIS}, which proves the formula for $\asc_{\mathcal{MP}(n)}(t)$.
\end{proof}

\begin{corollary}\cite[Theorem 4.2(i)]{china2020}
\label{cor:peak-reduced-rr}
    The polynomials $\{ \asc_{\mathcal{MP}(n)}(t) \}_{n \geq 0}$ satisfy the recursion
    \begin{equation}
    \label{eq:asc-rec}
        \asc_{\mathcal{MP}(n)}(t) =
        \begin{cases}
            1 \quad &\mbox{if } n = 0 \mbox{ or } n = 1, \\
            \displaystyle \frac{2n-1}{n+1} \asc_{\mathcal{MP}(n-1)}(t) + \frac{n-2}{n+1} \bigl(4t - 1\bigr) \asc_{\mathcal{MP}(n-2)}(t) \quad &\mbox{if } n \geq 2 .
        \end{cases}
    \end{equation}
    Moreover, $\asc_{\mathcal{MP}(n)}(t)$ is real-rooted for every $n \geq 0$ and $\left\{ \asc_{\mathcal{MP}(n)}(t) \right\}_{n \geq 0}$ is a generalized Sturm sequence.
\end{corollary}

Now, we will prove the real-rootedness of $\asc_{\mathcal{RP}(n)}(t)$. We start by stating the following lemma, which will be useful to find a closed formula for $\asc_{\mathcal{RP}(n)}(t)$:

\begin{lemma}[Lagrange Inversion Formula]\cite[Theorem 5.1.1]{Wilf1994}
\label{lem:lagrange}
    Let $f(y)$ and $\phi(y)$ be formal power series in $y$, with $\phi(0) = 1$. Then there is a unique formal power series $g = g(x)$ such that $g = x \cdot \phi(g)$. Further, the coefficients of $f(g(x))$ when expanded in a power series in $x$ about $x = 0$ are given by
    \[
        [x^n] f(g(x)) = \frac{1}{n} [y^{n-1}] ( f'(y) \cdot \phi(y)^n ) .
    \]
\end{lemma}

\begin{theorem}
\label{th:peak-connected-reduced}
    For each $n \geq 2$,
    \[
        \asc_{\mathcal{RP}(n)}(t) = \sum_{i = 1}^{\lfloor \nicefrac{(n-1)}{2} \rfloor} \frac{1}{i} \binom{n-1}{i-1} \binom{n-2-i}{i-1} t^i = t \cdot \leftidx{_3}F_2 \left( 1-n, \frac{3-n}{2}, \frac{4-n}{2}; 2, 3-n; 4t \right),
    \]
    where $\leftidx{_q}F_p (a_1, \ldots, a_p; b_1, \ldots, b_q; t)$ is the generalized hypergeometric function.
\end{theorem}

\begin{proof}
    Let $\mathcal{MP} \coloneqq \bigsqcup_{n \geq 1} \mathcal{MP}(n)$ be the set of all Motzkin paths. Given a Motzkin path $M \in \mathcal{MP}$, recall that $\mathrm{u}(M)$ denotes its number of $U$ steps and let $|M|$ denote its length. Set
    \[
        M(x,y) \coloneqq \sum_{M \in \mathcal{MP}} x^{|M|} y^{\mathrm{u}(M)} .
    \]
    Recall that any non-empty Motzkin path $M$ can be written either as $HM'$ or $UADB$, where $A$, $B$ and $M'$ are shorter Motzkin paths, and $D$ is the first step of $M$ whose height is $0$. Hence,
    \begin{align}
        M(x,y) &= 1 + \sum_{M' \in \mathcal{MP}} x^{|HM'|} y^{\mathrm{u}(HM')} + \sum_{A,B \in \mathcal{MP}} x^{|UADB|} y^{\mathrm{u}(UADB)} \nonumber \\
               &= 1 + x \sum_{M' \in \mathcal{MP}} x^{|M'|} y^{\mathrm{u}(M')} + y x^2 \sum_{A,B \in \mathcal{MP}} x^{|A| + |B|} y^{\mathrm{u}(A) + \mathrm{u}(B)} \label{eq:M-series}\\
               &= 1 + x \cdot M(x,y) + yx^2 \cdot [M(x,y)]^2 . \nonumber
    \end{align}

    Let $\mathcal{RP} \coloneqq \bigsqcup_{n \geq 1} \mathcal{RP}(n)$ be the set of all Riordan paths, and set
    \[
        R(x,y) \coloneqq \sum_{R \in \mathcal{RP}} x^{|R|} y^{\mathrm{u}(R)} .
    \]
    Observe that any non-empty Riordan path $R$ has a unique decomposition $R = (U M_1 D) \cdots (U M_r D)$, where $M_1, \ldots, M_r$ are Motzkin paths. Hence,
    \begin{align}
        R(x,y) &= 1 + \sum_{r \geq 1} \sum_{M_1, \ldots, M_r \in \mathcal{MP}} x^{|(U M_1 D) \cdots (U M_r D)|} y^{\mathrm{u}((U M_1 D) \cdots (U M_r D))}  \nonumber\\
               &= 1 + \sum_{r \geq 1} [ y x^2 \cdot M(x,y) ]^r \label{eq:J-series}\\
               & = \frac{1}{1 - y x^2 \cdot M(x,y)}. \nonumber
    \end{align}
    By \eqref{eq:M-series},
    \begin{equation}
    \label{eq:M-series2}
        1 - y x^2 \cdot M(x,y) = \frac{1 + x \cdot M(x,y)}{M(x,y)},
    \end{equation}
    and combining \eqref{eq:J-series} and \eqref{eq:M-series2} we get
    \begin{equation}
    \label{eq:1}
        R(x,y) = \frac{M(x,y)}{1 + x \cdot M(x,y)}.
    \end{equation}
    Let $Y(x,y) = x \cdot M(x,y)$. By \eqref{eq:M-series},
    \[
        Y(x,y) = x \cdot \left\{ 1 + Y(x,y) + y \cdot [Y(x,y)]^2 \right\}
    \]
    and, by \eqref{eq:1},
    \[
        R(x,y) = \frac{Y(x,y)}{x \cdot [1 + Y(x,y)]} .
    \]
    
    Write $Y = Y(x,y)$. Then $Y$ is determined implicitly by
    \[
        Y = x \cdot \phi(Y),
    \]
    where $\phi(Y) = 1 + Y + y Y^2$, and we have
    \[
        R(x,y) = \frac{1}{x} f(Y),
    \]
    where $f(Y) = Y \cdot (1+Y)^{-1}$.
    
    By definition, $[t^i] \asc_{\mathcal{RP}(n)}(t)$ equals the number of Riordan paths of length $n-1$ with $i$ $U$ steps, which is given by $[x^{n-1} y^i] R(x,y)$. Hence, by Lemma \ref{lem:lagrange},
    \begin{align*}
        [t^i] \asc_{\mathcal{RP}(n)}(t) = [x^{n-1} y^i] R(x,y) = [x^n y^i] f(Y) &= \frac{1}{n} [Y^{n-1} y^i] ( f'(Y) \cdot \phi(Y)^n ) \\
                                                                                &= \frac{1}{n} [Y^{n-1} y^i] \frac{\left( (1+Y) + y Y^2 \right)^n}{(1 + Y)^2} \\
                                                                                &= \frac{1}{n} [Y^{n-1} y^i] \frac{1}{(1 + Y)^2} \sum_{j=0}^n \binom{n}{j} (y Y^2)^j (1 + Y)^{n - j} \\
                                                                                &= \frac{1}{n} [Y^{n-1}] \binom{n}{i} Y^{2i}(1 + Y)^{n - 2 - i} \\
                                                                                &= \frac{1}{n} \binom{n}{i} \binom{n-2-i}{n-1-2i} \\
                                                                                &= \frac{1}{i} \binom{n-1}{i-1} \binom{n-2-i}{i-1},
    \end{align*}
    which concludes the proof.
\end{proof}

\begin{corollary}
\label{cor:peak-connected-reduced-rr}
    The polynomial $\asc_{\mathcal{RP}(n)}(t)$ is real-rooted for every $n \geq 0$. 
\end{corollary}

\begin{proof}
    The claim is clear for $n = 0$ and $n = 1$, see Table \ref{tab:peak-reduced}. For $n \geq 1$, write
    \[
        \asc_{\mathcal{RP}(n+1)}(t) = t \sum_{i = 0}^{\lfloor \nicefrac{(n-2)}{2} \rfloor} \frac{1}{i+1} \binom{n}{i} \binom{n-2-i}{i} t^i .
    \]
    The proof consists of two parts: first, we prove that the polynomials
    \[
        p_{n+1}(t) = \sum_{i=0}^{\lfloor \nicefrac{(n-2)}{2} \rfloor} \binom{n-2-i}{i} t^i
    \]
    are real-rooted. Then, we show that the sequence
    \[
        \sigma = \left\{ \frac{1}{i+1} \binom{n}{i} \right\}_{i=0}^n
    \]
    is a multiplier sequence, which implies that $\asc_{\mathcal{RP}(n+1)}(t)$ is real-rooted.

    In fact, by induction, one can verify that the polynomials $\{p_n(t)\}_{n \geq 2}$ satisfy the recursion
    \[
        p_n(t) =
        \begin{cases}
            0 &\mbox{if } n = 2, \\
            1 &\mbox{if } n = 3, \\
            p_{n-1}(t) + t \cdot p_{n-2}(t) &\mbox{if } n \geq 4.
        \end{cases}
    \]
    Moreover, since $p_2(t) \interl p_3(t)$ clearly, one can use induction and Lemma \ref{lem:interlacing-properties}(i) to prove that $p_{n-1}(t) \interl p_n(t)$ for all $n \geq 4$, which proves the real-rootedness of the polynomials $\{p_n(t)\}_{n \geq 2}$.

    Now, we prove that $\sigma$ is a multiplier sequence. In fact, by Lemma \ref{lem:schur-polya}, it is sufficient to show that the polynomial
    \[
        \Phi_\sigma (t) = \sum_{i=0}^n \frac{1}{i+1} \binom{n}{i} \frac{t^i}{i!} = \frac{1}{n+1} \sum_{i=0}^n \binom{n+1}{i+1} \frac{t^i}{i!}
    \]
    is real-rooted. In fact, $\displaystyle \Phi_\sigma (t) = \frac{L_n^{(1)}(-t)}{n+1}$, where
    \[
        L_n^{(\alpha)}(t) = \sum_{i=0}^n (-1)^i \binom{n + \alpha}{n-i} \frac{t^i}{i!}
    \]
    is the generalized Laguerre polynomial of degree $n$ \cite{sonine1880, booklaguerre}, which is known to be real-rooted for $\alpha > -1$, and concludes the proof.
\end{proof}

\begin{table}[H]
\centering
\begin{tabular}{@{}r r r@{}}
\toprule
$n$ & $\Peak_{\mathcal{R}(n)}(t)$ & $\Peak_{\mathcal{R}(n)}^c(t)$ \\
\midrule
0  & $1$ & $0$ \\
1  & $t$ & $t$ \\
2  & $t^2$ & $0$ \\
3  & $t^2 + t^3$ & $t^2$ \\
4  & $3t^3 + t^4$ & $t^3$ \\
5  & $2t^3 + 6t^4 + t^5$ & $2t^3 + t^4$ \\
6  & $10t^4 + 10t^5 + t^6$ & $5t^4 + t^5$ \\
7  & $5t^4 + 30t^5 + 15t^6 + t^7$ & $5t^4 + 9t^5 + t^6$ \\
8  & $35t^5 + 70t^6 + 21t^7 + t^8$ & $21t^5 + 14t^6 + t^7$ \\
9  & $14t^5 + 140t^6 + 140t^7 + 28t^8 + t^9$ & $14t^5 + 56t^6 + 20t^7 + t^8$ \\
10 & $126t^6 + 420t^7 + 252t^8 + 36t^9 + t^{10}$ & $84t^6 + 120t^7 + 27t^8 + t^9$ \\
\bottomrule
\end{tabular}
\caption{The distribution of peaks over reduced unit interval graphs and connected reduced unit interval graphs on $[n]$. The coefficients of $t^n \Peak_{\mathcal{R}(n)}(1/t)$ are given by \cite[\oeis{A055151}]{OEIS}, while the coefficients of $t^n \Peak_{\mathcal{R}^c(n)}(1/t)$ are given by \cite[\oeis{A132081}]{OEIS}.}
\label{tab:peak-reduced}
\end{table}

\section{Open questions and further directions}
\label{sec:open}

We finish this work with some open questions and further directions. In Section \ref{sec:abelian}, we studied Abelian unit interval graphs and we characterized them as the unit interval graphs $\Gamma_\avec \in \DP(n)$ for which there exists $k \in [n]$ such that $\{1, \ldots, k\}$ and $\{k+1, \ldots, n\}$ are cliques of $\Gamma_\avec$, see Lemma \ref{charac-abelian}. Hence, we can generalize this definition as follows: we say that $\Gamma_\avec \in \DP(n)$ is an \defin{$m$-Abelian unit interval graph} if there exist $1 \leq k_1\leq \cdots \leq k_m \leq n$ such that $\{1, \ldots, k_1 \}, \{k_1 + 1, \ldots, k_2 \}, \ldots, \{k_m + 1, \ldots, n\}$ are cliques of $\Gamma_\avec$. In particular, an Abelian unit interval graph is $1$-Abelian. Similarly, in Section \ref{sec:line-graphs}, we defined the class of $m$-nested unit interval graphs as those unit interval graphs whose peak-nesting is at most $m$.

The following questions naturally arise:

\begin{question}
    Are there nice closed formulas for the peak and the peak-nesting polynomials over all (or over the connected) $m$-Abelian unit interval graphs and $m$-nesting unit interval graphs? Are their sequences of coefficients related to other combinatorial objects? For which values of $m \in \mathbb{N}$ are some of these polynomials real-rooted, ultra log-concave, log-concave or unimodal? Are any of these symmetric? Do they have a nonnegative expansion in the gamma-basis? If so, are there nice interpretations for these coefficients?
\end{question}

We conclude this section by presenting a (unexpected) bijection between unit interval graphs and lattice path matroids \cite{BoninMierNoy2003, BoninMier2006, Bonin2010}. Observe that, if the set of peak-cliques of $\Gamma_\avec$ is
\[
    \mathcal{P}_\avec = \{P_1, \ldots, P_m\},
\]
then $[n] = \bigcup_{k=1}^m P_k$, where each $P_k = \{a_k, a_k +1, \ldots, b_k\}$ for some $1 \leq a_k \leq b_k \leq n$. Moreover, relabeling the $P_k$'s if necessary, it is clear that $a_1 < \cdots < a_k$ and $b_1 < \cdots < b_k$, see Example \ref{ex:def-peaks}. From \cite[Theorem 2.1]{multi-path}, it follows that $\mathcal{M}_\avec = (P_1, \ldots, P_m)$ is the presentation of some lattice path matroid on $[n]$ with $\rank(\mathcal{M}_\avec) = m$, and we call $\mathcal{M}_\avec$ (resp., $\Gamma_\avec$) the \defin{associated LPM} (resp., the \defin{associated unit interval graph}) of $\Gamma_\avec$ (resp., $\mathcal{M}_\avec$). This allows us to characterize the three subclasses of unit interval line graphs studied in this paper as subclasses of lattice path matroids:

\begin{proposition}
\label{prop:lpms}
The LPMs associated to
    \begin{enumerate}[(i)]
        \item Abelian unit interval graphs are those with presentation $\mathcal{M}_{\mathcal{A}} = (P_1, \ldots, P_m)$, where $P_1 \cup P_m = [n]$;
        \item $m$-nested unit interval graphs are those whose diagram does not contain $((m+1) \times m)$-grids;
        \item reduced unit interval line graphs on $[n]$ are those with presentation $\mathcal{M}_{\mathcal{R}} = ([a_1, b_1], \ldots, [a_m, b_m])$ such that, for every $i \in [n-1]$, either $i = b_j$ or $i+1 = a_j$ for some $j \in [m]$.
    \end{enumerate}
\end{proposition}

\begin{proof}
\leavevmode
    \begin{enumerate}[(i)]
        \item The proof follows immediately from Lemma \ref{charac-abelian};
        \item The claim follows from the fact that a lattice path matroid $M = (P_1, \ldots, P_k)$ that contains a $((m+1) \times m)$-grid in its diagram contains an element $i+m$, $i \geq 1$, where $i+m \in P_j, P_{j+1}, \ldots, P_{j+m}$ for some $j$, see Figure \ref{fig:line-lpms}. Hence, $\pnest(i+m) \geq m+1$ in its associated unit interval graph;
        \begin{figure}[H]
            \centering
                \begin{tikzpicture}[scale=2.0]
                    \draw[black,line width=1.2pt] (0,0)--(0,3);
                    \draw[black,line width=1.2pt] (2,0)--(2,3);
                    \draw[black,line width=1.2pt] (4,0)--(4,3);
    
                    \draw[black,line width=1.2pt] (0,0)--(4,0);
                    \draw[black,line width=1.2pt] (0,1)--(4,1);
                    \draw[black,line width=1.2pt] (0,2)--(4,2);
                    \draw[black,line width=1.2pt] (0,3)--(4,3);
           
                    \node[left] at (0,0.5) {\small $i$};
                    \node[left] at (2,0.5) {\small $i+1$};
                    \node[left] at (4,0.5) {\small $i+2$};
                    
                    \node[left] at (0,1.5) {\small $i+1$};
                    \node[left] at (2,1.5) {\small $i+2$};
                    \node[left] at (4,1.5) {\small $i+3$};
    
                    \node[left] at (0,2.5) {\small $i+2$};
                    \node[left] at (2,2.5) {\small $i+3$};
                    \node[left] at (4,2.5) {\small $i+4$};
                \end{tikzpicture}
            \caption{A $(3 \times 2)$-grid in an LPM with its $N$ steps labeled. Observe that $i+2$ is shown at least $3$ times in an LPM containing such a grid.}
            \label{fig:line-lpms}
        \end{figure}
    \item In fact, if $\Gamma_\avec \in \mathcal{R}(n)$ is a reduced unit interval graph with set of peak-cliques $\mathcal{P}_\avec = \{ P_1, \ldots, P_m \}$, then there exists no $i \in [n-1]$ such that $i$ and $i+1$ belong to the same peak-cliques. Hence, if $i \neq b_j$ and $i+1 \neq a_j$ for all $j \in [m]$, then $i \in [a_j, b_j]$ if and only if $i+1 \in [a_j, b_j]$, which is a contradiction. Hence, the claim follows.
    \end{enumerate}
\end{proof}   

The following natural questions related to LPMs and unit interval graphs arise from this bijection and Proposition \ref{prop:lpms}:

\begin{question}
    How can we characterize important subclasses of lattice path matroids, such as Schubert matroids, snakes and panhandle matroids, as unit interval graphs? What are their peak and peak-nesting polynomials? Can we find nice descriptions of structures of lattice path matroids in terms of unit interval graphs, such as flats, circuits and bases? How are operations in LPMs, such as restriction, contraction and duality translated to unit interval graphs? How does that affect their corresponding peak and peak-nesting polynomials? Can we say anything interesting about other matroidal polynomials, such as Tutte polynomials and Ehrhart polynomials?
\end{question}

\section*{Declaration on the use of generative AI}

During the development and preparation of this work, the authors used ChatGPT 5.6 Sol (OpenAI) \cite{openai2026chatgpt} for exploratory discussions, including suggestions for mathematical approaches, as well as for language and presentation improvements, and in the elaboration of the pictures. Any AI-generated mathematical suggestions were independently checked, modified when necessary, and incorporated only after verification by the authors. Generative AI was not treated as a source of mathematical authority, and the authors remain fully responsible for all results, proofs, and conclusions in the manuscript.

We list the main mathematical contributions of AI to this project:
\begin{itemize}
    \item the finding of the formulas for $\Pnest_n(t)$ and $\Pnest_n^c(t)$ in Corollary \ref{ref:pnest-formulas}, and the recursive formulas for $\Pnest_{\mathcal{R}(n)}(t)$ and $\Pnest_{\mathcal{R}(n)}^c(t)$ in Proposition \ref{prop:rec-pnest-red};
    \item the hypergeometric formulas presented in Section \ref{sec:peak-all}, Proposition \ref{prop:twin-formula} and Theorem \ref{th:peak-reduced};
    \item the statement and the proof of Proposition \ref{charac-unit-interval-line};
    \item the Lagrange Inversion Formula's argument in the proof of Theorem \ref{th:peak-connected-reduced}.
\end{itemize}

\bibliographystyle{alphaurl}
\bibliography{bibliography}

\end{document}